\documentclass{article}%
\usepackage{amsfonts}
\usepackage{amsmath}
\usepackage{amssymb}
\usepackage{graphicx}%
\numberwithin{equation}{section}
\newtheorem{theorem}{Theorem}[section]

\newtheorem{claim}[theorem]{Claim}

\newtheorem{corollary}[theorem]{Corollary}

\newtheorem{definition}[theorem]{Definition}

\newtheorem{lemma}[theorem]{Lemma}

\newtheorem{proposition}[theorem]{Proposition}
\newtheorem{remark}[theorem]{Remark}

\newenvironment{proof}[1][Proof]{\noindent\textbf{#1.} }{\ \rule{0.5em}{0.5em}}
\begin{document}

\title{$BMO$ estimates for nonvariational operators with discontinuous coefficients
structured on H\"{o}rmander's vector fields on Carnot
groups}
\author{Marco Bramanti, Maria Stella Fanciullo}
\maketitle

\begin{abstract}
We consider the class of operators%
\[
Lu=\sum_{i,j=1}^{q}a_{ij}(x)X_{i}X_{j}u
\]
where $X_{1},X_{2},...,X_{q}$ are homogeneous left invariant H\"{o}rmander's
vector fields on $\mathbb{R}^{N}$ with respect to a structure of Carnot group,
$q\leq N,$ the matrix $\left\{  a_{ij}\right\}  $ is symmetric and uniformly
positive on $\mathbb{R}^{q},$ the coefficients $a_{ij}\ $\ belong to
$L^{\infty}\cap VLMO_{loc}\left(  \Omega\right)  $ ("vanishing logarithmic
mean oscillation") with respect to the distance induced by the vector fields
(in particular they can be discontinuous), $\Omega$ is a bounded domain of
$\mathbb{R}^{N}$. We prove local estimates in $BMO_{loc}\cap L^{p}$ of the
kind:%
\begin{align*}
&  \left\Vert X_{i}X_{j}u\right\Vert _{BMO_{loc}^{p}\left(  \Omega^{\prime
}\right)  }+\left\Vert X_{i}u\right\Vert _{BMO_{loc}^{p}\left(  \Omega
^{\prime}\right)  }\leq\\
&  \le c\left\{  \left\Vert Lu\right\Vert _{BMO_{loc}^{p}\left(
\Omega\right)  }+\left\Vert u\right\Vert _{BMO_{loc}^{p}\left(  \Omega\right)
}\right\}
\end{align*}
for any $\Omega^{\prime}\Subset\Omega$, $1<p<\infty$.

Even in the uniformly elliptic case $X_{i}=\partial_{x_{i}}$, $q=N$ our
estimates improve the known results.

\end{abstract}
\tableofcontents

\section{Introduction}

\textbf{Context and main result of the paper}

Let $X_{0},X_{1},...,X_{q}$ be a system of smooth vector fields,%
\[
X_{i}=\sum_{j=1}^{N}b_{ij}\left(  u\right)  \partial_{u_{j}}%
\]
defined in the whole $\mathbb{R}^{N}$, and assume they satisfy H\"{o}rmander's
rank condition in $\mathbb{R}^{N}$: the vector fields $X_{i},$ and their
commutators $\left[  X_{i},X_{j}\right]  ,\left[  X_{k},\left[  X_{i}%
,X_{j}\right]  ,\right]  ,...$ up to some fixed length span $\mathbb{R}^{N}$
at any point. Then, a famous theorem proved by H\"{o}rmander in 1967 (see
\cite{h}), states that the second order differential operator%
\begin{equation}
L=\sum_{i=1}^{q}X_{i}^{2}+X_{0}\label{Horm op}%
\end{equation}
is hypoelliptic, that is $u\in C^{\infty}\left(  \Omega\right)  $ whenever $u$
is a distributional solution to $Lu=f$ in an open set $\Omega\subset
\mathbb{R}^{N}$ with $f\in C^{\infty}\left(  \Omega\right)  $. In 1975 Folland
\cite{fo} studied the class of H\"{o}rmander's operators (\ref{Horm op}) which
admits an underlying structure of \emph{homogeneous group}. This means that
the vector fields $X_{i}$ are left invariant with respect to a Lie group
operation in $\mathbb{R}^{N}$ (which we think as \textquotedblleft
translations\textquotedblright) and the operator $L$ is homogeneous of degree
$2$ with respect to a one-parameter family of Lie group automorphisms (which
we think as \textquotedblleft dilations\textquotedblright). Then, Folland
proved that there exists a global fundamental solution $\Gamma$ for $L,$ which
is translation invariant and homogeneous of degree $2-Q$ with respect to the
dilations, where $Q$ is the so called \emph{homogeneous dimension }of the
group. This fact allows to apply the theory of singular integrals in
homogeneous groups, and derive from representation formulas suitable a priori
estimates for the second order derivatives $X_{i}X_{j}u$ ($i,j=1,2,...,q$) or
$X_{0}u$ (note that the \textquotedblleft drift\textquotedblright\ vector
field $X_{0}$ has weight two in the operator $L$). Later, in \cite{rs},
Rothschild and Stein showed how the analysis of a general operator
(\ref{Horm op}), also in absence of an underlying structure of homogeneous
group, can be performed by a suitable technique of \textquotedblleft lifting
and approximation\textquotedblright\ which reduces the study of $L$ to that of
an operator of the kind studied by Folland.

In the last decade, more general families of second order differential
operators modeled on H\"{o}rmander's vector fields have been studied, namely
operators of the kinds%
\begin{align}
\mathcal{L}  &  =\sum_{i,j=1}^{q}a_{ij}\left(  x\right)  X_{i}X_{j}%
\label{nonvar 1}\\
\mathcal{L}  &  =\sum_{i,j=1}^{q}a_{ij}\left(  x\right)  X_{i}X_{j}%
-\partial_{t}\label{nonvar 2}\\
\mathcal{L}  &  =\sum_{i,j=1}^{q}a_{ij}\left(  x\right)  X_{i}X_{j}+X_{0}
\label{nonvar 3}%
\end{align}
where the matrix $\left\{  a_{ij}\left(  x\right)  \right\}  _{i,j=1}^{q}$ is
symmetric positive definite, the coefficients are bounded and satisfy suitable
mild regularity assumptions, for instance they belong to H\"{o}lder or $VMO$
spaces defined with respect to the distance induced by the vector fields.
Since the $a_{ij}$'s are not $C^{\infty},$ these operators are no longer
hypoelliptic. Nevertheless, a priori estimates on second order derivatives
with respect to the vector fields are a natural result which does not in
principle require smoothness of the coefficients. Namely, a priori estimates
in $L^{p}$ (with coefficients $a_{ij}$ in $VMO\cap L^{\infty}$) have been
proved in \cite{bb2} for operators (\ref{nonvar 1}), in \cite{bb1} for
operators (\ref{nonvar 3}) but in homogeneous groups, and in \cite{BZ2} for
operators (\ref{nonvar 3}) in the general case; a priori estimates in
$C^{\alpha}$ spaces (with coefficients $a_{ij}$ in $C^{\alpha}$) have been
proved in \cite{bb4} for operators (\ref{nonvar 2}), in \cite{GL} for
operators (\ref{nonvar 3}) but in homogeneous groups, and in \cite{BZ2} for
operators (\ref{nonvar 3}) in the general case. See also the recent monograph
\cite{BBLU} for more results on these classes of operators and for a larger
bibliographic account.

A somewhat endpoint case of $L^{p}$ estimates consists in $BMO$ type
estimates, which is the issue that we address in this paper. We will prove,
for operators (\ref{nonvar 1}) in homogeneous groups, with coefficients
$a_{ij}=a_{ji}$ satisfying the condition%
\[
\Lambda|\xi|^{2}\leq\sum_{i,j=1}^{q}a_{ij}\left(  x\right)  \xi_{i}\xi_{j}%
\leq\Lambda^{-1}|\xi|^{2}\text{ }\forall\xi\in{\mathbb{R}}^{q},\text{a.e.
}x\in\Omega
\]
and having \textquotedblleft vanishing logarithmic mean
oscillation\textquotedblright\ ($VLMO$) in a bounded domain $\Omega$, a priori
estimates of the kind%
\begin{align*}
&  \sum_{i,j=1}^{q}\left\Vert X_{i}X_{j}u\right\Vert _{BMO_{loc}^{p}\left(
\Omega^{\prime}\right)  }+\sum_{i=1}^{q}\left\Vert X_{j}u\right\Vert
_{BMO_{loc}^{p}\left(  \Omega^{\prime}\right)  }\\
&  \leq c\left\{  \left\Vert Lu\right\Vert _{BMO_{loc}^{p}\left(
\Omega\right)  }+\left\Vert u\right\Vert _{BMO_{loc}^{p}\left(  \Omega\right)
}\right\}
\end{align*}
for any $p\in\left(  1,\infty\right)  ,\Omega^{\prime}\Subset\Omega$ (see
Theorem \ref{Thm main} for the precise statement; also, the precise meaning of
these norms will be defined later). Let us stress that the $VLMO$ assumption
allows some kind of discontinuity of the coefficients $a_{ij}$.

\bigskip

\textbf{Comparison with the existing literature}

Remarkably, this estimate appears to be new even in the nonvariational
uniformly elliptic case ($q=n$, $X_{i}=\partial_{x_{i}}$ for $i=1,2,...,n$).
Actually, a few papers are devoted to $BMO$ estimates for the second
derivatives of the solutions to nonvariational uniformly elliptic equations:
we can quote the old paper by Peetre \cite{pee}, establishing local $BMO$
estimates for elliptic equations with uniformly continuous coefficients (with
a continuity modulus $o\left(  1/\left\vert \log t\right\vert \right)  $), and
the more recent ones by Chang-Dafni-Stein \cite{cds}, containing global $BMO$
estimates for the laplacian, and by Chang-Li \cite{cl}, dealing with elliptic
operators with Dini-continuous coefficients. Elliptic equations in divergence
form with $VLMO$ coefficients have been studied by Acquistapace \cite{A},
Huang \cite{H} (in which also some nondivergence form equation has studied),
while $BMO$ estimates for some nonlinear equations have been established by
Caffarelli-Huang \cite{CH}.

For operators modeled on H\"{o}rmander vector fields and written in divergence
form, $BMO$ estimates have been proved by Di Fazio-Fanciullo in \cite{DFF},
while Bramanti-Brandolini in \cite{BB3} have proved $BMO$ type estimates in a
scale of spaces $BMO_{\phi}$ for operators (\ref{nonvar 1}) built on general
H\"{o}rmander's vector fields, assuming a certain modulus of continuity of the
coefficients $a_{ij}$.

\bigskip

\textbf{Problems and strategy}

Although the continuity requirement on the coefficients asked in \cite{BB3} is
not a strong one, it represents a significant difference with the $L^{p}$
theory developed in \cite{bb1}, \cite{bb2} under the $VMO$ assumption and with
the $BMO$ theory developed in the present paper under the $VLMO$ assumption.

The reason of this difference has its roots in the real variable machinery
which is applied to prove these estimates, namely suitable extensions of the
famous $L^{p}$ estimate for the commutator of a Calder\'{o}n-Zygmund operator
with the multiplication by a $BMO$ function, proved by Coifman-Rochberg-Weiss
in \cite{crw}, which was first applied to the proof of $L^{p}$ a priori
estimates for uniformly elliptic operators with $VMO$ coefficients by
Chiarenza-Frasca-Longo in \cite{cfl1}, \cite{cfl2}. To put the real analysis
problem into its context, let us recall some facts. It is known that, under
fairly broad assumptions, a singular integral operator maps $L^{\infty}$ into
$BMO$. Under much more stringent assumptions it can be proved that it maps
$BMO$ into $BMO$. This was shown in \cite{pee} for convolution type operators
in $\mathbb{R}^{N}$, and in \cite{BB3}, in spaces of homogeneous type of
finite measure, for singular integrals satisfying a strong cancellation
property. The multiplication operator for a function $a$ maps $BMO$ into $BMO$
provided $a\in L^{\infty}\cap LMO$ (where $LMO$ stands for \textquotedblleft
logarithmic bounded mean oscillation\textquotedblright), as proved in
\cite{steg}. We also need a result stating that the commutator of a singular
integral operators with the multiplication for $a$ maps $BMO$ into $BMO$, with
operator norm bounded by the $LMO$ seminorm of $a$. Actually, we want the
operator norm of the commutator to be small whenever $a$ has small
\emph{oscillation}, but not small absolute size. A result of this kind has
been proved by Sun-Su in \cite{SS} for singular integral operators of
convolution type in $\mathbb{R}^{N}$, satisfying a strong cancellation
property. In this case, the commutator is proved to map $BMO\cap L^{p}$ into
itself ($1<p<\infty$) continuously, with operator norm bounded by the $LMO$
seminorm of $a$:%
\[
\left[  \left[  T,a\right]  f\right]  _{BMO}+\left\Vert \left[  T,a\right]
f\right\Vert _{L^{p}}\leq c\left[  a\right]  _{LMO}\left(  \left[  f\right]
_{BMO}+\left\Vert f\right\Vert _{L^{p}}\right)  .
\]
This result is clever under several regards. First, it exploits the idea of
bounding the $BMO\cap L^{p}$ norm of the commutator with the analogous norm of
$f$, and not separately the $BMO$ seminorm of the commutator with the $BMO$
seminorm of $f$; secondly, it relies on the very strong cancellation
properties enjoyed by classical Calder\'{o}n-Zygmund convolution-type kernels.
The present paper starts from the idea of extending this commutator theorem to
the context of convolution-type singular integrals on homogeneous groups,
which should be useful to handle operators (\ref{nonvar 1}) in this context,
in view of the results and techniques of \cite{fo}, \cite{bb1}, \cite{BB3}.
However, in contrast with the global, convolution nature of our singular
integrals, we are interested in the study of an operator (\ref{nonvar 1}) on a
bounded domain $\Omega$; this means that we don't want to assume the
coefficients $a_{ij}$ defined on the whole $\mathbb{R}^{N}$, nor rely on an
extension result for $LMO$ in this abstract context. Therefore, the commutator
estimate that we prove has to be established directly in a \emph{local }form.
This forces us to go through the whole argument in \cite{SS} and reshape it on
a new kind of local $BMO_{loc}\left(  \Omega_{1},\Omega_{2}\right)  $ spaces,
defined averaging the function over the balls centered at points of some open
set $\Omega_{1}$ and contained in a larger open set $\Omega_{2}\subseteq
\Omega$. This fact also serves another scope, namely avoiding to handle the
sets $B\cap\Omega$, which would require some extra assumption on $\Omega$ in
order to use the doubling condition. Moreover, it would seem to us rather
unnatural to express the assumption on the coefficients $a_{ij}$ in a form
which involves their boundary behavior, since after all we are just proving
interior estimates. Under this respect, the present paper moves in the spirit
of the recent research about \textquotedblleft local real analysis in locally
homogeneous spaces\textquotedblright\ carried out in \cite{BZ}.

Once the real analysis part of this research is set into its proper frame, one
can try to follow as close as possible the general line first drawn in
\cite{bb1}. In doing so, another major problem arises, namely the necessity of
getting some new uniform upper bound related to the fundamental solution
$\Gamma\left(  x_{0},u\right)  $ of the \textquotedblleft frozen
operator\textquotedblright\
\[
L_{0}=\sum_{i,j=1}^{q}a_{ij}\left(  x_{0}\right)  X_{i}X_{j}.
\]
Actually, to apply the real variable machinery to the concrete singular
integral operators which appear in our representation formulas, we have to
resort to the technique of expansion of $\Gamma\left(  x_{0},\cdot\right)  $
in spherical harmonics, first employed by Calder\'{o}n-Zygmund in \cite{cz}
and already used in all the aforementioned papers dealing with $L^{p}$
estimates for nonvariational operators structured on H\"{o}rmander's vector
fields. To get a control on the coefficients of this expansion, we need some
upper bound on the $u$-derivatives of any order of $\Gamma\left(
x_{0},u\right)  $, say for $\left\vert u\right\vert =1$, uniform with respect
to $x_{0}$. In \cite{bb1} the following estimate was proved:
\begin{equation}
\underset{x\in\Omega,\left\Vert u\right\Vert =1}{{\sup}}\left\vert \left(
\frac{\partial}{\partial u}\right)  ^{\beta}\Gamma\left(  x;u\right)
\right\vert \leq c\left(  \beta\right)  \text{,} \label{unif BB1}%
\end{equation}
for any multiindex $\beta$. In the present situation we also need a control on
the $LMO$ norm, and not just the $L^{\infty}$ norm, of the coefficients of the
expansion. To get this, the bound (\ref{unif BB1}) is not enough, and we need
to establish the following:
\begin{equation}
\sup_{x_{1},x_{2}\in\Omega,\left\vert u\right\vert =1}\left\vert \left(
\frac{\partial}{\partial u}\right)  ^{\beta}\Gamma\left(  x_{1},u\right)
-\left(  \frac{\partial}{\partial u}\right)  ^{\beta}\Gamma\left(
x_{2},u\right)  \right\vert \leq c_{\beta}\left\Vert A\left(  x_{1}\right)
-A\left(  x_{2}\right)  \right\Vert \label{unif BF}%
\end{equation}
where $\left\Vert A\left(  \cdot\right)  \right\Vert $ is the matrix norm of
the coefficients $\left\{  a_{ij}\right\}  $. To establish a bound on the
derivatives of any order of a fundamental solution, uniform with respect to
some parameter, is always a difficult task when, as happens for operators
structured on H\"{o}rmander's vector fields, we cannot rely on any kind of
explicit formula for the fundamental solution. We will get the bound
(\ref{unif BF}) in \S \ \ref{sec uniform bound}, adapting results and
techniques contained in series of papers by Bonfiglioli-Lanconelli-Uguzzoni
(see \cite{BU}, \cite{BLU1}, \cite{BLU2}) in the context of Gaussian bounds
for operators (\ref{nonvar 2}). We point out that the reason why we did not
consider in this paper operators with drift $X_{0}$ is only related to this
part of the proof. Namely, the papers \cite{BU}, \cite{BLU1}, \cite{BLU2} deal
with operators (\ref{nonvar 1}) or (\ref{nonvar 2}), but not (\ref{nonvar 3}),
therefore proving (\ref{unif BF}) in presence of a drift would require a much
deeper revision of the techniques used in those papers, and perhaps a
completely different approach.

\bigskip

\textbf{Plan of the paper}

Section 2 contains some known facts, the definition and basic properties of
\emph{local }$BMO$-type spaces and the statement of our assumptions and main result.

In section 3 we write the representation formulas that we need for $X_{i}%
X_{j}u$ in terms of $Lu$. These formulas involve singular integrals with
\textquotedblleft variable kernels\textquotedblright\ and their commutators.
By the classical technique of expansion in spherical harmonics we rewrite
these operators in series of singular integral operators of convolution type.
We state some uniform bound on the fundamental solution of the frozen operator
and show their use in proving suitable bounds on the coefficients of the
expansion in spherical harmonics. Section 4 contains the core of the real
analysis machinery: $BMO_{loc}$ estimates for singular integrals and their
commutators are established, first for convolution kernels and then in the
general case, together with a number of other useful results, in particular a
local version of the one stating that $LMO\cap L^{\infty}$ multiplies $BMO$.
Section 5 contains the proof of our main result, in three steps: first,
exploiting all the results of the previous sections, we prove local estimates
for functions with small compact support; second, exploiting several techiques
and results from \cite{BB3}, we prove local estimates for functions with small
noncompact support; third, we conclude the proof of the result on any bounded
domain. Finally, the Appendix contains the proof of the uniform bound on the
fundamental solution of the frozen operator. Although this bound is crucial in
the paper, we have preferred postponing its proof to the Appendix because the
techniques employed there are completely different from those of the previous sections.

\bigskip

\textbf{Acknoledgement. }We wish to thank Francesco Uguzzoni for useful talks
about the proof of Theorem \ref{Thm unif BF}.

\section{Preliminaries}

\subsection{Carnot groups, vector fields and their
metric\label{subsec homogeneous groups}}

Here we recall a number of known definitions and facts about homogeneous
groups and left invariant vector fields. For the justification of our
assertions, further details and examples, we refer to \cite[p. 618-622]%
{stein}, \cite[\S \ 1.3]{BLUbook}, \cite{fo}, \cite{bb1}.

We call \emph{Carnot group }or \emph{stratified homogeneous group} the space
$\mathbb{R}^{N}$ equipped with a Lie group structure, (\textquotedblleft
translations\textquotedblright) together with a family of \textquotedblleft
dilations\textquotedblright\ that are group automorphisms and are given by%
\begin{equation}
D\left(  \lambda\right)  :\left(  x_{1},\ldots,x_{N}\text{ }\right)
\mapsto\left(  \lambda^{\omega_{1}}x_{1},\ldots,\lambda^{\omega_{N}}%
x_{N}\right)  \label{dilations}%
\end{equation}
for $\lambda>0$, where%
\[
\left(  \omega_{1},...,\omega_{N}\right)  =\left(
1,1,...,1,2,2,...,2,...,s,s,...,s\right)
\]
for some positive integer $s$. We denote by $\circ$ the translation, and
assume that the origin is the group identity and the Euclidean opposite is the
group inverse. We will denote by $G$ the space $\mathbb{R}^{N}$ with this
structure of homogeneous group, and we will write $c(G)$ for a constant
depending on the numbers $N$, $\omega_{1}$,$\ldots,\omega_{N}$ and the group
law $\circ$.

We say that a differential operator $Y$ on $\mathbb{R}^{N}$ is
\emph{homogeneous of degree} $\beta>0$ if
\[
Y\,\left(  f\,\left(  D\left(  \lambda\right)  x\right)  \right)
=\lambda^{\beta}\left(  Yf\right)  \left(  D\left(  \lambda\right)  x\right)
\]
for every test function $f$, $\lambda>0$, $x\in\mathbb{R}^{N}$. Also, we say
that a function $f$ is \emph{homogeneous of degree} $\alpha\in\mathbb{R}$ if
\[
f\,\left(  D\left(  \lambda\right)  x\right)  =\lambda^{\alpha}\,f\,\left(
x\right)  \text{ \ }\forall\lambda>0\text{, }x\in\mathbb{R}^{N}.
\]
Clearly, if $Y$ is a differential operator homogeneous of degree $\beta$ and
$f$ is a homogeneous function of degree $\alpha$, then $Yf$ is homogeneous of
degree $\alpha-\beta$.

Let us consider now the Lie algebra $\ell$ associated to the group $G$, that
is, the Lie algebra of left-invariant vector fields, with the Lie bracket
given by the commutator of vector fields%
\[
\left[  X,Y\right]  =XY-YX.
\]
We can fix a basis $X_{1},\ldots,X_{N}$ in $\ell$ choosing $X_{i}$ as the left
invariant vector field which agrees with $\partial_{x_{i}}$ at the origin. It
turns out that $X_{i}$ is homogeneous of degree $\omega_{i}$. In particular,
let $X_{1},X_{2},...,X_{q}$ be the elements of the basis which are homogeneous
of degree $1$. The Lie algebra $\ell$ turns out to be nilpotent and
stratified:
\begin{align*}
\ell &  =%
{\textstyle\bigoplus\limits_{i=1}^{s}}
V_{i}\text{ \ \ with }\left[  V_{1},V_{j}\right]  =V_{j+1}\text{ for }j\leq
s-1,\left[  V_{1},V_{j}\right]  =\{0\}\text{ otherwise;}\\
V_{1}  &  =\text{span}\left(  X_{1},X_{2},...,X_{q}\right)  .
\end{align*}
The number $s$ is called the \emph{step }of the Lie algebra. We also say that
the vector fields $X_{1},X_{2},...,X_{q}$ satisfy H\"{o}rmander's condition at
step $s$.

\bigskip

Let us introduce the \emph{control distance }$d$ induced by the vector fields.
This is a concept which can be defined for any system of H\"{o}rmander's
vector fields (even in absence of a group structure) but which possesses more
properties in Carnot groups.

\begin{definition}
\label{c-c-distance}For any $\delta>0$, $x,y\in\mathbb{R}^{N}$, let
$C_{\delta}\left(  x,y\right)  $ be the class of absolutely continuous
mappings $\varphi:[0,1]\rightarrow\Omega$ which satisfy%
\[
\varphi^{\prime}(t)=\sum_{i=1}^{q}a_{i}(t)\left(  X_{i}\right)  _{\varphi
\left(  t\right)  }\text{ a.e. }t\in\left(  0,1\right)
\]
with $\left\vert a_{i}(t)\right\vert \leq\delta$ for $i=1,...,q$,
$\varphi\left(  0\right)  =x,\varphi\left(  1\right)  =y$. We define%
\[
d(x,y)=\inf\left\{  \delta>0:C_{\delta}\left(  x,y\right)  \neq\varnothing
\right\}  .
\]

\end{definition}

Note that the finiteness of $d\left(  x,y\right)  $ for any two points
$x,y\in\Omega$ is not a trivial fact, but depends on a connectivity result
known as \textquotedblleft Chow's theorem\textquotedblright; it can be proved
that $d$ is actually a distance in $\mathbb{R}^{N}$, finite for any couple of
points. Moreover:

\begin{proposition}
\label{invariant} \label{Prop homogeneous distance}The distance $d$ is
translation invariant and $1$-homogeneous for dilations on the group.
\end{proposition}

This fact is probably known, but we are unable to give a precise reference, so
we will prove it.

\begin{proof}
Let $B\left(  x,r\right)  $ be the $d$-ball of center $x$ and radius $r$. Let
us prove that
\begin{align*}
y  &  \in B\left(  x,r\right)  \Longrightarrow x^{-1}\circ y\in B\left(
0,r\right) \\
y  &  \in B\left(  0,r\right)  \Longrightarrow D\left(  \lambda\right)  y\in
B\left(  0,\lambda r\right)  .
\end{align*}

Let $y\in B\left(  x,r\right)  ,$ then for every $\delta<r$ there exists
$\varphi\in C_{\delta}\left(  x,y\right)  $. Let%
\[
\varphi_{x}\left(  t\right)  =x^{-1}\circ\varphi\left(  t\right)  ,
\]
and let us show that $\varphi_{x}\in C_{\delta}\left(  0,x^{-1}\circ y\right)
$. Clearly $\varphi_{x}\left(  0\right)  =0$ and $\varphi_{x}\left(  1\right)
=x^{-1}\circ y$. Moreover, let $J_{x}$ denote the Jacobian matrix of the
function $y\mapsto x^{-1}\circ y,$ then%
\begin{align*}
\varphi_{x}^{\prime}(t)  &  =J_{x}\left(  \varphi\left(  t\right)  \right)
\cdot\varphi^{\prime}\left(  t\right)  =J_{x}\left(  \varphi\left(  t\right)
\right)  \cdot\overset{q}{\underset{i=0}{\sum}}a_{i}(t)\left(  X_{i}\right)
_{\varphi\left(  t\right)  }\\
&  =\overset{q}{\underset{i=0}{\sum}}a_{i}(t)J_{x}\left(  \varphi\left(
t\right)  \right)  \cdot\left(  X_{i}\right)  _{\varphi\left(  t\right)
}=\overset{q}{\underset{i=0}{\sum}}a_{i}(t)\left(  X_{i}\right)  _{x^{-1}%
\circ\varphi\left(  t\right)  }=\overset{q}{\underset{i=0}{\sum}}%
a_{i}(t)\left(  X_{i}\right)  _{\varphi_{x}\left(  t\right)  }%
\end{align*}
where the up to last identity follows by the translation invariance of the
$X_{i}$'s (see \cite[Prop. 1.2.3 p.14]{BLUbook}). So the first assertion is proved.

Let now $y\in B\left(  0,r\right)  ,$ then for every $\delta<r$ there exists
$\varphi\in C_{\delta}\left(  0,y\right)  $. Let%
\[
\varphi_{\lambda}\left(  t\right)  =D\left(  \lambda\right)  \varphi\left(
t\right)  \text{.}%
\]
Then $\varphi_{\lambda}\left(  0\right)  =0;\varphi_{\lambda}\left(  1\right)
=D\left(  \lambda\right)  y$. Recall that%
\[
X_{i}=\sum_{j=1}^{N}b_{ij}\left(  u\right)  \partial_{u_{j}}%
\]
with $b_{ij}$ homogeneous of degree $\omega_{j}-1$, for $i=1,2,...,q$. Hence%
\[
b_{ij}\left(  \varphi_{\lambda}\left(  t\right)  \right)  =\lambda^{\alpha
_{j}-1}b_{ij}\left(  \varphi\left(  t\right)  \right)
\]
and%
\begin{align*}
\varphi_{\lambda}^{\prime}(t)_{j}  &  =\lambda^{\omega_{j}}\varphi_{j}%
^{\prime}\left(  t\right)  =\lambda^{\omega_{j}}\sum_{i=1}^{q}a_{i}%
(t)b_{ij}\left(  \varphi\left(  t\right)  \right)  =\\
&  =\overset{q}{\underset{i=0}{\sum}}\lambda a_{i}(t)b_{ij}\left(
\varphi_{\lambda}\left(  t\right)  \right)
\end{align*}
with%
\[
\left\vert \lambda a_{i}(t)\right\vert \leq\lambda\delta,
\]
which shows that $\varphi_{\lambda}\in C_{\lambda\delta}\left(  x,y\right)  $.
Since this holds for any $\delta<r,$ we conclude $y\in B\left(  0,\lambda
r\right)  $.
\end{proof}

\bigskip

We can also define in $\mathbb{R}^{N}$ a \emph{homogeneous norm} $\left\Vert
\cdot\right\Vert $ as follows. For any $x\in\mathbb{R}^{N}$, $x\neq0$, set
\[
\left\Vert x\right\Vert =\rho\Leftrightarrow\left\vert D\left(  \,\frac
{1}{\rho}\right)  x\right\vert =1\text{,}%
\]
where $\left\vert \cdot\right\vert $ denotes the Euclidean norm; also, let
$\left\Vert 0\right\Vert =0$. Then:

$\left\Vert D(\lambda)x\right\Vert =\lambda\left\Vert x\right\Vert $ for every
$x\in\mathbb{R}^{N}$, $\lambda>0$;

the set $\{x\in\mathbb{R}^{N}:\left\Vert x\right\Vert =1\}$ coincides with the
euclidean unit sphere $\sum_{N}$;

the function $x\mapsto\left\Vert x\right\Vert $ is smooth outside the origin;

there exists $c(G)\geq1$ such that for every $x$, $y\in\mathbb{R}^{N}$%
\begin{align*}
\left\Vert x\circ y\right\Vert  &  \leq c(\left\Vert x\right\Vert +\left\Vert
y\right\Vert )\text{ and }\left\Vert x^{-1}\right\Vert =\left\Vert
x\right\Vert \text{;}\\
\frac{1}{c}\left\vert y\right\vert  &  \leq\left\Vert y\right\Vert \leq
c\left\vert y\right\vert ^{1/s}\text{ if }\left\Vert y\right\Vert \leq1,
\end{align*}
hence the function%
\begin{equation}
\rho(x,y)=\left\Vert y^{-1}\circ x\right\Vert \label{rho}%
\end{equation}
is a \emph{quasidistance}, that is:%
\begin{align*}
\rho\left(  x,y\right)   &  \geq0\text{ and }\rho\left(  x,y\right)  =0\text{
if and only if }x=y;\\
\rho\left(  y,x\right)   &  =\rho\left(  x,y\right) \\
\rho\left(  x,y\right)   &  \leq c\,\left(  \rho\left(  x,z\right)
+\rho\left(  z,y\right)  \right)
\end{align*}
for every $x$, $y$, $z\in\mathbb{R}^{N}$ and some positive constant
$c(G)\geq1$. If we define the $\rho$-balls as
\[
B_{\rho}\left(  x,r\right)  =\left\{  y\in\mathbb{R}^{N}:\rho\left(
x,y\right)  <r\right\}  ,
\]
then $B_{\rho}\left(  0,r\right)  =D\left(  r\right)  B_{\rho}\left(
0,1\right)  $. The Lebesgue measure in $\mathbb{R}^{N}$ is the Haar measure of
$G$, hence%
\[
\left\vert B_{\rho}(x,r)\right\vert =\left\vert B_{\rho}(0,1)\right\vert
r^{Q}\text{ }\forall x\in\mathbb{R}^{N},r>0,
\]
where
\[
Q=\omega_{1}+\ldots+\omega_{N}%
\]
(with $\omega_{i}$ as in (\ref{dilations})) is the \emph{homogeneous dimension
}of $\mathbb{R}^{N}$.

A consequence of the Prop. \ref{invariant} is that the function $d\left(
x,0\right)  $ is another \emph{homogeneous norm }in the sense of Folland (see
\cite{fo}), equivalent to the function $\left\Vert x\right\Vert $ defined
above. In particular, $d$ and $\rho$ are globally equivalent.

\subsection{Function spaces defined by local mean
oscillations\label{subsec function spaces}}

Throughout the following, we will assume that $X_{1},X_{2},...,X_{q}$ is a
system of left invariant homogeneous H\"{o}rmander's vector fields on a Carnot
group in $\mathbb{R}^{N}$, as described in
\S \ \ref{subsec homogeneous groups}, and $d$ is the control distance induced
by the vector fields in $\mathbb{R}^{N}$.

Let $\Omega$ be a bounded open connected subset of $\mathbb{R}^{N}$. In the
following definitions, the balls are always taken with respect to the distance
$d$.

\begin{definition}
[$BMO^{p}$ spaces]For $p\in\lbrack1,\infty)$ we say that $f\in BMO^{p}\left(
\Omega\right)  $ if%
\[
\left\Vert f\right\Vert _{BMO^{p}\left(  \Omega\right)  }=\left[  f\right]
_{BMO\left(  \Omega\right)  }+\left\Vert f\right\Vert _{L^{p}\left(
\Omega\right)  }<\infty
\]
with%
\[
\left[  f\right]  _{BMO\left(  \Omega\right)  }=\sup_{x\in\Omega,r>0}\frac
{1}{\left\vert B\left(  x,r\right)  \cap\Omega\right\vert }\int_{B\left(
x,r\right)  \cap\Omega}\left\vert f\left(  y\right)  -f_{B\left(  x,r\right)
\cap\Omega}\right\vert dy
\]
where $f_{A}=\frac{1}{\left\vert A\right\vert }\int f$.
\end{definition}

\begin{definition}
[Local $BMO$ spaces]Let $f\in L_{loc}^{1}\left(  \Omega\right)  $. We say that
$f\in BMO_{loc}\left(  \Omega\right)  $ if
\[
\left[  f\right]  _{BMO_{loc}\left(  \Omega\right)  }=\sup_{B\left(
x,r\right)  \subset\Omega}\frac{1}{|B\left(  x,r\right)  |}{\int}_{B\left(
x,r\right)  }|f(y)-f_{B\left(  x,r\right)  }|dy<+\infty\text{.}%
\]
Note that the requirement $B\left(  x,r\right)  \subset\Omega$ is meaningful
because the distance $d$ is defined in the whole $\mathbb{R}^{N}$.

Let $f\in L_{loc}^{1}\left(  \Omega\right)  $ and $\Omega_{1}\Subset\Omega
_{2}\subseteq\Omega$. We say that $f\in BMO_{loc}\left(  \Omega_{1},\Omega
_{2}\right)  $ if
\[
\left[  f\right]  _{BMO_{loc}\left(  \Omega_{1},\Omega_{2}\right)  }%
=\sup_{x\in\Omega_{1},B\left(  x,r\right)  \subset\Omega_{2}}\frac
{1}{|B\left(  x,r\right)  |}{\int}_{B\left(  x,r\right)  }|f(y)-f_{B\left(
x,r\right)  }|dy<+\infty\,\text{.}%
\]

\end{definition}

Note that $f\in BMO_{loc}\left(  \Omega\right)  $ if and only if $f\in
BMO_{loc}\left(  \Omega_{1},\Omega\right)  $ for any $\Omega_{1}\Subset\Omega$.

Moreover, for any $\Omega_{1}\Subset\Omega_{2}\subseteq\Omega$ we have the
following inclusions%
\[
BMO_{loc}\left(  \Omega\right)  \subset BMO_{loc}\left(  \Omega_{1},\Omega
_{2}\right)
\]
with%
\[
\left[  f\right]  _{BMO_{loc}\left(  \Omega_{1},\Omega_{2}\right)  }%
\leq\left[  f\right]  _{BMO_{loc}\left(  \Omega\right)  },
\]
and
\[
BMO_{loc}\left(  \Omega_{2}\right)  \subset BMO_{loc}\left(  \Omega_{1}%
,\Omega_{2}\right)  \subset BMO_{loc}\left(  \Omega_{1}\right)
\]
with%
\[
\left[  f\right]  _{BMO_{loc}\left(  \Omega_{1}\right)  }\leq\left[  f\right]
_{BMO_{loc}\left(  \Omega_{1},\Omega_{2}\right)  }\leq\left[  f\right]
_{BMO_{loc}\left(  \Omega_{2}\right)  }\text{.}%
\]

\begin{definition}
[Local $LMO$ spaces]Let $\Omega_{1}\Subset\Omega_{2}\subseteq\Omega$. We say
that $f\in LMO_{loc}\left(  \Omega_{1},\Omega_{2}\right)  $ if
\begin{align*}
&  \left[  f\right]  _{LMO_{loc}\left(  \Omega_{1},\Omega_{2}\right)  }=\\
&  =\sup_{x\in\Omega_{1},B\left(  x,r\right)  \subset\Omega_{2}}\frac
{1+\log\frac{diam\Omega_{2}}{r}}{|B\left(  x,r\right)  |}{\int}_{B\left(
x,r\right)  }|f(y)-f_{B\left(  x,r\right)  }|dy<+\infty.
\end{align*}

Analogously we define the space $LMO_{loc}\left(  \Omega\right)  $ taking the
sup over all the balls $B\left(  x,r\right)  \subset\Omega$.

We say that $f\in VLMO_{loc}\left(  \Omega\right)  $ if $f\in LMO_{loc}\left(
\Omega\right)  $ and%
\[
\eta_{f}\left(  r\right)  =\sup_{B\left(  x,\rho\right)  \subset\Omega
,\rho\leq r}\frac{1+\log\frac{diam\Omega}{\rho}}{|B\left(  x,\rho\right)
|}{\int}_{B\left(  x,\rho\right)  }|f(y)-f_{B\left(  x,\rho\right)
}|dy\rightarrow0\text{ for }r\rightarrow0\text{.}%
\]

\end{definition}

Let $\overline{x}\in\Omega,0<r_{1}<r_{2}$ such that $B\left(  \overline
{x},r_{2}\right)  \subset\Omega$. Then%
\[
VLMO_{loc}\left(  \Omega\right)  \subset LMO_{loc}\left(  B\left(
\overline{x},r_{1}\right)  ,B\left(  \overline{x},r_{2}\right)  \right)
\]
with%
\begin{equation}
\left[  f\right]  _{LMO_{loc}\left(  B\left(  \overline{x},r_{1}\right)
,B\left(  \overline{x},r_{2}\right)  \right)  }\leq\eta_{f}\left(
2r_{2}\right)  . \label{eta f}%
\end{equation}

\begin{definition}
[Local $BMO^{p}$ spaces]Let $p\in\left(  1,\infty\right)  $ and $\Omega
_{1}\Subset\Omega_{2}\subseteq\Omega$. We say that $f\in BMO_{loc}^{p}\left(
\Omega_{1},\Omega_{2}\right)  $ if%
\[
\left\Vert f\right\Vert _{BMO_{loc}^{p}\left(  \Omega_{1},\Omega_{2}\right)
}\equiv\left\Vert f\right\Vert _{L^{p}\left(  \Omega_{2}\right)  }+\left[
f\right]  _{BMO_{loc}\left(  \Omega_{1},\Omega_{2}\right)  }<\infty.
\]
Also, we say that $f\in BMO_{loc}^{p}\left(  \Omega_{1}\right)  $ if%
\[
\left\Vert f\right\Vert _{BMO_{loc}^{p}\left(  \Omega_{1}\right)  }%
\equiv\left\Vert f\right\Vert _{L^{p}\left(  \Omega_{1}\right)  }+\left[
f\right]  _{BMO_{loc}\left(  \Omega_{1}\right)  }<\infty.
\]

\end{definition}

We have
\[
BMO_{loc}^{p}\left(  \Omega_{2}\right)  \subset BMO_{loc}^{p}\left(
\Omega_{1},\Omega_{2}\right)  \subset BMO_{loc}^{p}\left(  \Omega_{1}\right)
.
\]

Let us note that the spaces $BMO_{loc}\left(  \Omega_{1},\Omega_{2}\right)  $,
$BMO_{loc}^{p}\left(  \Omega_{1},\Omega_{2}\right)  $ and so on, are
increasing with respect to both $\Omega_{1}$ and $\Omega_{2}$. The following
fact, which will be useful several times, says that for compactly supported
functions also an inclusion in the reverse order holds, in some sense.

We will often use $BMO_{loc}^{p}\left(  \Omega_{1},\Omega_{2}\right)  $ spaces
with $\Omega_{1},\Omega_{2}$ two concentric balls; just to shorten notations,
we will set%
\[
BMO_{loc}\left(  B\left(  \overline{x},R_{1};R_{2}\right)  \right)  \equiv
BMO_{loc}\left(  B\left(  \overline{x},R_{1}\right)  ,B\left(  \overline
{x},R_{2}\right)  \right)  .
\]

\begin{lemma}
\label{Lem KR}Let $f$ $\in BMO_{loc}^{p}\left(  B\left(  \overline{x},R;KR
\right)  \right)  $ for some $K>3,$ with sprt$f\subset B\left(  \overline
{x},R\right)  $. Then%
\[
\left[  f\right]  _{BMO_{loc}\left(  B\left(  \overline{x},R; KR \right)
\right)  }\leq2\left(  \left[  f\right]  _{BMO_{loc}\left(  B\left(
\overline{x},R;3R \right)  \right)  }+\frac{1}{\left\vert B\left(
\overline{x},R\right)  \right\vert ^{1/p}}\left\Vert f\right\Vert
_{L^{p}\left(  B\left(  \overline{x},R\right)  \right)  }\right)  .
\]

\end{lemma}

\begin{proof}
Let us consider a ball $B\left(  x_{0},r\right)  $ with $x_{0}\in B\left(
\overline{x},R\right)  $ and $B\left(  x_{0},r\right)  \subset B\left(
\overline{x},KR\right)  .$

If $B\left(  x_{0},r\right)  \subset B\left(  \overline{x},3R\right)  ,$ then
obviously%
\[
\frac{1}{|B\left(  x_{0},r\right)  |}{\int}_{B\left(  x_{0},r\right)
}|f(y)-f_{B\left(  x_{0},r\right)  }|dy\leq\left[  f\right]  _{BMO_{loc}%
\left(  B\left(  \overline{x},R;3R\right)  \right)  },
\]
so let us assume $B\left(  x_{0},r\right)  \varsubsetneq B\left(  \overline
{x},3R\right)  $. This means that $B\left(  x_{0},r\right)  \supset B\left(
\overline{x},R\right)  $, hence for any $c\in\mathbb{R}$%
\begin{align*}
&  \frac{1}{|B\left(  x_{0},r\right)  |}{\int}_{B\left(  x_{0},r\right)
}|f(y)-f_{B\left(  x_{0},r\right)  }|dy\leq\frac{2}{|B\left(  x_{0},r\right)
|}{\int}_{B\left(  x_{0},r\right)  }|f(y)-c|dy\\
&  =\frac{2}{|B\left(  x_{0},r\right)  |}\left(  {\int}_{B\left(  \overline
{x},R\right)  }|f(y)-c|dy+\left\vert c\right\vert \left\vert B\left(
x_{0},r\right)  \setminus B\left(  \overline{x},R\right)  \right\vert \right)
\\
&  \leq\frac{2}{|B\left(  \overline{x},R\right)  |}{\int}_{B\left(
\overline{x},R\right)  }|f(y)-c|dy+2\left\vert c\right\vert \frac{\left\vert
B\left(  x_{0},r\right)  \setminus B\left(  \overline{x},R\right)  \right\vert
}{|B\left(  x_{0},r\right)  |}%
\end{align*}
choosing $c=f_{B\left(  \overline{x},R\right)  }$%
\begin{align*}
&  \leq\frac{2}{|B\left(  \overline{x},R\right)  |}{\int}_{B\left(
\overline{x},R\right)  }|f(y)-f_{B\left(  \overline{x},R\right)
}|dy+2\left\vert f_{B\left(  \overline{x},R\right)  }\right\vert \\
&  \leq2\left(  \left[  f\right]  _{BMO_{loc}\left(  B\left(  \overline
{x},R;3R\right)  \right)  }+\frac{1}{\left\vert B\left(  \overline
{x},R\right)  \right\vert ^{1/p}}\left\Vert f\right\Vert _{L^{p}\left(
B\left(  \overline{x},R\right)  \right)  }\right)  .
\end{align*}

\end{proof}

In the sequel we will need the following proposition that gives us a
comparison between local and global $BMO^{p}$ spaces.

\begin{proposition}
\label{Prop norme locali globali}Let $s>0,p\in\lbrack1,\infty)$ and
$\overline{x}\in\mathbb{R}^{N}$ be fixed.

\noindent(a) If $f\in BMO_{loc}^{p}\left(  B\left(  \overline{x},s;3s\right)
\right)  $, then $f\in BMO^{p}\left(  B\left(  \overline{x},s\right)  \right)
$ with
\[
\left\Vert f\right\Vert _{BMO^{p}\left(  B\left(  \overline{x},s\right)
\right)  }\leq c\left\Vert f\right\Vert _{BMO_{loc}^{p}\left(  B\left(
\overline{x},s;3s\right)  \right)  }%
\]
for some constant $c=c\left(  G\right)  $.

\noindent(b) If $f\in BMO^{p}\left(  B\left(  \overline{x},R\right)  \right)
$ for some $R>s$ and sprt$f\subset B\left(  \overline{x},s\right)  $, then the
function $\widetilde{f}$ obtained extending $f$ to zero outside $B\left(
\overline{x},R\right)  $ belongs to $BMO_{loc}^{p}\left(  B\left(
\overline{x},s;3s\right)  \right)  $, with%
\[
\left\Vert \widetilde{f}\right\Vert _{BMO_{loc}^{p}\left(  B\left(
\overline{x},s;3s\right)  \right)  }\leq\left\Vert f\right\Vert _{BMO^{p}%
\left(  B\left(  \overline{x},R\right)  \right)  }+\frac{c}{\left(
R-s\right)  ^{Q/p}}\left\Vert f\right\Vert _{L^{p}\left(  B\left(
\overline{x},s\right)  \right)  }%
\]
with $c=c\left(  G,p\right)  $.
\end{proposition}

\begin{proof}
(a) Pick $x\in B\left(  \overline{x},s\right)  $ and $r>0$ such that $B\left(
x,r\right)  \subset B\left(  \overline{x},3s\right)  $. Then, for a constant
$c$ to be chosen later,%
\begin{align*}
&  \frac{1}{\left\vert B\left(  x,r\right)  \cap B\left(  \overline
{x},s\right)  \right\vert }\int_{B\left(  x,r\right)  \cap B\left(
\overline{x},s\right)  }\left\vert f\left(  y\right)  -c\right\vert dy\\
&  \leq\frac{1}{c_{1}\left\vert B\left(  x,r\right)  \right\vert }%
\int_{B\left(  x,r\right)  }\left\vert f\left(  y\right)  -c\right\vert dy
\end{align*}
because
\[
\left\vert B\left(  x,r\right)  \cap B\left(  \overline{x},s\right)
\right\vert \geq c_{1}\left\vert B\left(  x,r\right)  \right\vert
\]
by the regularity of metric balls (see Lemma 4.2 in \cite{BB3}), since $r$ is
bounded (being $B\left(  x,r\right)  \subset B\left(  \overline{x},3s\right)
$). From the proof in \cite[Lemma 4.2]{BB3} one reads that $c_{1}$ is a
constant only depending on the doubling constant, that is on $G$. Choosing now
$c=f_{B\left(  x,r\right)  }$ we get%
\[
\frac{1}{\left\vert B\left(  x,r\right)  \cap B\left(  \overline{x},s\right)
\right\vert }\int_{B\left(  x,r\right)  \cap B\left(  \overline{x},s\right)
}\left\vert f\left(  y\right)  -c\right\vert \leq\frac{1}{c_{1}}\left[
f\right]  _{BMO_{loc}\left(  B\left(  \overline{x},s;3s\right)  \right)  }.
\]
If, on the other hand, $x\in B\left(  \overline{x},s\right)  $ but $B\left(
x,r\right)  \varsubsetneq B\left(  \overline{x},3s\right)  $, this means that
$B\left(  x,r\right)  \supset B\left(  \overline{x},s\right)  $, hence%
\begin{align*}
&  \frac{1}{\left\vert B\left(  x,r\right)  \cap B\left(  \overline
{x},s\right)  \right\vert }\int_{B\left(  x,r\right)  \cap B\left(
\overline{x},s\right)  }\left\vert f\left(  y\right)  -c\right\vert dy\\
&  =\frac{1}{\left\vert B\left(  \overline{x},s\right)  \right\vert }%
\int_{B\left(  \overline{x},s\right)  }\left\vert f\left(  y\right)
-c\right\vert dy\leq\left[  f\right]  _{BMO_{loc}\left(  B\left(  \overline
{x},s;3s\right)  \right)  }%
\end{align*}
(choosing $c=f_{B\left(  \overline{x},s\right)  }$).

(b) Pick $x\in B\left(  \overline{x},s\right)  $ and $r>0$ such that $B\left(
x,r\right)  \subset B\left(  \overline{x},3s\right)  $. Then%
\begin{align*}
&  \frac{1}{\left\vert B\left(  x,r\right)  \right\vert }\int_{B\left(
x,r\right)  }\left\vert \widetilde{f}\left(  y\right)  -c\right\vert dy\\
&  =\frac{1}{\left\vert B\left(  x,r\right)  \right\vert }\left\{
\int_{B\left(  x,r\right)  \cap B\left(  \overline{x},R\right)  }\left\vert
f\left(  y\right)  -c\right\vert dy+\int_{B\left(  x,r\right)  \setminus
B\left(  \overline{x},R\right)  }\left\vert c\right\vert dy\right\}
\end{align*}
choosing $c=f_{B\left(  x,r\right)  \cap B\left(  \overline{x},R\right)  }$%
\begin{align*}
&  =\frac{1}{\left\vert B\left(  x,r\right)  \right\vert }\left\{
\int_{B\left(  x,r\right)  \cap B\left(  \overline{x},R\right)  }\left\vert
f\left(  y\right)  -f_{B\left(  x,r\right)  \cap B\left(  \overline
{x},R\right)  }\right\vert dy+\left\vert f_{B\left(  x,r\right)  \cap B\left(
\overline{x},R\right)  }\right\vert \left\vert B\left(  x,r\right)  \setminus
B\left(  \overline{x},R\right)  \right\vert \right\} \\
&  \equiv I+II.
\end{align*}

Now,%
\begin{align*}
I  &  \leq\frac{1}{\left\vert B\left(  x,r\right)  \cap B\left(  \overline
{x},R\right)  \right\vert }\int_{B\left(  x,r\right)  \cap B\left(
\overline{x},R\right)  }\left\vert f\left(  y\right)  -f_{B\left(  x,r\right)
\cap B\left(  \overline{x},R\right)  }\right\vert dy\\
&  \leq\left[  f\right]  _{BMO\left(  B\left(  \overline{x},R\right)  \right)
}.
\end{align*}
On the other hand,%
\[
II=\left\vert f_{B\left(  x,r\right)  \cap B\left(  \overline{x},R\right)
}\right\vert \frac{\left\vert B\left(  x,r\right)  \setminus B\left(
\overline{x},R\right)  \right\vert }{\left\vert B\left(  x,r\right)
\right\vert }.
\]
The term $II$ does not vanish$\ $only if $B\left(  x,r\right)  $ contains a
point outside $B\left(  \overline{x},R\right)  $ (otherwise $\left\vert
B\left(  x,r\right)  \setminus B\left(  \overline{x},R\right)  \right\vert
=0$). Since $x\in B\left(  \overline{x},s\right)  $, this implies $r>R-s\ $and%
\[
\left\vert B\left(  x,r\right)  \cap B\left(  \overline{x},R\right)
\right\vert \geq c\left(  R-s\right)  ^{Q}.
\]
Then, by H\"{o}lder inequality%
\begin{align*}
II  &  \leq\left\vert f_{B\left(  x,r\right)  \cap B\left(  \overline
{x},R\right)  }\right\vert \leq\left(  \frac{1}{\left\vert B\left(
x,r\right)  \cap B\left(  \overline{x},R\right)  \right\vert }\int_{B\left(
x,r\right)  \cap B\left(  \overline{x},R\right)  }\left\vert f\left(
y\right)  \right\vert ^{p}dy\right)  ^{1/p}\\
&  \leq\frac{c}{\left(  R-s\right)  ^{Q/p}}\left\Vert f\right\Vert
_{L^{p}\left(  B\left(  \overline{x},R\right)  \right)  }%
\end{align*}

Then
\begin{equation}
\lbrack\widetilde{f}]_{BMO_{loc}\left(  B\left(  \overline{x},s;3s\right)
\right)  }\leq\lbrack f]_{BMO\left(  B\left(  \overline{x},R\right)  \right)
}+\frac{c}{\left(  R-s\right)  ^{Q/p}}\left\Vert f\right\Vert _{L^{p}\left(
B\left(  \overline{x},s\right)  \right)  } \label{stimaseminorma}%
\end{equation}
from which (b) follows.
\end{proof}

\begin{definition}
[Sobolev and BMO Sobolev spaces]We say that $u\in S^{2,p}\left(
\Omega\right)  $ if%
\[
\left\Vert u\right\Vert _{S^{2,p}\left(  \Omega\right)  }\equiv\sum
_{i,j=1}^{q}\left\Vert X_{i}X_{j}u\right\Vert _{L^{p}\left(  \Omega\right)
}+\sum_{i=1}^{q}\left\Vert X_{i}u\right\Vert _{L^{p}\left(  \Omega\right)
}+\left\Vert u\right\Vert _{L^{p}\left(  \Omega\right)  }<\infty,
\]
where the derivatives $X_{i}u$ are defined in the usual weak (distributional) sense.

We say that $u\in S_{loc}^{2,p,\ast}\left(  \Omega\right)  $ if%
\begin{align*}
&  \left\Vert u\right\Vert _{S_{loc}^{2,p,\ast}\left(  \Omega\right)  }%
\equiv\\
&  \equiv\sum_{i,j=1}^{q}\left\Vert X_{i}X_{j}u\right\Vert _{BMO_{loc}%
^{p}\left(  \Omega\right)  }+\sum_{i=1}^{q}\left\Vert X_{i}u\right\Vert
_{BMO_{loc}^{p}\left(  \Omega\right)  }+\left\Vert u\right\Vert _{BMO_{loc}%
^{p}\left(  \Omega\right)  }<\infty.
\end{align*}

We say that $u\in S^{2,p,\ast}\left(  \Omega\right)  $ if%
\begin{align*}
&  \left\Vert u\right\Vert _{S^{2,p,\ast}\left(  \Omega\right)  }\equiv\\
&  \equiv\sum_{i,j=1}^{q}\left\Vert X_{i}X_{j}u\right\Vert _{BMO^{p}\left(
\Omega\right)  }+\sum_{i=1}^{q}\left\Vert X_{i}u\right\Vert _{BMO^{p}\left(
\Omega\right)  }+\left\Vert u\right\Vert _{BMO^{p}\left(  \Omega\right)
}<\infty.
\end{align*}
Analogously, for $\Omega_{1}\Subset\Omega_{2}\subseteq\Omega$, we can define
the spaces $S_{loc}^{2,p,\ast}\left(  \Omega_{1},\Omega_{2}\right)  $,
replacing $BMO^{p}\left(  \Omega\right)  $ norms with $BMO_{loc}^{p}\left(
\Omega_{1},\Omega_{2}\right)  $.
\end{definition}

\subsection{Assumptions and main results\label{sec main result}}

We now keep the assumptions stated at beginning of
\S \ \ref{subsec function spaces} about the vector fields $X_{1}%
,X_{2},...,X_{q}$, the domain $\Omega\subset\mathbb{R}^{N}$ and the distance
$d$. Let us consider operators of the form
\begin{equation}
Lu\equiv\sum_{i,j=1}^{q}a_{ij}\left(  x\right)  X_{i}X_{j}u \label{L}%
\end{equation}
where $a_{ij}=a_{ji}$ satisfy the \textquotedblleft ellipticity
condition\textquotedblright:
\[
\Lambda|\xi|^{2}\leq\sum_{i,j=1}^{q}a_{ij}\left(  x\right)  \xi_{i}\xi_{j}%
\leq\Lambda^{-1}|\xi|^{2}\text{ }\forall\xi\in{\mathbb{R}}^{q},x\in\Omega.
\]
Moreover%
\[
a_{ij}\in VLMO_{loc}\left(  \Omega\right)  \cap L^{\infty}\left(
\Omega\right)  .
\]
We also assume that the homogeneous dimension (see
\S \ \ref{subsec homogeneous groups}) is $Q\geq3$. This fact (necessary to
apply Folland's results in \cite{fo}) simply rules out the case of uniformly
elliptic equations in two variables. We stress the fact that, instead,
uniformly elliptic operators in $n\geq3$ variables \emph{are }covered by the
present theory. Our main result is the following:

\begin{theorem}
\label{Thm main}Under the above assumptions, for any $\Omega_{1}\Subset
\Omega_{2}\Subset\Omega$, $1<p<\infty$ we have%
\[
\left\Vert u\right\Vert _{S_{loc}^{2,p,\ast}\left(  \Omega_{1},\Omega
_{2}\right)  }\leq c\left\{  \left\Vert Lu\right\Vert _{BMO_{loc}^{p}\left(
\Omega_{2},\Omega\right)  }+\left\Vert u\right\Vert _{BMO_{loc}^{p}\left(
\Omega_{2},\Omega\right)  }\right\}
\]
for any $u\in S_{loc}^{2,p,\ast}\left(  \Omega\right)  $. The constant $c$
only depends on the group $G$, the numbers $p$ and $\Lambda$, the $VLMO_{loc}$
moduli of the coefficients, $\Omega_{1},\Omega_{2}$. The previous inequality
also implies the following%
\[
\left\Vert u\right\Vert _{S^{2,p,\ast}\left(  \Omega_{1}\right)  }\leq
c\left\{  \left\Vert Lu\right\Vert _{BMO_{loc}^{p}\left(  \Omega\right)
}+\left\Vert u\right\Vert _{BMO_{loc}^{p}\left(  \Omega\right)  }\right\}  .
\]

\end{theorem}

We will also prove a similar local estimate stated for \textquotedblleft
standard\textquotedblright\ $BMO^{p}$ spaces, see Theorem
\ref{Thm main variation}.

\section{Representation formulas and reduction to singular integrals of
convolution type}

In order to state suitable representation formulas for the operator $L$ in
(\ref{L}), we proceed as follows. For any $x_{0}\in\Omega$, let us
\textquotedblleft freeze\textquotedblright\ at $x_{0}$ the coefficients
$a_{ij}(x)$, and consider
\[
L_{0}=\sum_{i,j=1}^{q}a_{ij}(x_{0})X_{i}X_{j}.
\]
As shown in \cite{bb1}, this operator can be rewritten as a sum of squares of
H\"{o}r\-man\-der's vector fields; in particular, it is hypoelliptic.
Moreover, $L_{0}$ is left invariant, homogeneous of degree 2, and coincides
with its formal transpose; hence Folland's theory in \cite{fo} applies, and
assures the existence of a $\left(  2-Q\right)  $-homogeneous fundamental
solution, smooth outside the pole. Let us denote it by $\Gamma\left(
x_{0};\cdot\right)  $, to indicate its dependence on the frozen coefficients
$a_{ij}\left(  x_{0}\right)  $. Also, set for $i,j=1,\ldots,q$,
\[
\Gamma_{ij}\left(  x_{0};u\right)  =X_{i}X_{j}\,\left[  \Gamma\left(
x_{0};\cdot\right)  \right]  \left(  u\right)  .
\]
The next theorem summarizes the basic properties of $\Gamma\left(  x_{0}%
;\cdot\right)  $ that we will need in the following. All of them are due to
Folland \cite[Thm. 2.1 and Corollary 2.8]{fo} or Folland-Stein
\cite[Proposition 8.5]{fs} (see also \cite{bb1}).

\begin{theorem}
\label{ith:fundsol:lzero}Assume that the homogeneous dimension of $G$ is
$Q\geq3$. For every $x_{0}\in\Omega$ the operator $L_{0}$ has a unique
fundamental solution $\Gamma\left(  x_{0};\cdot\right)  $ such that:

$(a)$ $\Gamma\left(  x_{0};\cdot\right)  \in C^{\infty}\left(  \mathbb{R}%
^{N}\setminus\left\{  0\right\}  \right)  ;$

$(b)$ $\Gamma\left(  x_{0};\cdot\right)  $ is homogeneous of degree $(2-Q);$

$(c)$ for every test function $f$ and every $v\in\mathbb{R}^{N}$,
\[
f(v)\text{ }=\int_{\mathbb{R}^{N}}\text{ }\Gamma\left(  x_{0}\,;\,u^{-1}\circ
v\right)  \,L_{0}f(u)\,du;
\]
moreover, for every $i,j=1,\ldots,q$, there exist constants $\alpha_{ij}%
(x_{0})$ such that%
\begin{equation}
X_{i}X_{j}\,f(v)\text{ }=\text{ }P.V.\int_{\mathbb{R}^{N}}\text{ }\Gamma
_{ij}\left(  x_{0};\,u^{-1}\circ v\right)  \,L_{0}f(u)du\text{ }+\alpha
_{ij}(x_{0})\cdot L_{0}f(v); \label{rep formula L0}%
\end{equation}

$(d)$ $\Gamma_{ij}\left(  x_{0}\,;\cdot\right)  \in C^{\infty}\left(
\mathbb{R}^{N}\setminus\left\{  0\right\}  \right)  ;$

$(e)$ $\Gamma_{ij}\left(  x_{0}\,;\cdot\right)  $ is homogeneous of degree
$-Q;$

$(f)$ for every $R>r>0$,%
\[
\int_{r<d(0,x)<R}\text{ }\Gamma_{ij}\left(  x_{0};\,x\right)  \,dx=\int
_{d(0,x)=1}\text{ }\Gamma_{ij}\left(  x_{0};\,x\right)  \,d\sigma(x)=0.
\]

\end{theorem}

Here and in the following,%
\[
P.V.\int\left(  ...\right)  dy=\lim_{\varepsilon\rightarrow0}\int_{d\left(
x,y\right)  >\varepsilon}\left(  ...\right)  dy=\lim_{\varepsilon\rightarrow
0}\int_{\rho\left(  x,y\right)  >\varepsilon}\left(  ...\right)  dy.
\]

The cancellation properties stated at point (f) still hold with $d(0,x)$
replaced with $\left\Vert x\right\Vert $.

A second fundamental result we need contains a bound on the derivatives of
$\Gamma$, uniform with respect to $x_{0}$, and is proved in \cite[Thm.
12]{bb1}:

\begin{theorem}
\label{Thm unif BB1}For every multi-index $\beta$, there exists a constant
$c=c(\beta,G,\Lambda)$ such that
\[
\underset{x\in\Omega}{\underset{\left\Vert u\right\Vert =1}{{\sup}}%
}\,\left\vert \left(  \frac{\partial}{\partial u}\right)  ^{\beta}\Gamma
_{ij}\left(  x\,;\,u\right)  \right\vert \,\leq\text{ }c\text{,}%
\]
for any $i,j=1,\ldots,q$; moreover, for the $\alpha_{ij}$'s appearing in
(\ref{rep formula L0}), the uniform bound
\begin{equation}
\underset{x\in\Omega}{{\sup}}\left\vert \alpha_{ij}(x)\right\vert \leq c_{2}
\label{unif bound alfa_ij}%
\end{equation}
holds for some constant $c_{2}=c_{2}\left(  G,\Lambda\right)  $.
\end{theorem}

The above theorem will be useful but not sufficient for our aims. In \S \ 6 we
will also prove the following uniform bound:

\begin{theorem}
\label{Thm unif BF}For any nonnegative integer $p,$ there exists a constant
$c_{\Lambda,p}$ such that for any $x_{1},x_{2},y\in\mathbb{R}^{N}$ we have:
\begin{align}
&  \left\vert X_{i_{1}}X_{i_{2}}...X_{i_{p}}\Gamma\left(  x_{1},y\right)
-X_{i_{1}}X_{i_{2}}...X_{i_{p}}\Gamma\left(  x_{2},y\right)  \right\vert
\label{uniform 6}\\
&  \leq c_{\Lambda,p}\left\Vert A\left(  x_{1}\right)  -A\left(  x_{2}\right)
\right\Vert \left\Vert y\right\Vert ^{2-Q-p}\nonumber
\end{align}
where the differential operators $X_{i_{j}}$ ($i_{j}\in\left\{
1,2,...,q\right\}  $) act on the $y$-variable and $A=\left\{  a_{ij}\right\}
_{i,j=1}^{q}$.
\end{theorem}

Here $\left\Vert y\right\Vert $ is the homogeneous norm in $G$, while
$\left\Vert A\left(  x_{1}\right)  -A\left(  x_{2}\right)  \right\Vert $
denotes the usual matrix norm in $\mathbb{R}^{2q}$.

By the representation formula (\ref{rep formula L0}), writing $L_{0}=L+\left(
L_{0}-L\right)  \ $and then letting $x$ be equal to $x_{0}$, we get the following:

\begin{theorem}
Let $u\in C_{0}^{\infty}\left(  \Omega\right)  $. Then, for $i,j=1,\ldots,q$
and every $x\in\Omega$%
\begin{align}
X_{i}X_{j}u\left(  x\right)   &  =P.V.\int_{\Omega}\Gamma_{ij}\left(
x;y^{-1}\circ x\right)  \left\{  \sum_{h,k=1}^{q}\left[  a_{hk}(x)-a_{hk}%
(y)\right]  \,X_{h}X_{k}\,u(y)\right.  +\label{rep formula}\\
&  \left.  +Lu(y)\frac{{}}{{}}\right\}  dy+\alpha_{ij}(x)\cdot Lu(x)
.\nonumber
\end{align}
The previous formula still holds, for a.e. $x$, if $u=v\phi$ with $v\in
S^{2,p}\left(  \Omega\right)  $ and $\phi\in C_{0}^{\infty}\left(
\Omega\right)  $.
\end{theorem}

In order to rewrite the above formula in a more compact form, let us introduce
the following singular integral operators:
\begin{equation}
K_{ij}f\left(  x\right)  =P.V.\int_{\Omega}\Gamma_{ij}\left(  x;y^{-1}\circ
x\right)  f\left(  y\right)  dy. \label{K_ij}%
\end{equation}
Moreover, for an operator $K$ and a function $a\in L^{\infty}\left(
\Omega\right)  $, define the commutator%
\begin{equation}
C[K,a]\left(  f\right)  =K\left(  af\right)  -a\cdot K\left(  f\right)  .
\label{comm astratto}%
\end{equation}
Then (\ref{rep formula}) becomes%
\begin{equation}
X_{i}X_{j}\,u=K_{ij}\left(  Lu\right)  -\sum_{h,k=1}^{q}C\left[  K_{ij}%
,a_{hk}\right]  \left(  X_{h}X_{k}\,u\right)  +\alpha_{ij}\cdot Lu
\label{rep formula astratta}%
\end{equation}
for any $u\in C_{0}^{\infty}\left(  \Omega\right)  ,i,j=1,\ldots,q$.

Next, we are going to expand the \textquotedblleft variable
kernel\textquotedblright\ $\Gamma_{ij}\left(  x;u\right)  $ in series of
spherical harmonics. At this point it is more convenient to use the $\rho
$-balls (defined by the quasidistance (\ref{rho})), which have the property
that $B_{\rho}\left(  0,r\right)  =D\left(  r\right)  B_{E}\left(  0,1\right)
$, where $B_{E}$ stands for the Euclidean ball.

Let us denote by%
\[
\left\{  Y_{km}\left(  y\right)  \right\}  _{\substack{k=1,...,g_{m}%
\\m=0,1,...,\infty}}
\]
a complete orthonormal system of $L^{2}\left(  \Sigma_{N}\right)  $ consisting
in spherical harmonics; here $m$ is the degree of the harmonic homogeneous
polynomial $Y_{km},$ and $g_{m}$ the dimension of the space of harmonic
homogeneous polynomial of degree $m$ in $N$ variables. Then, as in \cite{bb1},
for any fixed $x\in\Omega$, $y\in\Sigma_{N}$, we can expand:%
\[
\Gamma_{ij}\left(  x;y\right)  =\sum_{m=1}^{\infty}\sum_{k=1}^{g_{m}}%
c_{ij}^{km}\left(  x\right)  \frac{Y_{km}\left(  y^{\prime}\right)
}{\left\Vert y\right\Vert ^{Q}}\text{ \ \ for }i,j=1,\ldots,q
\]
where $y^{\prime}=D\left(  \left\Vert y\right\Vert ^{-1}\right)  y$, so that
$y^{\prime}\in\Sigma_{N}.$

We explicitly note that for $m=0$ the coefficients in the above expansion are
zero, because of the vanishing property of $\Gamma_{ij}\left(  x;\cdot\right)
$. Also, note that the integral of $Y_{km}\left(  y\right)  $ over $\Sigma
_{N}$, for $m\geq1$, is zero. Then
\begin{equation}
K_{ij}\left(  f\right)  (x)=\sum_{m=1}^{\infty}\sum_{k=1}^{g_{m}}c_{ij}%
^{km}(x)\,T_{km}f(x) \label{spher expansion}%
\end{equation}
with
\begin{align}
T_{km}f(x)  &  =P.V.\int\text{ }H_{km}(y^{-1}\circ x)\,f(y)\,dy\nonumber\\
H_{km}\left(  x\right)   &  =\frac{Y_{km}\left(  x^{\prime}\right)
}{\left\Vert x\right\Vert ^{Q}}. \label{spher har kern}%
\end{align}

We will use the following bounds about spherical harmonics:
\begin{align}
g_{m}  &  \leq c(N)\cdot m^{N-2}\text{ \ for every }m=1,2,\ldots\label{g_m}\\
\left\vert \left(  \frac{\partial}{\partial x}\right)  ^{\beta}\,Y_{km}%
(x)\right\vert  &  \leq c(N)\cdot m^{\left(  \frac{N-2}{2}+\left\vert
\beta\right\vert \right)  }\text{ for }x\in\Sigma_{N},k=1,\ldots
,g_{m},m=1,2,\ldots. \label{D Y_km}%
\end{align}

Moreover, if $f\in C^{\infty}(\Sigma_{N})$ and if $f(x)\,\sim\,\sum
_{k,m}\,b_{km}\,Y_{km}(x)$ is the Fourier expansion of $f(x)$ with respect to
$\left\{  Y_{km}\right\}  $, that is
\[
b_{km}=\int_{\Sigma_{N}}f(x)\,Y_{km}(x)\,d\sigma(x)
\]
then, for every positive integer $n$ there exists $c_{n}$ such that
\begin{equation}
\left\vert b_{km}\right\vert \leq c_{n}\cdot m^{-2n}\sup
_{\substack{_{\left\vert \beta\right\vert =2n}\\x\in\Sigma_{N}}}\left\vert
\left(  \frac{\partial}{\partial x}\right)  ^{\beta}\,f(x)\right\vert .
\label{sup bk}%
\end{equation}
In view of Theorem \ref{Thm unif BB1}, we get by (\ref{sup bk}) the following
bound on the coefficients $c_{ij}^{km}\left(  x\right)  $ appearing in the
expansion (\ref{spher expansion}): \ for every positive integer $n$ there
exists a constant $c=c(n,G,\Lambda)$ such that
\begin{equation}
\sup_{x\in\Omega}\left\vert c_{ij}^{km}\left(  x\right)  \right\vert \leq
c(n,G,\Lambda)\cdot m^{-2n} \label{sup fourier}%
\end{equation}
for every $m=1,2,\ldots;$ $k=1,\ldots,g_{m};$ $i,j=1,\ldots,q$.

{}We will also need a bound on the $LMO_{loc}$ seminorm of these coefficients
and the functions $\alpha_{ij}$:

\begin{theorem}
\label{Thm unif bound fourier}For every $n>0$ there exists $c_{n}>0$ such that%
\[
\left[  c_{ij}^{km}\right]  _{LMO_{loc}\left(  B\left(  \overline{x}%
,R_{1};R_{2}\right)  \right)  }\leq c_{n}\cdot m^{-2n}\cdot\left[  A\right]
_{LMO_{loc}\left(  B\left(  \overline{x},R_{1};R_{2}\right)  \right)  }
\label{bound spherical}%
\]
for any $k,m,i,j,R_{1}<R_{2}$ (with $c_{n}$ independent of $R_{1},R_{2}$). We
have set%
\[
\left[  A\right]  _{LMO_{loc}\left(  B\left(  \overline{x},R_{1};R_{2}\right)
\right)  }=\sup_{h,l}\left[  a_{hl}\right]  _{LMO_{loc}\left(  B\left(
\overline{x},R_{1};R_{2}\right)  \right)  }.
\]
Also,%
\begin{equation}
\left[  \alpha_{ij}\right]  _{LMO_{loc}\left(  B\left(  \overline{x}%
,R_{1};R_{2}\right)  \right)  }\leq c\left[  A\right]  _{LMO_{loc}\left(
B\left(  \overline{x},R_{1};R_{2}\right)  \right)  }. \label{bound alfa ij}%
\end{equation}

\end{theorem}

For the proof of the above Theorem we need the following:

\begin{lemma}
\label{Prop alfa i j}With the above notation, we have:%
\[
\alpha_{ij}\left(  x\right)  =-\int_{\Sigma_{N}}X_{j}\Gamma\left(  x,y\right)
\sum_{k=1}^{n}b_{ik}(y)\nu_{k}\,d\sigma(y)
\]
where $\nu=\left(  \nu_{1},\nu_{2},,...\nu_{n}\right)  $ is the outer normal
to $\Sigma_{N}$ (hence $\nu_{k}=y_{k}$, but this is irrelevant) and
$X_{i}=\sum_{k=1}^{n}b_{k}(y)\partial_{y_{k}}.$
\end{lemma}

\begin{proof}
We follow an argument in \cite[proof of Proposition 2.11]{DP}. Let $\eta$ be a
cutoff function such that $0\leq\eta\leq1,$ $\eta\left(  y\right)  =1$ for
$\left\Vert y\right\Vert \geq1,\eta\left(  y\right)  =0$ for $\left\Vert
y\right\Vert \leq1/2,$ and let $\eta_{\varepsilon}\left(  y\right)
=\eta\left(  D\left(  1/\left(  \varepsilon\right)  \right)  y\right)  .$
Reasoning like in the quoted proof, it is enough to show that%
\begin{align*}
A_{ij}^{\varepsilon}\left(  x_{0},x\right)   &  \equiv\int_{\left\Vert
y^{-1}\circ x\right\Vert <\varepsilon}X_{i}\left[  \eta_{\varepsilon}\left(
y^{-1}\circ x\right)  X_{j}\Gamma\left(  x_{0},y^{-1}\circ x\right)  \right]
dy\rightarrow\\
&  \rightarrow-\int_{\Sigma_{N}}X_{j}\Gamma\left(  x_{0},y\right)  \sum
_{k=1}^{n}b_{k}(y)\nu_{k}\,d\sigma(y)
\end{align*}
for $\varepsilon\rightarrow0.$ Namely, $\alpha_{ij}\left(  x\right)
=\lim_{\varepsilon\rightarrow0}A_{ij}^{\varepsilon}\left(  x,x\right)  $.
Actually the frozen point $x_{0}$ is irrelevant in this calculation, so we
will drop it, writing $A_{ij}^{\varepsilon}\left(  x\right)  ,\Gamma\left(
y\right)  ,$ etc. We can write:%
\begin{align*}
A_{ij}^{\varepsilon}\left(  x\right)   &  =\int_{\left\Vert y^{-1}\circ
x\right\Vert <\varepsilon}\left[  X_{i}\eta_{\varepsilon}\cdot X_{j}%
\Gamma+\eta_{\varepsilon}\cdot X_{i}X_{j}\Gamma\right]  \left(  y^{-1}\circ
x\right)  dy=\\
&  =\int_{\left\Vert y\right\Vert <\varepsilon}\left[  X_{i}\eta_{\varepsilon
}\cdot X_{j}\Gamma+\eta_{\varepsilon}\cdot X_{i}X_{j}\Gamma\right]  \left(
y\right)  dy.
\end{align*}
Then, since $X_{i}\left[  \eta_{\varepsilon}\left(  y\right)  \right]
=X_{i}\left[  \eta\left(  D\left(  1/\varepsilon\right)  y\right)  \right]
=\frac{1}{\varepsilon}\left(  X_{i}\eta\right)  \left(  D\left(
1/\varepsilon\right)  y\right)  $ and $\Gamma$ is $2-Q$-homogeneous, the
change of variables $y=D\left(  \varepsilon\right)  w$ gives%
\[
A_{ij}^{\varepsilon}\left(  x\right)  =\int_{\left\Vert w\right\Vert
<1}\left[  X_{i}\eta\cdot X_{j}\Gamma+\eta\cdot X_{i}X_{j}\Gamma\right]
\left(  w\right)  dw=\int_{\left\Vert w\right\Vert <1}X_{i}\left[  \eta\cdot
X_{j}\Gamma\right]  \left(  w\right)  dw.
\]
Next, we apply the divergence theorem, recalling that $X_{i}f\left(  w\right)
=\sum_{k=1}^{n}b_{ik}(w)\partial_{w_{k}}f\left(  w\right)  =\sum_{k=1}%
^{n}\partial_{w_{k}}\left[  b_{ik}f\right]  \left(  w\right)  .$ Then, letting
$\nu_{k}$ be the $k$-th component of the outer normal at the surface
$\left\Vert w\right\Vert =1,$
\[
A_{ij}^{\varepsilon}\left(  x\right)  =-\int_{\left\Vert w\right\Vert
=1}\left[  \eta X_{j}\Gamma\sum_{k=1}^{n}b_{ik}\nu_{k}\right]  \left(
w\right)  d\sigma\left(  w\right)  =-\int_{\left\Vert w\right\Vert =1}\left[
X_{j}\Gamma\sum_{k=1}^{n}b_{ik}\nu_{k}\right]  \left(  w\right)
d\sigma\left(  w\right)  ,
\]
which gives the desired result.
\end{proof}

\bigskip

\begin{proof}
[Proof of Theorem \ref{Thm unif bound fourier} from Theorem \ref{Thm unif BF}%
]First of all, since the vector fields $X_{1},...,X_{q}$ satisfy
H\"{o}rmander's condition, any derivative $D_{y}^{\beta}\Gamma_{ij}\left(
x,y\right)  $ can be expressed as a linear combination of derivatives
$X_{i_{1}}X_{i_{2}}...X_{i_{p}}\Gamma\left(  x,y\right)  $, for $p$ large
enough. Therefore, (\ref{uniform 6}) implies also the following%
\begin{equation}
\left\vert D_{y}^{\beta}\Gamma_{ij}\left(  x_{1},y\right)  -D_{y}^{\beta
}\Gamma_{ij}\left(  x_{2},y\right)  \right\vert \leq c_{\Lambda,\beta
}\left\Vert A\left(  x_{1}\right)  -A\left(  x_{2}\right)  \right\Vert
\left\Vert y\right\Vert ^{2-Q-c\left(  \beta\right)  }.
\label{BLU revisited 2}%
\end{equation}
Then, let $B_{r}$ be any ball centered at some point of $B\left(  \overline
{x},R_{1}\right)  $ and contained in $B\left(  \overline{x},R_{2}\right)  $.
Since%
\[
c_{ij}^{km}\left(  x\right)  =\int_{\Sigma_{N}}\Gamma_{ij}\left(  x,y\right)
\,Y_{km}(y)\,d\sigma(y),
\]
we can write:%
\begin{align*}
c_{ij}^{km}\left(  x\right)  -\left(  c_{ij}^{km}\right)  _{B_{r}}  &
=\int_{\Sigma_{N}}\left[  \Gamma_{ij}\left(  x,y\right)  -\left(  \frac
{1}{|B_{r}|}\int_{B_{r}}\Gamma_{ij}\left(  x,y\right)  dx\right)  \right]
Y_{km}(y)d\sigma(y)\\
&  \equiv\int_{\Sigma_{N}}g_{ij}\left(  x,y\right)  Y_{km}(y)d\sigma(y).
\end{align*}
Then, by (\ref{sup bk}), we know that for every positive integer $n$ there
exists $c_{n}$ such that
\begin{align}
&  \left\vert c_{ij}^{km}\left(  x\right)  -\left(  c_{ij}^{km}\right)
_{B_{r}}\right\vert \leq c_{n}\cdot m^{-2n}\sup_{\left\vert \beta\right\vert
=2n,\text{ }y\in\Sigma_{N}}\left\vert \left(  \frac{\partial}{\partial
y}\right)  ^{\beta}g_{ij}\left(  x,y\right)  \right\vert \nonumber\\
&  =c_{n}\cdot m^{-2n}\sup_{\left\vert \beta\right\vert =2n,\text{ }y\in
\Sigma_{N}}\left\vert D_{y}^{\beta}\Gamma_{ij}\left(  x,y\right)  -\frac
{1}{|B_{r}|}\int_{B_{r}}D_{y}^{\beta}\Gamma_{ij}\left(  u,y\right)
du\right\vert . \label{bound coefficients 1}%
\end{align}

By (\ref{BLU revisited 2}) we have:
\begin{align*}
&  \frac{1}{\left\vert B_{r}\right\vert }\int_{B_{r}}\left\vert D_{y}^{\beta
}\Gamma_{ij}\left(  x,y\right)  -\frac{1}{\left\vert B_{r}\right\vert }%
\int_{B_{r}}D_{y}^{\beta}\Gamma_{ij}\left(  u,y\right)  du\right\vert dx\\
&  =\frac{1}{\left\vert B_{r}\right\vert }\int_{B_{r}}\left\vert \frac
{1}{\left\vert B_{r}\right\vert }\int_{B_{r}}\left[  D_{y}^{\beta}\Gamma
_{ij}\left(  x,y\right)  -D_{y}^{\beta}\Gamma_{ij}\left(  u,y\right)  \right]
du\right\vert dx\\
&  \leq\frac{1}{\left\vert B_{r}\right\vert }\int_{B_{r}}\frac{1}{\left\vert
B_{r}\right\vert }\int_{B_{r}}\left\vert D_{y}^{\beta}\Gamma_{ij}\left(
x,y\right)  -D_{y}^{\beta}\Gamma_{ij}\left(  u,y\right)  \right\vert dudx\\
&  \leq c_{\Lambda,\beta}\left\Vert y\right\Vert ^{2-Q-c\left(  \beta\right)
}\frac{1}{\left\vert B_{r}\right\vert }\int_{B_{r}}\frac{1}{\left\vert
B_{r}\right\vert }\int_{B_{r}}\left\Vert A\left(  x\right)  -A\left(
u\right)  \right\Vert dudx\\
&  \leq2c_{\Lambda,\beta}\left\Vert y\right\Vert ^{2-Q-c\left(  \beta\right)
}\frac{1}{\left\vert B_{r}\right\vert }\int_{B_{r}}\left\Vert A\left(
x\right)  -\frac{1}{\left\vert B_{r}\right\vert }\int_{B_{r}}A\left(
u\right)  du\right\Vert dx
\end{align*}
(where $\left\Vert \cdot\right\Vert $ inside the last integral just denotes
the matrix norm). Therefore by (\ref{bound coefficients 1}) we have:%
\[
\frac{1}{\left\vert B_{r}\right\vert }\int_{B_{r}}\left\vert c_{ij}%
^{km}\left(  x\right)  -\left(  c_{ij}^{km}\right)  _{B_{r}}\right\vert dx\leq
c_{n}\cdot m^{-2n}\cdot\frac{1}{\left\vert B_{r}\right\vert }\int_{B_{r}%
}\left\Vert A\left(  x\right)  -\frac{1}{\left\vert B_{r}\right\vert }%
\int_{B_{r}}A\left(  u\right)  du\right\Vert dx,
\]
and%
\[
\left[  c_{ij}^{km}\right]  _{LMO}\leq c_{r}\cdot m^{-2r}\cdot\left[
A\right]  _{LMO}.
\]

Next, we prove \eqref{bound alfa ij}. By Lemma \ref{Prop alfa i j} we can
write%
\[
\alpha_{ij}\left(  x\right)  -\left(  \alpha_{ij}\right)  _{B_{r}}%
=-\int_{\Sigma_{N}}\left[  \Gamma_{j}\left(  x,y\right)  -\frac{1}{\left\vert
B_{r}\right\vert }\int_{B_{r}}\Gamma_{j}\left(  w,y\right)  dw\right]
\sum_{k=1}^{n}b_{ik}(y)\nu_{k}\,d\sigma(y)
\]
so that, by (\ref{uniform 6}),%
\begin{align*}
\left\vert \alpha_{ij}\left(  x\right)  -\left(  \alpha_{ij}\right)  _{B_{r}%
}\right\vert  &  \leq c\int_{\Sigma_{N}}\left\vert \Gamma_{j}\left(
x,y\right)  -\frac{1}{\left\vert B_{r}\right\vert }\int_{B_{r}}\Gamma
_{j}\left(  w,y\right)  dw\right\vert \,d\sigma(y)\leq\\
&  \leq c\int_{\Sigma_{N}}\left\vert \frac{1}{\left\vert B_{r}\right\vert
}\int_{B_{r}}\left[  \Gamma_{j}\left(  x,y\right)  -\Gamma_{j}\left(
w,y\right)  \right]  dw\right\vert \,d\sigma(y)\leq\\
&  \leq c\int_{\Sigma_{N}}\frac{1}{\left\vert B_{r}\right\vert }\int_{B_{r}%
}\left\vert \Gamma_{j}\left(  x,y\right)  -\Gamma_{j}\left(  w,y\right)
\right\vert dw\,d\sigma(y)\leq\\
&  \leq c\int_{\Sigma_{N}}\frac{1}{\left\vert B_{r}\right\vert }\int_{B_{r}%
}\left\Vert A\left(  x\right)  -A\left(  w\right)  \right\Vert \left\Vert
y\right\Vert ^{1-Q}dw\,d\sigma(y)=\\
&  =c\frac{1}{\left\vert B_{r}\right\vert }\int_{B_{r}}\left\Vert A\left(
x\right)  -A\left(  w\right)  \right\Vert dw\,
\end{align*}
and%
\[
\frac{1}{\left\vert B_{r}\right\vert }\int_{B_{r}}\left\vert \alpha
_{ij}\left(  x\right)  -\left(  \alpha_{ij}\right)  _{B_{r}}\right\vert dx\leq
c\frac{1}{\left\vert B_{r}\right\vert }\int_{B_{r}}\frac{1}{\left\vert
B_{r}\right\vert }\int_{B_{r}}\left\Vert A\left(  x\right)  -A\left(
w\right)  \right\Vert dwdx\leq
\]
as before%
\[
\leq c\frac{1}{\left\vert B_{r}\right\vert }\int_{B_{r}}\left\Vert A\left(
x\right)  -\frac{1}{\left\vert B_{r}\right\vert }\int_{B_{r}}A\left(
\cdot\right)  \right\Vert dx,
\]
which gives the desired bound on $LMO$ norm of $\alpha_{ij}$ in terms of that
of the matrix $A$.
\end{proof}

\section{Singular integral estimates}

The main object of this section is to prove the following two theorems, which
will be the key tool in order to derive our local $BMO^{p}$ estimates from the
representation formula (\ref{rep formula astratta}).

\begin{theorem}
[Singular integral estimate]\label{Thm Singular integral variable kernel}If
$K_{ij}$ are the singular integral operators defined in \eqref{K_ij}, then for
any $p\in\left(  1,\infty\right)  $ there exists $C>0$ such that:
\begin{align*}
\left[  K_{ij}f\right]  _{BMO_{loc}\left(  B\left(  \overline{x},R;3R\right)
\right)  }  &  \leq C\left(  1+\left[  A\right]  _{LMO_{loc}\left(
\Omega\right)  }\right)  \cdot\\
&  \cdot\left(  \left[  f\right]  _{BMO_{loc}\left(  B\left(  \overline
{x},R;3R\right)  \right)  }+\frac{\Vert f\Vert_{L^{p}(B(\overline{x},R))}%
}{|B(\overline{x},R)|^{1/p}}\right)
\end{align*}
for any ball $B\left(  \overline{x},3R\right)  \subset\Omega,f\in
BMO_{loc}^{p}\left(  B\left(  \overline{x},R;3R\right)  \right)  $ with
sprt$f\subset B\left(  \overline{x},R\right)  $. The number $C$ depends on
$p,G,\Lambda$.
\end{theorem}

\begin{theorem}
[Local commutator estimate]\label{Thm commutator variable kernel}Let $b\in
LMO_{loc}(\Omega)$, $K_{ij}$ as before, and $C\left[  K_{ij},b\right]  $ the
commutator, defined as in $\eqref{comm astratto}$. Then for any $p\in\left(
1,\infty\right)  $ there exists a constant $C$ and two absolute constants
$K>H>3$, such that for any $R>0$ with $B\left(  \overline{x},KR\right)
\subset\Omega$, any $f\in BMO_{loc}^{p}\left(  B\left(  \overline
{x},R;3R\right)  \right)  $ with sprt$f\subset$ $B\left(  \overline
{x},R\right)  $,%
\begin{align*}
&  [C\left[  K_{ij},b\right]  f]_{BMO_{loc}\left(  B\left(  \overline
{x},R;3R\right)  \right)  }\leq C\left(  1+\left[  A\right]  _{LMO_{loc}%
\left(  \Omega\right)  }\right)  [b]_{LMO_{loc}\left(  B\left(  \overline
{x},HR;KR\right)  \right)  }\cdot\\
&  \cdot\left\{  \lbrack f]_{BMO_{loc}\left(  B\left(  \overline
{x},R;3R\right)  \right)  }+\frac{1}{\left\vert B\left(  \overline
{x},R\right)  \right\vert ^{1/p}}\Vert f\Vert_{L^{p}\left(  B\left(
\overline{x},R\right)  \right)  }\right\}  .
\end{align*}
The number $C$ depends on $p,G,\Lambda$, but not on $f,b,R$. This means in
particular that, if $b\in VLMO_{loc}(\Omega),$ then for any $\varepsilon>0$
there exists $R>0$ such that%
\[
C[b]_{LMO_{loc}\left(  B\left(  \overline{x},HR;KR\right)  \right)
}<\varepsilon.
\]

\end{theorem}

The above two theorems will be derived exploiting the expansion in spherical
harmonics and the bounds on the corresponding coefficients discussed in the
previous section, applying similar theorems regarding singular integrals of
convolution type (modeled on spherical harmonics), and a multiplication
theorem (see next Theorem \ref{Thm multiplication}). Therefore the plan of
this section is the following: after establishing some basic estimates
regarding $BMO$ type norms (\S \ \ref{subsec real analysis}), and particularly
the aforementioned multiplication theorem, we will state and prove singular
integral estimates for convolution kernels
(\S \ \ref{subsec singular convolution}) and then we will prove Theorems
\ref{Thm Singular integral variable kernel} and
\ref{Thm commutator variable kernel} (\S \ \ref{subsec variable singular}).

\subsection{Preliminary real analysis estimates\label{subsec real analysis}}

Throughout the following, we will need a localized version of two well-known
facts, namely a multiplication theorem and John-Nirenberg theorem.

\begin{theorem}
\label{Thm multiplication}There exists an absolute constant $K>3$ such that if
$f\in BMO_{loc}^{p}\left(  B\left(  \overline{x},R;KR\right)  \right)  $ for
some $1<p<\infty$, and $\psi\in L^{\infty}\cap LMO_{loc}\left(  B\left(
\overline{x},R;3R\right)  \right)  $, then $\psi f\in BMO_{loc}^{p}\left(
B\left(  \overline{x},R;3R\right)  \right)  $ and%
\begin{align*}
\left[  f\psi\right]  _{BMO_{loc}\left(  B\left(  \overline{x},R;3R\right)
\right)  }  &  \leq C\left(  \left\Vert \psi\right\Vert _{L^{\infty}\left(
B\left(  \overline{x},3R\right)  \right)  }+\left[  \psi\right]
_{LMO_{loc}\left(  B\left(  \overline{x},R;3R\right)  \right)  }\right)
\cdot\\
&  \cdot\left(  \lbrack f]_{BMO_{loc}\left(  B\left(  \overline{x}%
,R;KR\right)  \right)  }+\frac{\Vert f\Vert_{L^{p}(B(\overline{x},KR))}%
}{|B(\overline{x},R)|^{1/p}}\right)  ,
\end{align*}
with $C=C\left(  G\right)  .$
\end{theorem}

\begin{proof}
Let $n$ be a positive integer such that $2^{n}r<3R\leq2^{n+1}r$. We set
$B_{k}=B(x_{0},2^{k}r)$, with $k=0,1,...,n$. Since $\int_{B_{\rho}%
}|f-f_{B_{\rho}}|dx\leq2\int_{B_{t}}|f-f_{B_{t}}|dx$ for $\rho<t$, we have
that for $f\in BMO_{loc}^{p}\left(  B\left(  \overline{x},R;13R\right)
\right)  $, $B(x_{0},r)\subset B(\overline{x},3R)$, $x_{0}\in B(\overline
{x},R)$
\begin{align}
&  |f_{B(x_{0},r)}|\leq\sum_{k=0}^{n}|f_{B_{k}}-f_{B_{k+1}}|+|f_{B_{n+1}%
}|\nonumber\\
&  \leq\sum_{k=0}^{n}\frac{1}{|B_{k}|}\int_{B_{k+1}}|f(x)-f_{B_{k+1}}%
|dx+\frac{c}{|B(\overline{x},R)|^{1/p}}\Vert f\Vert_{L^{p}(B(\overline
{x},7R))}\nonumber\\
&  \leq\frac{2^{Q}}{\log2}\sum_{k=0}^{n}\int_{2^{k+1}r}^{2^{k+2}r}\frac{ds}%
{s}\frac{1}{|B_{k+1}|}\int_{B_{k+1}}|f(x)-f_{B_{k+1}}|dx+\frac{c}%
{|B(\overline{x},R)|^{1/p}}\Vert f\Vert_{L^{p}(B(\overline{x},7R))}\nonumber\\
&  \leq c\sum_{k=0}^{n}\int_{2^{k+1}r}^{2^{k+2}r}\left(  \frac{s}{2^{k+1}%
r}\right)  ^{Q}\frac{1}{|B\left(  x_{0},s\right)  |}\int_{B\left(
x_{0},s\right)  }|f(x)-f_{B\left(  x_{0},s\right)  }|dx\frac{ds}{s}+\frac
{c}{|B(\overline{x},R)|^{1/p}}\Vert f\Vert_{L^{p}(B(\overline{x}%
,7R))}\nonumber\\
&  \leq c\int_{2r}^{2^{n+2}r}\frac{ds}{s}[f]_{BMO_{loc}\left(  B\left(
\overline{x},R;13R\right)  \right)  }+\frac{c}{|B(\overline{x},R)|^{1/p}}\Vert
f\Vert_{L^{p}(B(\overline{x},7R))}\nonumber\\
&  \leq c\log\frac{3R}{r}\left[  f\right]  _{BMO_{loc}\left(  B\left(
\overline{x},R;13R\right)  \right)  }+\frac{c}{\left\vert B\left(
\overline{x},R\right)  \right\vert ^{1/p}}\Vert f\Vert_{L^{p}\left(  B\left(
\overline{x},7R\right)  \right)  }. \label{4.1}%
\end{align}

Then (see also Lemma 2.4 in \cite{BB3})
\begin{align*}
&  \frac{1}{|B(x_{0},r)|}\int_{B(x_{0},r)}|\psi f-(\psi f)_{B(x_{0},r)}|dx\\
&  \leq\left\vert \frac{1}{|B(x_{0},r)|}\int_{B(x_{0},r)}|\psi f-(\psi
f)_{B(x_{0},r)}|dx-\frac{|f_{B(x_{0},r)}|}{|B(x_{0},r)|}\int_{B(x_{0},r)}%
|\psi-(\psi)_{B(x_{0},r)}|dx\right\vert \\
&  +\frac{|f_{B(x_{0},r)}|}{|B(x_{0},r)|}\int_{B(x_{0},r)}|\psi-(\psi
)_{B(x_{0},r)}|dx\\
&  \leq\frac{2}{|B(x_{0},r)|}\int_{B(x_{0},r)}|\psi||f-(f)_{B(x_{0},r)}|dx\\
&  +c\left\{  \log\frac{3R}{r}\left[  f\right]  _{BMO_{loc}\left(  B\left(
\overline{x},R;13R\right)  \right)  }+\frac{1}{|B\left(  \overline
{x},R\right)  |^{1/p}}\Vert f\Vert_{L^{p}\left(  B\left(  \overline
{x},7R\right)  \right)  }\right\}  \cdot\\
&  \cdot\frac{1}{|B(x_{0},r)|}\int_{B(x_{0},r)}|\psi-(\psi)_{B(x_{0},r)}|dx\\
&  \leq c\Vert\psi\Vert_{L^{\infty}}[f]_{BMO_{loc}\left(  B\left(
\overline{x},R;3R\right)  \right)  }+c\left[  \psi\right]  _{LMO_{loc}\left(
B\left(  \overline{x},R;3R\right)  \right)  }[f]_{BMO_{loc}\left(  B\left(
\overline{x},R;13R\right)  \right)  }\\
&  +c\left[  \psi\right]  _{LMO_{loc}\left(  B\left(  \overline{x}%
,R;3R\right)  \right)  }\frac{\Vert f\Vert_{L^{p}(B(\overline{x},7R))}%
}{|B(\overline{x},R)|^{1/p}}\,.
\end{align*}

\end{proof}

From Theorem \ref{Thm multiplication} we can derive also the following (see
also Lemma 4.12 in \cite{BB3}):

\begin{corollary}
\label{coroll molt smooth}Let $\psi\in C^{1}\left(  B\left(  x,R\right)
\right)  $ such that, for some $t<s<R,$ $\psi=1$ in $B\left(  x,t\right)  $,
$\psi=0$ outside $B\left(  x,s\right)  ,$ $\psi\leq1$ and $\left\vert
D\psi\right\vert \leq c/\left(  s-t\right)  $. Then for any $f\in
BMO_{loc}\left(  B\left(  \overline{x},R;KR\right)  \right)  $ one has%
\[
\left[  f\psi\right]  _{BMO_{loc}\left(  B\left(  \overline{x},R;3R\right)
\right)  }\leq\frac{C}{\left(  s-t\right)  }\left(  [f]_{BMO_{loc}\left(
B\left(  \overline{x},R;KR\right)  \right)  }+\frac{\Vert f\Vert
_{L^{p}(B(\overline{x},KR))}}{|B(\overline{x},R)|^{1/p}}\right)  .
\]

\end{corollary}

The proof of the following John-Nirenberg Theorem is similar to the proof of
Theorem A in \cite{MMNO}.

\begin{theorem}
\label{JN}There exist positive constants $c_{1}$,$c_{2}$ and $\alpha>1$ such
that for any $f\in BMO_{loc}(B(\overline{x},R;\alpha R))$, any ball
$B(x_{0},r)\subset B(\overline{x},3R)$ with $x_{0}\in B(\overline{x},R)$, and
$\forall\lambda>0$, we have%
\[
\left\vert \{x\in B:\left\vert f(x)-f_{B}\right\vert >\lambda\}\right\vert
\leq c_{1}\exp(-c_{2}\lambda/[f]_{BMO_{loc}(B(\overline{x},R;\alpha
R))})|B|\,.
\]

\end{theorem}

In a standard way it is also possible to prove

\begin{corollary}
\label{coroll JN}Let $f\in BMO_{loc}(B(\overline{x},R;\alpha R))$ and
$1<p<+\infty$. Then there exists a constant $c=c(p)$ such that
\begin{align*}
\sup_{x\in B(\overline{x},R),B\left(  x,r\right)  \subset B(\overline{x},3R)}
&  \left(  \frac{1}{|B\left(  x,r\right)  |}{\int}_{B\left(  x,r\right)
}|f(y)-f_{B\left(  x,r\right)  }|^{p}dy\right)  ^{1/p}\leq\\
&  \leq c[f]_{BMO_{loc}(B(\overline{x},R;\alpha R))}%
\end{align*}
with $\alpha>1$ as in the previous theorem.
\end{corollary}

We now state and prove some further preliminary results which will be useful
in the next subsection.

\begin{lemma}
\label{Lem1}Let $f\in BMO_{loc}\left(  B\left(  \overline{x},R;5R\right)
\right)  $ and $B=B(x_{0},r)\subset B\left(  \overline x,3R\right)  ,$
$x_{0}\in B\left(  \overline{x},R\right)  $. Then $\forall\beta>0$
\[
\int_{B\left(  \overline{x},R\right)  \setminus B\left(  x_{0},2r\right)
}\frac{|f(y)-f_{B}|}{d(x_{0},y)^{Q+\beta}}dy\leq\frac{C}{r^{\beta}%
}[f]_{BMO_{loc}\left(  B\left(  \overline{x},R;5R\right)  \right)  }\text{,}%
\]
where $C$ depends only on $Q$ and $\beta$.
\end{lemma}

\begin{proof}
Set $B_{k}=B(x_{0},2^{k}r)$, $k=0,1,2,..,n,$ with $2^{n}r\leq2R<2^{n+1}r.$%
\begin{align*}
\left\vert f_{B_{k+1}}-f_{B_{k}}\right\vert  &  =\left\vert \frac{1}{|B_{k}%
|}\int_{B_{k}}(f-f_{B_{k+1}})dy\right\vert \\
&  \leq\frac{2^{Q}}{|B_{k+1}|}{\int}_{B_{k+1}}|f(y)-f_{B_{k+1}}|dy\leq
2^{Q}[f]_{BMO_{loc}\left(  B\left(  \overline{x},R;5R\right)  \right)  }\,,
\end{align*}
from which
\begin{equation}
\left\vert f_{B_{k+1}}-f_{B_{0}}\right\vert \leq(k+1)2^{Q}[f]_{BMO_{loc}%
\left(  B\left(  \overline{x},R;5R\right)  \right)  }\,. \label{4.7.1}%
\end{equation}

Then
\begin{align*}
&  \int_{B\left(  \overline{x},R\right)  \setminus B\left(  x_{0},2r\right)
}\frac{|f(y)-f_{B}|}{d(x_{0},y)^{Q+\beta}}dy\leq\int_{B\left(  x_{0}%
,2^{n+1}r\right)  \setminus B\left(  x_{0},2r\right)  }\frac{|f(y)-f_{B}%
|}{d(x_{0},y)^{Q+\beta}}dy\\
&  =\sum_{k=1}^{n}\int_{B_{k+1}\setminus B_{k}}\frac{|f(y)-f_{B}|}%
{d(x_{0},y)^{Q+\beta}}dy\leq\sum_{k=1}^{n}\int_{B_{k+1}}\frac{|f(y)-f_{B}%
|}{(r2^{k})^{Q+\beta}}dy
\end{align*}%
\begin{align}
&  \leq\sum_{k=1}^{n}\frac{c}{(r2^{k})^{\beta}}\frac{1}{\left\vert
B_{k+1}\right\vert }{\int}_{B_{k+1}}|f(y)-f_{B}|dy\label{Lem dis1}\\
&  \leq\sum_{k=1}^{n}\frac{c}{(r2^{k})^{\beta}}([f]_{BMO_{loc}\left(  B\left(
\overline{x},R;5R\right)  \right)  }+|f_{B_{k+1}}-f_{B}|)\nonumber
\end{align}
by (\ref{4.7.1})%
\begin{equation}
\leq\sum_{k=1}^{+\infty}\frac{c}{(r2^{k})^{\beta}}(2+k)[f]_{BMO_{loc}\left(
B\left(  \overline{x},R;5R\right)  \right)  }\leq\frac{c}{r^{\beta}%
}\,[f]_{BMO_{loc}\left(  B\left(  \overline{x},R;5R\right)  \right)
}.\nonumber
\end{equation}

\end{proof}

\begin{lemma}
\label{2.4}Let $f$ be in $LMO_{loc}\left(  B\left(  \overline{x},R;5R\right)
\right)  $ and $B=B(x_{0},r)\subset B\left(  \overline{x},3R\right)  ,$
$x_{0}\in B\left(  \overline{x},R\right)  $. Then $\forall\beta>0$%
\[
\int_{B\left(  \overline{x},R\right)  \setminus B\left(  x_{0},2r\right)
}\frac{|f(y)-f_{B}|}{d(x_{0},y)^{Q+\beta}}dy\leq\frac{C}{r^{\beta}\left(
1+\log\frac{5R}{r}\right)  }[f]_{LMO_{loc}\left(  B\left(  \overline
{x},R;5R\right)  \right)  }\text{,}%
\]
where $C$ depends only on $Q$ and $\beta$.
\end{lemma}

\begin{proof}
The proof is similar to that of (\ref{4.1}). With the same notation, we have%
\[
\left\vert f_{B_{k+1}}-f_{B_{k}}\right\vert \leq\frac{2^{Q}}{\log\frac
{10R}{2^{k+1}r}}[f]_{LMO_{loc}\left(  B\left(  \overline{x},R;5R\right)
\right)  }=\frac{c}{n-k+1}[f]_{LMO_{loc}\left(  B\left(  \overline
{x},R;5R\right)  \right)  }%
\]
from which%
\[
\left\vert f_{B_{k+1}}-f_{B_{0}}\right\vert \leq c\frac{k+1}{n-k+1}%
[f]_{LMO_{loc}\left(  B\left(  \overline{x},R;5R\right)  \right)  }\,.
\]
Then (\ref{Lem dis1}) gives, recalling that $2^{n}r\leq2R<2^{n+1}r$%

\begin{align*}
&  \int_{B\left(  \overline{x},R\right)  \setminus B\left(  x_{0},2r\right)
}\frac{|f(y)-f_{B}|}{d(x_{0},y)^{Q+\beta}}dy\leq\\
&  \leq\sum_{k=1}^{n}\frac{c}{(r2^{k})^{\beta}}\left(  \frac{1}{\left\vert
B_{k+1}\right\vert }{\int}_{B_{k+1}}|f(y)-f_{B_{k+1}}|dy+|f_{B_{k+1}}%
-f_{B_{0}}|\right) \\
&  \leq\sum_{k=1}^{n}\frac{c}{(r2^{k})^{\beta}}\left(  \frac{1}{1+\log\left(
\frac{5R}{2^{k}r}\right)  }+\frac{k+1}{n-k+1}\right)  [f]_{LMO_{loc}\left(
B\left(  \overline{x},R;5R\right)  \right)  }\\
&  \leq\frac{c}{r^{\beta}}[f]_{LMO_{loc}\left(  B\left(  \overline
{x},R;5R\right)  \right)  }\sum_{k=1}^{n}\frac{1}{2^{k\beta}}\left(
\frac{k+2}{n-k+1}\right) \\
&  \leq\frac{c}{r^{\beta}\left(  1+\log\left(  \frac{5R}{r}\right)  \right)
}[f]_{LMO_{loc}\left(  B\left(  \overline{x},R;5R\right)  \right)  }\sum
_{k=1}^{n}\frac{n+1}{2^{k\beta}}\left(  \frac{k+2}{n-k+1}\right) \\
&  \leq\frac{C}{r^{\beta}\left(  1+\log\frac{5R}{r}\right)  }[f]_{LMO_{loc}%
\left(  B\left(  \overline{x},R;5R\right)  \right)  }%
\end{align*}
since%
\begin{align*}
\sum_{k=1}^{n}\frac{n+1}{2^{k\beta}}\left(  \frac{k+2}{n-k+1}\right)   &
=\sum_{k=1}^{n}\frac{k+2}{2^{k\beta}}\left(  1+\frac{k}{n-k+1}\right) \\
&  \leq\sum_{k=1}^{\infty}\frac{(k+2)\left(  k+1\right)  }{2^{k\beta}%
}=c_{\beta}\text{.}%
\end{align*}

\end{proof}

\subsection{Estimates for singular integrals of convolution
type\label{subsec singular convolution}}

In this section we prove the $BMO$-type estimates for singular integrals and
their commutators. Here we deal with the convolution kernels $H_{km}$ defined
in (\ref{spher har kern}). The explicit form of the kernel is not important;
it will be enough to point out the relevant properties which we will use. So,
let
\[
k\left(  x,y\right)  =k_{0}\left(  y^{-1}\circ x\right)
\]
be one of our singular integral kernels, defined in the whole $\mathbb{R}^{N}%
$. The following properties hold true (see \cite[Prop. 1]{bb1}):
\begin{align}
\left\vert k(x,y)\right\vert  &  \leq\frac{A}{d\left(  x,y\right)  ^{Q}}\text{
\ }\forall x,y\in{\mathbb{R}}^{N}\label{standard 1}\\
\left\vert k(x,y)-k(x_{0},y)\right\vert +\left\vert k(y,x)-k(y,x_{0}%
)\right\vert  &  \leq B\frac{d(x_{0},x)}{d(x_{0},y)^{Q+1}} \label{standard 2}%
\end{align}
for any $x_{0},x,y\in\mathbb{R}^{N}$ with $d(x_{0},y)\geq2d(x_{0},x)$.%
\[
\int_{r_{1}<d\left(  x,y\right)  <r_{2}}k\left(  x,y\right)  dy=0=\int
_{r_{1}<d\left(  x,y\right)  <r_{2}}k\left(  y,x\right)  dy
\]
for any $0<r_{1}<r_{2}<\infty$.

The last property requires some comments. It is known that the integral of
$Y_{km}\left(  y\right)  $ over $\Sigma_{N}$, for $m\geq1$, is zero. This also
implies that
\[
\int_{r_{1}<\rho\left(  x,y\right)  <r_{2}}k\left(  x,y\right)  dy=0=\int
_{r_{1}<\rho\left(  x,y\right)  <r_{2}}k\left(  y,x\right)  dy
\]
for any $0<r_{1}<r_{2}<\infty$, since $B_{\rho}\left(  0,r\right)  =D\left(
r\right)  B_{\rho}\left(  0,1\right)  =D\left(  r\right)  B_{E}\left(
0,1\right)  $. It is less obvious that this vanishing property still holds
with respect to $d$-balls. However, this is true in view of the homogeneity of
$d$ (see Proposition \ref{Prop homogeneous distance}), that is a homogeneous norm.

Let now%
\[
Tf\left(  x\right)  =P.V.\int_{\mathbb{R}^{N}}k\left(  x,y\right)  f\left(
y\right)  dy.
\]
All the quantitative estimates that we will prove in this section on the
operator $T$ will depend on $k$ only through the numbers $A,B$ in
(\ref{standard 1})-(\ref{standard 2}). We will show in the next section how to
quantify this dependence in the case of our concrete kernels $H_{km}$.

\bigskip

Throughout this section, let $B\left(  \overline{x},R\right)  $ be a fixed
$d$-ball such that $B\left(  \overline{x},KR\right)  \subset\Omega$ for some
large $K>0$ which will be chosen later. We are interested in studying
$Tf\left(  x\right)  $ and its commutator for sprt$f\subset B\left(
\overline{x},R\right)  $ and $x\in B\left(  \overline{x},3R\right)  ,$ hence
for $d\left(  x,y\right)  <4R$. So, let $\psi\left(  x,y\right)  =\psi
_{0}\left(  y^{-1}\circ x\right)  $ be a cutoff function such that%
\[
B\left(  0,4R\right)  \prec\psi_{0}\prec B\left(  0,5R\right)  .
\]
Hence for sprt$f\subset B\left(  \overline{x},R\right)  $ and $x\in B\left(
\overline{x},3R\right)  $,%
\[
Tf\left(  x\right)  =\int_{\mathbb{R}^{N}}k\left(  x,y\right)  f\left(
y\right)  dy=\int_{\mathbb{R}^{N}}k\left(  x,y\right)  \psi\left(  x,y\right)
f\left(  y\right)  dy\equiv\widetilde{T}f\left(  x\right)  \text{.}%
\]
We will also let%
\[
\widetilde{k}\left(  x,y\right)  =k\left(  x,y\right)  \psi\left(  x,y\right)
\text{.}%
\]
Note that for any $x\in B\left(  \overline{x},3R\right)  ,g\in L_{loc}%
^{1}\left(  \mathbb{R}^{N}\right)  $ we have%
\[
\widetilde{T}g\left(  x\right)  =\int_{B\left(  \overline{x},8R\right)
}k\left(  x,y\right)  \psi\left(  x,y\right)  g\left(  y\right)  dy\text{.}%
\]
Note that the cancellation property%
\[
\int_{r_{1}<d\left(  x,y\right)  <r_{2}}k\left(  x,y\right)  dy=0\text{
\ }\forall r_{1}<r_{2}%
\]
implies that%
\[
\int_{r_{1}<d\left(  x,y\right)  <r_{2}}k\left(  x,y\right)  \psi\left(
x,y\right)  dy=0\text{ \ }\forall r_{1}<r_{2}%
\]
and so%
\begin{equation}
\widetilde{T}\left(  c\right)  =0\text{ in }B\left(  \overline{x},3R\right)
\text{, for any constant }c\text{.} \label{T(c)=0}%
\end{equation}

Now, for $b\in LMO_{loc}\left(  B\left(  \overline{x},R;KR\right)  \right)  $,
$f\in C_{0}^{\infty}\left(  B\left(  \overline{x},R\right)  \right)  ,x\in
B\left(  \overline{x},3R\right)  $%
\begin{align*}
T_{b}f\left(  x\right)   &  =\left[  T,b\right]  \left(  f\right)  \left(
x\right)  =T\left(  bf\right)  \left(  x\right)  -b\left(  x\right)  Tf\left(
x\right) \\
&  =\int_{\mathbb{R}^{N}}k\left(  x,y\right)  \left[  b\left(  y\right)
-b\left(  x\right)  \right]  f\left(  y\right)  dy\\
&  =\left[  \widetilde{T},b\right]  \left(  f\right)  \left(  x\right)
\equiv\widetilde{T}_{b}f\left(  x\right)  .
\end{align*}
Note that for $x\in B\left(  \overline{x},3R\right)  $ and $g\in L_{loc}%
^{1}\left(  \mathbb{R}^{N}\right)  $ we have%
\[
\left[  \widetilde{T},b\right]  \left(  g\right)  \left(  x\right)
=\int_{B\left(  \overline{x},8R\right)  }k\left(  x,y\right)  \psi\left(
x,y\right)  \left[  b\left(  y\right)  -b\left(  x\right)  \right]  g\left(
y\right)  dy
\]
which is meaningful provided $b$ is defined in $B\left(  \overline
{x},8R\right)  $ (hence,\ we will pick $K\geq8$).

The aim of the previous definitions is the following: on the one hand, we want
to define a \textquotedblleft local\textquotedblright\ commutator, without the
necessity of extending the function $b$ to the whole space $\mathbb{R}^{N}$;
but, on the other hand, we need to preserve the strong cancellation property
(\ref{T(c)=0}), which will be essential in the sequel.

\begin{theorem}
\label{2.7}Let $T$ be a singular integral operator as before. There exists an
absolute constant $K>3$ such that for any ball $B\left(  \overline
{x},KR\right)  \subset\Omega$ we have%
\begin{align*}
\lbrack Tf]_{BMO_{loc}\left(  B\left(  \overline{x},R;3R\right)  \right)  }
&  \leq c[f]_{BMO_{loc}\left(  B\left(  \overline{x},R;KR\right)  \right)
};\\
\lbrack Tf]_{LMO_{loc}\left(  B\left(  \overline{x},R;3R\right)  \right)  }
&  \leq c[f]_{LMO_{loc}\left(  B\left(  \overline{x},R;KR\right)  \right)  }%
\end{align*}
for any $f\in BMO_{loc}\left(  B\left(  \overline{x},R;KR\right)  \right)  $
(or $LMO_{loc}$, respectively) with sprt$f\subset B\left(  \overline
{x},R\right)  $, some constant $c$ independent of $R$ and $f$.
\end{theorem}

We recall that the singular integral operator $T$ and the commutator $T_{b}$
are continuous in $L^{p}$ (see \cite{bb2} and \cite{BC}). Moreover we note
that in \cite{BB3} a general continuity result for singular integral operators
on the scale of spaces $BMO_{\phi}$ has been proved, which in particular
applies to $BMO$ and $LMO$. However, the bounds we need here are formulated in
terms of local $BMO$ and $LMO$ spaces; moreover, the strong cancellation
property we can rely on makes it easy to present a short self-contained proof.

\bigskip

\begin{proof}
We are going to prove the second inequality; the proof of the first follows by
the same reasoning, just dropping all the \textquotedblleft$\log
$\textquotedblright\ functions.

Let $B=B\left(  x_{0},r\right)  \subset B\left(  \overline{x},3R\right)  $
with $x_{0}\in B\left(  \overline{x},R\right)  $. Since sprt$f\subset B\left(
\overline{x},R\right)  $ for $x\in B\left(  \overline{x},3R\right)  $
$\widetilde{T}f\left(  x\right)  =Tf\left(  x\right)  ,$ so in the following
we will always handle $\widetilde{T}f$ instead of $Tf$. Let us split%
\[
f(x)=f_{2B}+(f(x)-f_{2B})\chi_{2B}(x)+(f(x)-f_{2B})\chi_{(2B)^{c}}%
(x)=f_{1}+f_{2}(x)+f_{3}(x).
\]

By (\ref{T(c)=0}), $\widetilde{T}f_{1}=0$, hence for any $c\in\mathbb{R}$%
\begin{align*}
&  \frac{1+\log\frac{6R}{r}|}{|B|}\int_{B}\left\vert \widetilde{T}%
f(x)-(\widetilde{T}f)_{B}\right\vert dx\\
&  \leq2\frac{1+\log\frac{6R}{r}}{|B|}\int_{B}\left\vert \widetilde
{T}f(x)-c\right\vert dx\\
&  \leq2\frac{1+\log\frac{6R}{r}}{|B|}\left(  \int_{B}\left\vert \widetilde
{T}f_{2}(x)\right\vert dx+\int_{B}\left\vert \widetilde{T}f_{3}%
(x)-c\right\vert dx\right) \\
&  \equiv I+II.
\end{align*}

By H\"{o}lder and John-Nirenberg inequalities and the $L^{2}$ continuity of
$\widetilde{T}$ we have
\begin{align*}
I  &  \leq2\left(  1+\log\frac{6R}{r}\right)  \left(  \frac{1}{|B|}\int
_{B}|\widetilde{T}f_{2}(x)|^{2}\right)  ^{1/2}\\
&  \leq c\left(  1+\log\frac{6R}{r}\right)  \left(  \frac{1}{|B|}%
\int_{{\mathbb{R}}^{n}}|f_{2}(x)|^{2}dx\right)  ^{1/2}\\
&  \leq c\left(  1+\log\frac{12R}{2r}\right)  \left(  \frac{1}{|2B|}{\int
}_{2B}|f(x)-f_{2B}|^{2}dx\right)  ^{1/2}\\
&  \leq c[f]_{LMO_{loc}\left(  B\left(  \overline{x},R;6R\right)  \right)
}\text{.}%
\end{align*}

To bound $II,$ pick $x^{\ast}\in B$ such that $\widetilde{T}f_{3}(x^{\ast
})<+\infty$ (this is true for a.e. $x^{\ast}\in B$ since $\widetilde{T}%
f_{3}\in L^{2}\left(  B\right)  $) and choose $c=\widetilde{T}f_{3}(x^{\ast}%
)$. Then for any $x\in B$, by (\ref{T(c)=0}) we can write, by \ref{standard 2}%
,%
\begin{align*}
\left\vert \widetilde{T}f_{3}(x)-\widetilde{T}f_{3}(x^{\ast})\right\vert  &
=\left\vert \int\left[  \widetilde{k}(x,y)-\widetilde{k}(x^{\ast},y)\right]
f_{3}(y)dy\right\vert \\
&  =\left\vert \int_{B\left(  \overline{x},8R\right)  \setminus B\left(
x_{0},2r\right)  }\left[  \widetilde{k}(x,y)-\widetilde{k}(x^{\ast},y)\right]
\left[  f(y)-f_{2B}\right]  dy\right\vert \\
&  \leq c\int_{B\left(  \overline{x},8R\right)  \setminus B\left(
x_{0},2r\right)  }\frac{d(x,x^{\ast})}{d(x^{\ast},y)^{Q+1}}|f(y)-f_{2B}|dy\\
&  \leq cr\int_{B\left(  \overline{x},8R\right)  \setminus B\left(
x_{0},2r\right)  }\frac{|f(y)-f_{2B}|}{d(x_{0},y)^{Q+1}}dy.
\end{align*}
Applying Lemma \ref{2.4} with $\beta=1$ we get%
\[
\int_{B\left(  \overline{x},8R\right)  \setminus B\left(  x_{0},2r\right)
}\frac{|f(y)-f_{B}|}{d(x_{0},y)^{Q+1}}dy\leq\frac{C}{r\left(  1+\log\frac
{3R}{r}\right)  }[f]_{LMO_{loc}\left(  B\left(  \overline{x},R;KR\right)
\right)  }%
\]
for some large constant $K$ (independent of $R$), hence%
\[
II\leq c[f]_{LMO_{loc}\left(  B\left(  \overline{x},R;KR\right)  \right)  }%
\]
and we are done.
\end{proof}

\begin{remark}
\label{remark modified T}In the previous theorem the number $3$ has nothing
special: we will also need, in the following, a modified version of the
previous estimate: given a constant $K>3$ there exists a constant $K^{\prime
}>K$ such that for any ball $B\left(  \overline{x},K^{\prime}R\right)
\subset\Omega$ we have%
\begin{align*}
\lbrack Tf]_{BMO_{loc}\left(  B\left(  \overline{x},R;KR\right)  \right)  }
&  \leq c[f]_{BMO_{loc}\left(  B\left(  \overline{x},R;K^{\prime}R\right)
\right)  };\\
\lbrack Tf]_{LMO_{loc}\left(  B\left(  \overline{x},R;3R\right)  \right)  }
&  \leq c[f]_{LMO_{loc}\left(  B\left(  \overline{x},R;KR\right)  \right)  }%
\end{align*}
for any $f\in BMO_{loc}\left(  B\left(  \overline{x},R;K^{\prime}R\right)
\right)  $ (or $LMO_{loc}$, respectively) with sprt$f\subset B\left(
\overline{x},R\right)  .$
\end{remark}

\begin{theorem}
\label{Thm commutator convolution}Let $b\in LMO_{loc}(\Omega)$, then for any
$p\in\left(  1,\infty\right)  $ there exists a constant $C$ such that for any
$R>0$ such that $B\left(  \overline{x},KR\right)  \subset\Omega$, any $f\in
BMO_{loc}^{p}\left(  B\left(  \overline{x},R;3R\right)  \right)  $ with
sprt$f\subset$ $B\left(  \overline{x},R\right)  $,%
\begin{align}
&  [T_{b}f]_{BMO_{loc}\left(  B\left(  \overline{x},R;3R\right)  \right)
}\leq C[b]_{LMO_{loc}\left(  B\left(  \overline{x},4R;KR\right)  \right)
}\cdot\nonumber\\
&  \cdot\left\{  \lbrack f]_{BMO_{loc}\left(  B\left(  \overline
{x},R;3R\right)  \right)  }+\frac{1}{\left\vert B\left(  \overline
{x},R\right)  \right\vert ^{1/p}}\Vert f\Vert_{L^{p}\left(  B\left(
\overline{x},R\right)  \right)  }\right\}  . \label{commutator}%
\end{align}
The number $C$ depends on $p$ and the constants of the singular kernel of $T$,
but not on $f,b,R$. This means in particular that, if $b\in VLMO_{loc}%
(\Omega),$ then for any $\varepsilon>0$ there exists $R>0$ such that%
\[
C[b]_{LMO_{loc}\left(  B\left(  \overline{x},4R;KR\right)  \right)
}<\varepsilon.
\]

\end{theorem}

\begin{remark}
\label{remark modified Tb}Again, the absolute constant $3$ appearing in the
previous estimates is not so important: replacing $3$ with a larger number
simply causes the constants $K\ $and $4$ be replaced by larger constants.

Also, note that by \cite[Thm. 7.1]{BZ},
\begin{equation}
\left\Vert T_{b}f\right\Vert _{L^{p}\left(  B\left(  \overline{x},3R\right)
\right)  }\leq C[b]_{LMO_{loc}\left(  B\left(  \overline{x},4R;KR\right)
\right)  }\left\Vert f\right\Vert _{L^{p}\left(  B\left(  \overline
{x},3R\right)  \right)  } \label{commutator Lp}%
\end{equation}
which, coupled with the above theorem, gives, for a function $f$ supported in
$B\left(  \overline{x},R\right)  $,%
\begin{align*}
&  \left\Vert T_{b}f\right\Vert _{BMO_{loc}^{p}\left(  B\left(  \overline
{x},R;3R\right)  \right)  }\leq C[b]_{LMO_{loc}\left(  B\left(  \overline
{x},4R;KR\right)  \right)  }\cdot\\
&  \cdot\left\{  \lbrack f]_{BMO_{loc}\left(  B\left(  \overline
{x},R;3R\right)  \right)  }+\left(  1+\frac{1}{\left\vert B\left(
\overline{x},R\right)  \right\vert ^{1/p}}\right)  \Vert f\Vert_{L^{p}\left(
B\left(  \overline{x},R\right)  \right)  }\right\}  .
\end{align*}

\end{remark}

\begin{proof}
[Proof of Theorem \ref{Thm commutator convolution}]Since we want to bound
$T_{b}f\left(  x\right)  $ for $f$ supported in $B\left(  \overline
{x},R\right)  $ and $x\in B\left(  \overline{x},3R\right)  $, as remarked at
the beginning of this section we have%
\[
T_{b}f\left(  x\right)  =\widetilde{T}_{b}f\left(  x\right)  ,
\]
hence from now on we will work with $\widetilde{T}_{b}$.

Let $B=B\left(  x_{0},r\right)  \subset$ $B\left(  \overline{x},3R\right)  $
with $x_{0}\in B\left(  \overline{x},R\right)  $, and let us split, like in
the proof of Theorem \ref{2.7}:%
\[
f(x)=f_{2B}+(f(x)-f_{2B})\chi_{2B}(x)+(f(x)-f_{2B})\chi_{(2B)^{c}}%
(x)=f_{1}+f_{2}\left(  x\right)  +f_{3}\left(  x\right)
\]
from which
\begin{align*}
\widetilde{T}_{b}f-(\widetilde{T}_{b}f)_{B}  &  =(\widetilde{T}_{b}%
f_{1}-(\widetilde{T}_{b}f_{1})_{B})+(\widetilde{T}_{b}f_{2}-(\widetilde{T}%
_{b}f_{2})_{B})+\\
+(\widetilde{T}_{b}f_{3}-(\widetilde{T}_{b}f_{3})_{B})  &  =I+II+III\,.
\end{align*}

Then, since $f_{1}$ is constant, we have $\widetilde{T}f_{1}=0$ (see
(\ref{T(c)=0})), hence%
\[
I=\widetilde{T}(bf_{1})-f_{1}(\widetilde{T}b)_{B}=f_{1}(\widetilde
{T}b-(\widetilde{T}b)_{B})
\]
hence by (\ref{4.1}) and Theorem \ref{2.7} we have%
\begin{align}
&  \frac{1}{|B|}\int_{B}|I|dx=\frac{|f_{2B}|}{|B|}\int_{B}|\widetilde
{T}b-(\widetilde{T}b)_{B}|dx\nonumber\\
&  \leq c\left\{  \log\frac{3R}{r}\left[  f\right]  _{BMO_{loc}\left(
B\left(  \overline{x},R;13R\right)  \right)  }+\frac{1}{\left\vert B\left(
\overline{x},R\right)  \right\vert ^{1/p}}\Vert f\Vert_{L^{p}\left(  B\left(
\overline{x},R\right)  \right)  }\right\}  \cdot\nonumber\\
&  \cdot\frac{1}{|B|}\int_{B}|\widetilde{T}b-(\widetilde{T}b)_{B}%
|dx\nonumber\\
&  \leq c\left\{  \frac{\log\frac{3R}{r}}{1+\log\frac{6R}{r}}\left[  f\right]
_{BMO_{loc}\left(  B\left(  \overline{x},R;13R\right)  \right)  }+\frac
{1}{\left\vert B\left(  \overline{x},R\right)  \right\vert ^{1/p}}\Vert
f\Vert_{L^{p}\left(  B\left(  \overline{x},R\right)  \right)  }\right\}
\cdot\nonumber\\
&  \cdot\frac{1+\log\frac{6R}{r}}{|B|}\int_{B}|\widetilde{T}b-(\widetilde
{T}b)_{B}|dx\nonumber\\
&  \leq c[\widetilde{T}b]_{LMO_{loc}\left(  B\left(  \overline{x},R,3R\right)
\right)  }\left\{  \left[  f\right]  _{BMO_{loc}\left(  B\left(  \overline
{x},R;13R\right)  \right)  }+\frac{1}{\left\vert B\left(  \overline
{x},R\right)  \right\vert ^{1/p}}\Vert f\Vert_{L^{p}\left(  B\left(
\overline{x},R\right)  \right)  }\right\} \nonumber\\
&  \leq c[b]_{LMO_{loc}\left(  B\left(  \overline{x},R,KR\right)  \right)
}\left\{  \left[  f\right]  _{BMO_{loc}\left(  B\left(  \overline
{x},R;13R\right)  \right)  }+\frac{1}{\left\vert B\left(  \overline
{x},R\right)  \right\vert ^{1/p}}\Vert f\Vert_{L^{p}\left(  B\left(
\overline{x},R\right)  \right)  }\right\}  . \label{lunga 1}%
\end{align}

Now we consider
\begin{align*}
\frac{1}{|B|}\int_{B}|II|dx  &  =\frac{1}{|B|}\int_{B}|\widetilde{T}_{b}%
f_{2}-(\widetilde{T}_{b}f_{2})_{B}|dx\\
&  \leq\frac{2}{|B|}\int_{B}|\widetilde{T}_{b}f_{2}|dx\leq2|B|^{-1/2}%
\Vert\widetilde{T}_{b}f_{2}\Vert_{L^{2}\left(  B\left(  x_{0},r\right)
\right)  }%
\end{align*}
by H\"{o}lder inequality. Next, we apply the $L^{2}$ boundedness of
$\widetilde{T}_{b}$. The local continuity result on the commutator of
$\widetilde{T}_{b}$ in $L^{p}$ proved in \cite[Thm. 7.1]{BZ} implies that%
\[
\Vert\widetilde{T}_{b}f_{2}\Vert_{L^{2}\left(  B\left(  \overline
{x},3R\right)  \right)  }\leq c[b]_{BMO_{loc}\left(  B\left(  \overline
{x},4R;KR\right)  \right)  }\Vert f_{2}\Vert_{L^{2}\left(  B\left(
\overline{x},3R\right)  \right)  }%
\]
for some absolute constant $K>4,$ with $c$ independent of $R$. Hence, since
$B\left(  x_{0},r\right)  \subset B\left(  \overline{x},3R\right)  ,$%
\begin{align}
\frac{1}{|B|}\int_{B}|II|dx  &  \leq\nonumber\\
&  \leq c|B|^{-1/2}[b]_{BMO_{loc}\left(  B\left(  \overline{x},4R;KR\right)
\right)  }\Vert f_{2}\Vert_{L^{2}\left(  B\left(  x_{0},2r\right)  \right)
}\nonumber\\
&  =c[b]_{BMO_{loc}\left(  B\left(  \overline{x},4R;KR\right)  \right)
}\left(  \frac{1}{|2B|}\int_{2B}|f(x)-f_{2B}|^{2}dx\right)  ^{1/2}\nonumber\\
&  \leq c[b]_{BMO_{loc}\left(  B\left(  \overline{x},4R;KR\right)  \right)
}[f]_{BMO_{loc}\left(  B\left(  \overline{x},R;KR\right)  \right)  }
\label{lunga 2}%
\end{align}
where we have used also Corollary \ref{coroll JN}.

Last, we come to%
\begin{align*}
III  &  =\widetilde{T}(bf_{3})(x)-b(x)\widetilde{T}f_{3}(x)-\frac{1}{|B|}%
\int_{B}(\widetilde{T}_{b}f_{3})(z)dz=\\
&  =\widetilde{T}(\left(  b-b_{B}\right)  f_{3})-\left(  b-b_{B}\right)
\widetilde{T}f_{3}+\\
&  -\left(  \frac{1}{|B|}\int_{B}\widetilde{T}((b-b_{B})f_{3})(z)dz-\frac
{1}{|B|}\int_{B}(b(z)-b_{B})\widetilde{T}f_{3}(z)dz\right)
\end{align*}%
\begin{align*}
&  =-(b(x)-b_{B})(\widetilde{T}f_{3}(x)-\widetilde{T}f_{3}(x_{0}%
))-(b(x)-b_{B})\widetilde{T}f_{3}(x_{0})+\\
&  +\frac{1}{|B|}\int_{B}(b(z)-b_{B})(\widetilde{T}f_{3}(z)-\widetilde{T}%
f_{3}(x_{0}))dz+\\
&  -\frac{1}{|B|}\int_{B}[\widetilde{T}((b-b_{B})f_{3})(z)-\widetilde
{T}((b-b_{B})f_{3})(x)]dz\\
&  \equiv III_{1}+III_{2}+III_{3}+III_{4}\,.
\end{align*}

First, we want to bound, for $x\in B$,
\begin{align*}
\left\vert \widetilde{T}f_{3}(x)-\widetilde{T}f_{3}(x_{0})\right\vert  &
=\left\vert \int_{{\mathbb{R}}^{n}\setminus2B}(\widetilde{k}(x,y)-\widetilde
{k}(x_{0},y)(f(y)-f_{2B})dy\right\vert \leq\\
&  \leq c\int_{B\left(  \overline{x},7R\right)  \setminus B\left(
x_{0},2r\right)  }\frac{d(x_{0}.x)}{d(x_{0},y)^{Q+1}}|f(y)-f_{2B}|dy\\
&  =c\int_{B\left(  \overline{x},R\right)  \setminus B\left(  x_{0},2r\right)
}\frac{d(x_{0},x)}{d(x_{0},y)^{Q+1}}|f(y)-f_{2B}|dy+\\
&  +c\int_{B\left(  \overline{x},7R\right)  \setminus\left(  B\left(
\overline{x},R\right)  \cup B\left(  x_{0},2r\right)  \right)  }\frac
{d(x_{0},x)}{d(x_{0},y)^{Q+1}}|f(y)-f_{2B}|dy\\
&  \equiv I+II
\end{align*}

By Lemma \ref{Lem1}, with $\beta=1$
\begin{align*}
I  &  \leq cr\left(  \int_{B\left(  \overline{x},R\right)  \setminus B\left(
x_{0},2r\right)  }\frac{|f(y)-f_{B}|}{d(x_{0},y)^{Q+1}}dy+\int_{B\left(
\overline{x},R\right)  \setminus B\left(  x_{0},2r\right)  }\frac
{|f_{B}-f_{2B}|}{d(x_{0},y)^{Q+1}}dy\right) \\
&  \leq cr\left\{  \frac{1}{r}[f]_{BMO_{loc}\left(  B\left(  \overline
{x},R;5R\right)  \right)  }+\frac{1}{r}|f_{B}-f_{2B}|\right\}  \leq
c[f]_{BMO_{loc}\left(  B\left(  \overline{x},R;5R\right)  \right)  }.
\end{align*}
Since sprt$f\subset B\left(  \overline{x},R\right)  ,$%
\[
II=c|f_{2B}|\int_{B\left(  \overline{x},7R\right)  \setminus\left(  B\left(
\overline{x},R\right)  \cup B\left(  x_{0},2r\right)  \right)  }\frac
{d(x_{0},x)}{d(x_{0},y)^{Q+1}}dy\leq c|f_{2B}|
\]
by (\ref{4.1})%
\[
\leq c\log\frac{3R}{r}\left[  f\right]  _{BMO_{loc}\left(  B\left(
\overline{x},R;13R\right)  \right)  }+\frac{c}{\left\vert B\left(
\overline{x},R\right)  \right\vert ^{1/p}}\Vert f\Vert_{L^{p}\left(  B\left(
\overline{x},R\right)  \right)  }.
\]
Therefore%
\begin{align*}
\left\vert \widetilde{T}f_{3}(x)-\widetilde{T}f_{3}(x_{0})\right\vert  &  \leq
c\left(  1+\log\frac{3R}{r}\right)  \left[  f\right]  _{BMO_{loc}\left(
B\left(  \overline{x},R;13R\right)  \right)  }+\\
&  +\frac{c}{\left\vert B\left(  \overline{x},R\right)  \right\vert ^{1/p}%
}\Vert f\Vert_{L^{p}\left(  B\left(  \overline{x},R\right)  \right)  }%
\end{align*}
which implies
\begin{align}
&  \frac{1}{|B|}\int_{B}\left\vert III_{1}\right\vert dx\leq\frac{1}{|B|}%
\int_{B}\left\vert b(x)-b_{B}\right\vert dx\cdot\label{7}\\
&  \cdot\left\{  c\left(  1+\log\frac{3R}{r}\right)  \left[  f\right]
_{BMO_{loc}\left(  B\left(  \overline{x},R;13R\right)  \right)  }+\frac
{c}{\left\vert B\left(  \overline{x},R\right)  \right\vert ^{1/p}}\Vert
f\Vert_{L^{p}\left(  B\left(  \overline{x},R\right)  \right)  }\right\}
\nonumber\\
&  \leq c\left[  f\right]  _{BMO_{loc}\left(  B\left(  \overline
{x},R;13R\right)  \right)  }[b]_{LMO_{loc}\left(  B\left(  \overline
{x},R;3R\right)  \right)  }+\nonumber\\
&  +\frac{c}{\left\vert B\left(  \overline{x},R\right)  \right\vert ^{1/p}%
}\Vert f\Vert_{L^{p}\left(  B\left(  \overline{x},R\right)  \right)
}[b]_{BMO_{loc}\left(  B\left(  \overline{x},R;3R\right)  \right)  }.\nonumber
\end{align}
and%
\begin{align}
\frac{1}{|B|}\int_{B}|III_{3}|dx  &  \leq c\left[  f\right]  _{BMO_{loc}%
\left(  B\left(  \overline{x},R;13R\right)  \right)  }[b]_{LMO_{loc}\left(
B\left(  \overline{x},R;3R\right)  \right)  }\label{11}\\
&  +\frac{c}{\left\vert B\left(  \overline{x},R\right)  \right\vert ^{1/p}%
}\Vert f\Vert_{L^{p}\left(  B\left(  \overline{x},R\right)  \right)
}[b]_{BMO_{loc}\left(  B\left(  \overline{x},R;3R\right)  \right)  }.\nonumber
\end{align}

In order to bound $III_{2}$, we now start proving that $\widetilde{T}%
f_{3}(x_{0})$ exists and satisfies the estimate
\begin{equation}
\left\vert \widetilde{T}f_{3}(x_{0})\right\vert \leq c\left(  1+\log\left(
\frac{4R}{r}\right)  \right)  [f]_{BMO_{loc}\left(  B\left(  \overline
{x},R;5R\right)  \right)  }\,.\label{f_3}%
\end{equation}
Indeed, let $j=1,2,...,n\ $with $2^{n}r<7R\leq2^{n+1}r,$ $B_{j}=B\left(
x_{0},2^{j}r\right)  $. Then (here we have to modify the technique previously
used, to exploit the cancellation property of the kernel $\widetilde{k}$)%
\begin{align*}
\left\vert \widetilde{T}f_{3}(x_{0})\right\vert  &  =\left\vert \int_{B\left(
\overline{x},6R\right)  \setminus B(x_{0},2r)}\widetilde{k}(x_{0}%
,y)(f(y)-f_{2B})dy\right\vert =\\
&  =\left\vert \int_{B(x_{0},2^{n+1}r)\setminus B(x_{0},2r)}\left(
...\right)  dy-\int_{B(x_{0},2^{n+1}r)\setminus B\left(  \overline
{x},6R\right)  }\left(  ...\right)  dy\right\vert \equiv\left\vert
A-B\right\vert .
\end{align*}%
\[
\left\vert A\right\vert \leq\sum_{j=1}^{n}\left\vert \int_{B_{j+1}\setminus
B_{j}}\widetilde{k}(x_{0},y)(f(y)-f_{2B})dy\right\vert =\sum_{j=1}%
^{n}\left\vert \int_{B_{j+1}\setminus B_{j}}\widetilde{k}(x_{0}%
,y)(f(y)-f_{B_{j+1}})dy\right\vert
\]
by the cancellation property of $\widetilde{k}$%
\begin{align*}
&  \leq c\sum_{j=1}^{n}\int_{B_{j+1}\setminus B_{j}}\frac{1}{d(x_{0},y)^{Q}%
}|f(y)-f_{B_{j+1}}|dy\\
&  \leq c\sum_{j=1}^{n}\frac{1}{|B_{j+1}|}\int_{B_{j+1}}|f(y)-f_{B_{j+1}}|dy\\
&  \leq cn[f]_{BMO_{loc}\left(  B\left(  \overline{x},R;5R\right)  \right)
}\leq c\left(  1+\log\left(  \frac{4R}{r}\right)  \right)  [f]_{BMO_{loc}%
\left(  B\left(  \overline{x},R;5R\right)  \right)  }.
\end{align*}
On the other hand,%
\begin{align}
\left\vert B\right\vert  &  \leq\int_{B(x_{0},2^{n+1}r)\setminus B\left(
\overline{x},6R\right)  }\left\vert \widetilde{k}(x_{0},y)(f(y)-f_{2B}%
)\right\vert dy\nonumber\\
&  \leq\frac{c}{R^{Q}}\int_{B(x_{0},2^{n+1}r)}\left\vert f(y)-f_{2B}%
)\right\vert dy\leq\frac{c}{\left\vert B_{n+1}\right\vert }\int_{B_{n+1}%
}\left\vert f(y)-f_{2B})\right\vert dy\label{9-}\\
&  \leq\frac{c}{\left\vert B_{n+1}\right\vert }\int_{B_{n+1}}\left\vert
f(y)-f_{B_{n+1}}+f_{B_{n+1}}-f_{B_{n}}+\cdots-f_{2B}\right\vert dy\nonumber\\
&  \leq\frac{c}{\left\vert B_{n+1}\right\vert }\int_{B_{n+1}}\left\vert
f(y)-f_{B_{n+1}}\right\vert dy+\sum_{j=2}^{n}\left\vert f_{B_{j+1}}-f_{B_{j}%
}\right\vert \nonumber\\
&  \leq cn[f]_{BMO_{loc}\left(  B\left(  \overline{x},R;5R\right)  \right)
}\leq c\left(  1+\log\left(  \frac{4R}{r}\right)  \right)  [f]_{BMO_{loc}%
\left(  B\left(  \overline{x},R;5R\right)  \right)  },\nonumber
\end{align}
hence \eqref{f_3} is proved. Therefore
\begin{align}
\frac{1}{|B|}\int_{B}|III_{2}|dx &  \leq c\log\left(  \frac{4R}{r}\right)
[f]_{BMO_{loc}\left(  B\left(  \overline{x},R;5R\right)  \right)  }\frac
{1}{|B|}\int_{B}|b(x)-b_{B}|dx\label{9}\\
&  \leq c[f]_{BMO_{loc}\left(  B\left(  \overline{x},R;5R\right)  \right)
}[b]_{LMO_{loc}\left(  B\left(  \overline{x},R;3R\right)  \right)
}\,.\nonumber
\end{align}

Now we want to bound, $\forall x,y\in B$,
\begin{align}
&  \left\vert \widetilde{T}((b-b_{B})f_{3})(x)-\widetilde{T}((b-b_{B}%
)f_{3})(y)\right\vert \label{5}\\
&  \leq\int_{B\left(  \overline{x},8R\right)  \setminus B\left(
x_{0},2r\right)  }|\widetilde{k}(x,z)-\widetilde{k}(y,z)||b(z)-b_{B}%
||f(z)-f_{2B}|\,dz\nonumber\\
&  \leq c\int_{B\left(  \overline{x},8R\right)  \setminus B\left(
x_{0},2r\right)  }\frac{d(x,y)}{d(x_{0},z)^{Q+1}}|b(z)-b_{B}||f(z)-f_{2B}%
|dz\nonumber\\
&  \leq cr\int_{B\left(  \overline{x},8R\right)  \setminus B\left(
x_{0},2r\right)  }\frac{1}{d(x_{0},z)^{Q+1}}|b(z)-b_{B}||f(z)-f_{2B}%
|dz.\nonumber
\end{align}
Let $B_{j}=B\left(  x_{0},2^{j}r\right)  $, let $n$ be such that
$2^{n}r<9R\leq2^{n+1}r$. Then%

\begin{align*}
&  \left\vert \widetilde{T}((b-b_{B})f_{3})(x)-\widetilde{T}((b-b_{B}%
)f_{3})(y)\right\vert \\
&  \leq cr\sum_{j=1}^{n}\int_{B_{j+1}\setminus B_{j}}\frac{|b(z)-b_{B}%
||f(z)-f_{2B}|}{d(x_{0},z)^{Q+1}}dz\\
&  \leq cr\sum_{j=1}^{n}\frac{1}{(2^{j}r)^{Q+1}}\int_{B_{j+1}}|b(z)-b_{B}%
||f(z)-f_{2B}|dz\\
&  \leq c\sum_{j=1}^{n}\frac{1}{2^{j}|B_{j+1}|}\left(  \int_{B_{j+1}%
}|b(z)-b_{B}|^{2}dz\right)  ^{1/2}\left(  \int_{B_{j+1}}|f(z)-f_{2B}%
|^{2}dz\right)  ^{1/2}\\
&  \leq c\sum_{j=1}^{n}\frac{1}{2^{j}}\left(  \frac{1}{|B_{j+1}|}{\int
}_{B_{j+1}}|b(z)-b_{B}|^{2}dz\right)  ^{1/2}\left(  \frac{1}{|B_{j+1}|}{\int
}_{B_{j+1}}|f(z)-f_{2B}|^{2}dz\right)  ^{1/2}\,.
\end{align*}
We observe that, reasoning like in (\ref{9-})
\begin{align}
&  \left(  \frac{1}{|B_{j+1}|}\int_{B_{j+1}}|f(z)-f_{2B}|^{2}dz\right)
^{1/2}\label{4}\\
&  \leq cj\left(  \frac{1}{|B_{j+1}|}\int_{B_{j+1}}|f(z)-f_{B_{j+1}}%
|^{2}dz\right)  ^{1/2}\nonumber\\
&  \leq cj[f]_{BMO_{loc}\left(  B\left(  \overline{x},R;19R\right)  \right)
}\,,\nonumber
\end{align}
and in the same way
\begin{equation}
\left(  \frac{1}{|B_{j+1}|}\int_{B_{j+1}}|b(z)-b_{B}|^{2}dz\right)  ^{1/2}\leq
c\left(  j+1\right)  [b]_{BMO_{loc}\left(  B\left(  \overline{x},R;19R\right)
\right)  }. \label{3}%
\end{equation}

Then, by \eqref{4} and \eqref{3}, \eqref{5} gives%
\begin{align*}
&  \left\vert \widetilde{T}((b-b_{B})f_{3})(x)-\widetilde{T}((b-b_{B}%
)f_{3})(y)\right\vert \\
&  \leq c\sum_{j=1}^{\infty}\frac{1}{2^{j}}j\left(  j+1\right)  [b]_{BMO_{loc}%
\left(  B\left(  \overline{x},R;19R\right)  \right)  }[f]_{BMO_{loc}\left(
B\left(  \overline{x},R;19R\right)  \right)  }\\
&  \leq c[b]_{BMO_{loc}\left(  B\left(  \overline{x},R;19R\right)  \right)
}[f]_{BMO_{loc}\left(  B\left(  \overline{x},R;19R\right)  \right)  }.
\end{align*}

Hence%
\begin{equation}
\frac{1}{|B|}\int_{B}|III_{4}|dx\leq c[b]_{BMO_{loc}\left(  B\left(
\overline{x},R;19R\right)  \right)  }[f]_{BMO_{loc}\left(  B\left(
\overline{x},R;19R\right)  \right)  }. \label{10}%
\end{equation}

From (\ref{lunga 1}), (\ref{lunga 2}), \eqref{7}, \eqref{11}, \eqref{9}, and
\eqref{10} we have%
\begin{align*}
&  [\widetilde{T}_{b}f]_{BMO_{loc}\left(  B\left(  \overline{x},R;3R\right)
\right)  }\\
&  \leq c[b]_{LMO_{loc}\left(  B\left(  \overline{x},R;KR\right)  \right)
}\left\{  \left[  f\right]  _{BMO_{loc}\left(  B\left(  \overline
{x},R;13R\right)  \right)  }+\frac{c}{\left\vert B\left(  \overline
{x},R\right)  \right\vert ^{1/p}}\Vert f\Vert_{L^{p}\left(  B\left(
\overline{x},R\right)  \right)  }\right\} \\
&  +c[b]_{BMO_{loc}\left(  B\left(  \overline{x},4R;KR\right)  \right)
}[f]_{BMO_{loc}\left(  B\left(  \overline{x},R;KR\right)  \right)  }+\\
&  +c[b]_{LMO_{loc}\left(  B\left(  \overline{x},R;3R\right)  \right)
}\left[  f\right]  _{BMO_{loc}\left(  B\left(  \overline{x},R;13R\right)
\right)  }+\\
&  +\frac{c}{\left\vert B\left(  \overline{x},R\right)  \right\vert ^{1/p}%
}\Vert f\Vert_{L^{p}\left(  B\left(  \overline{x},R\right)  \right)
}[b]_{BMO_{loc}\left(  B\left(  \overline{x},R;3R\right)  \right)  }+\\
&  +c[b]_{LMO_{loc}\left(  B\left(  \overline{x},R;3R\right)  \right)
}[f]_{BMO_{loc}\left(  B\left(  \overline{x},R;5R\right)  \right)  }+\\
&  +c[b]_{BMO_{loc}\left(  B\left(  \overline{x},R;19R\right)  \right)
}[f]_{BMO_{loc}\left(  B\left(  \overline{x},R;19R\right)  \right)  }%
\end{align*}%
\[
\leq c[b]_{LMO_{loc}\left(  B\left(  \overline{x},4R;KR\right)  \right)
}\left\{  [f]_{BMO_{loc}\left(  B\left(  \overline{x},R;KR\right)  \right)
}+\frac{c}{\left\vert B\left(  \overline{x},R\right)  \right\vert ^{1/p}}\Vert
f\Vert_{L^{p}\left(  B\left(  \overline{x},R\right)  \right)  }\right\}
\]
by Lemma \ref{Lem KR}%
\[
\leq c[b]_{LMO_{loc}\left(  B\left(  \overline{x},4R;KR\right)  \right)
}\cdot\left\{  \lbrack f]_{BMO_{loc}\left(  B\left(  \overline{x},R;3R\right)
\right)  }+\frac{1}{\left\vert B\left(  \overline{x},R\right)  \right\vert
^{1/p}}\Vert f\Vert_{L^{p}\left(  B\left(  \overline{x},R\right)  \right)
}\right\}
\]
and (\ref{commutator}) is proved. The last assertion in the statement of the
theorem then follows by (\ref{eta f}).
\end{proof}

\subsection{Estimates for singular integrals with variable
kernels\label{subsec variable singular}}

Now we are ready to prove the two theorems stated at the beginning of this
section, regarding singular integrals with variable kernels and their commutator.

\bigskip

\begin{proof}
[Proof of Theorem \ref{Thm Singular integral variable kernel}]Let $f\in
BMO_{loc}\left(  B\left(  \overline{x},R;3R\right)  \right)  $ with
sprt$f\subset B\left(  \overline{x},R\right)  $, $B\left(  \overline
{x},3R\right)  \subset\Omega$. By the expansion in spherical harmonics
(\ref{spher expansion}) and Theorem \ref{Thm multiplication},%
\begin{align*}
&  \left[  K_{ij}f\right]  _{BMO_{loc}\left(  B\left(  \overline
{x},R;3R\right)  \right)  }\leq\sum_{m=1}^{\infty}\sum_{k=1}^{g_{m}}\left[
c_{ij}^{km}\,T_{km}f\right]  _{BMO_{loc}\left(  B\left(  \overline
{x},R;3R\right)  \right)  }\\
&  \leq C\sum_{m=1}^{\infty}\sum_{k=1}^{g_{m}}\left(  \left\Vert c_{ij}%
^{km}\right\Vert _{L^{\infty}\left(  B\left(  \overline{x},3R\right)  \right)
}+\left[  c_{ij}^{km}\right]  _{LMO_{loc}\left(  B\left(  \overline
{x},R;3R\right)  \right)  }\right)  \cdot\\
&  \cdot\left(  \left[  T_{km}f\right]  _{BMO_{loc}\left(  B\left(
\overline{x},R,KR\right)  \right)  }+\frac{\Vert T_{km}f\Vert_{L^{p}%
(B(\overline{x},KR))}}{|B(\overline{x},R)|^{1/p}}\right)  \,.
\end{align*}
We now apply Theorem \ref{2.7} with Remark \ref{remark modified T}, getting%
\begin{align*}
\left[  T_{km}f\right]  _{BMO_{loc}\left(  B\left(  \overline{x},R;KR\right)
\right)  }  &  \leq c\left(  p,G,k,m\right)  \left[  f\right]  _{BMO_{loc}%
\left(  B\left(  \overline{x},R;K^{\prime}R\right)  \right)  }\\
\left\Vert T_{km}f\right\Vert _{L^{p}\left(  B\left(  \overline{x},KR\right)
\right)  }  &  \leq c\left(  p,G,k,m\right)  \left\Vert f\right\Vert
_{L^{p}\left(  B\left(  \overline{x},KR\right)  \right)  }%
\end{align*}
where the constant $c\left(  p,G,k,m\right)  $ depends on $k,m$ only through
the constants $A,B$ appearing in (\ref{standard 1}), (\ref{standard 2}) when
the kernel is $H_{km}$. In turn, these constants are bounded in terms of%
\[
\sup_{\left\vert u\right\vert =1}\left\vert H_{km}\left(  u\right)
\right\vert +\sup_{\left\vert u\right\vert =1}\left\vert \nabla H_{km}\left(
u\right)  \right\vert \leq c(G)\cdot m^{\frac{N}{2}}.
\]
by (\ref{D Y_km}). A standard linearity argument shows then that actually%
\[
c\left(  p,G,k,m\right)  \leq c(p,G)\cdot m^{\frac{N}{2}}.
\]
Next, we use the bounds on the coefficients $c_{ij}^{km}$ contained in
(\ref{sup fourier}) and Theorem \ref{Thm unif bound fourier}. They assure that
for any positive integer $n$ there exists a constant $c(n,G,\Lambda)$ such
that%
\begin{align*}
\left\Vert c_{ij}^{km}\right\Vert _{L^{\infty}}+\left[  c_{ij}^{km}\right]
_{LMO_{loc}\left(  B\left(  \overline{x},R;3R\right)  \right)  }  &  \leq
c(n,G,\Lambda)\cdot m^{-2n}\left(  1+\left[  A\right]  _{LMO_{loc}\left(
B\left(  \overline{x},R;3R\right)  \right)  }\right) \\
&  \leq c(n,G,\Lambda)\cdot m^{-2n}\left(  1+\left[  A\right]  _{LMO_{loc}%
\left(  \Omega\right)  }\right)
\end{align*}
where $A=\left\{  a_{ij}\right\}  _{i,j=1}^{q}$. Recalling also the bound
(\ref{g_m}) on the number $g_{m},$ we get%
\begin{align*}
&  \left[  K_{ij}f\right]  _{BMO_{loc}\left(  B\left(  \overline
{x},R,3R\right)  \right)  }\leq C\sum_{m=1}^{\infty}c(N)m^{N-2}\cdot
c(n,G,\Lambda)m^{-2n}\left(  1+\left[  A\right]  _{LMO_{loc}\left(
\Omega\right)  }\right)  \cdot\\
&  \cdot c(p,G)\cdot m^{\frac{N}{2}}\left(  \left[  f\right]  _{BMO_{loc}%
\left(  B\left(  \overline{x},R;K^{\prime}R\right)  \right)  }+\frac{\Vert
f\Vert_{L^{p}(B(\overline{x},R))}}{|B(\overline{x},R)|^{1/p}}\right) \\
&  =c\left(  G\right)  \left(  1+\left[  A\right]  _{LMO_{loc}\left(
\Omega\right)  }\right)  \left(  \left[  f\right]  _{BMO_{loc}\left(  B\left(
\overline{x},R;K^{\prime}R\right)  \right)  }+\frac{\Vert f\Vert
_{L^{p}(B(\overline{x},R))}}{|B(\overline{x},R)|^{1/p}}\right)  \cdot\\
&  \cdot\sum_{m=1}^{\infty}c(n,G,\Lambda)m^{N+\frac{N}{2}-2-2n}\\
&  =c\left(  p,G,\Lambda\right)  \left(  1+\left[  A\right]  _{LMO_{loc}%
\left(  \Omega\right)  }\right)  \left(  \left[  f\right]  _{BMO_{loc}\left(
B\left(  \overline{x},R;K^{\prime}R\right)  \right)  }+\frac{\Vert
f\Vert_{L^{p}(B(\overline{x},R))}}{|B(\overline{x},R)|^{1/p}}\right)
\end{align*}
where we have finally fixed $n$ large enough to make the series converge. So
the theorem is proved.
\end{proof}

\bigskip

\begin{proof}
[Proof of Theorem \ref{Thm commutator variable kernel}]The proof is similar to
the previous one. We have:%
\begin{align*}
&  [C\left[  K_{ij},b\right]  f]_{BMO_{loc}\left(  B\left(  \overline
{x},R;3R\right)  \right)  }\leq\sum_{m=1}^{\infty}\sum_{k=1}^{g_{m}}\left[
c_{ij}^{km}\,C\left[  T_{km},b\right]  f\right]  _{BMO_{loc}\left(  B\left(
\overline{x},R;3R\right)  \right)  }\\
&  \leq C\sum_{m=1}^{\infty}\sum_{k=1}^{g_{m}}\left(  \left\Vert c_{ij}%
^{km}\right\Vert _{L^{\infty}\left(  B\left(  \overline{x},3R\right)  \right)
}+\left[  c_{ij}^{km}\right]  _{LMO_{loc}\left(  B\left(  \overline
{x},R;3R\right)  \right)  }\right)  \cdot\\
&  \cdot\left(  \left[  C\left[  T_{km},b\right]  f\right]  _{BMO_{loc}\left(
B\left(  \overline{x},R;KR\right)  \right)  }+\frac{\Vert C\left[
T_{km},b\right]  f\Vert_{L^{p}(B(\overline{x},KR))}}{|B(\overline{x}%
,R)|^{1/p}}\right)  \,.
\end{align*}
Next, by Theorem \ref{Thm commutator convolution} and Remark
\ref{remark modified Tb},%
\begin{align*}
&  \left[  C\left[  T_{km},b\right]  f\right]  _{BMO_{loc}\left(  B\left(
\overline{x},R;KR\right)  \right)  }\leq c(p,G)\cdot m^{\frac{N}{2}%
}[b]_{LMO_{loc}B\left(  \overline{x},HR,K^{\prime}R\right)  }\cdot\\
\cdot &  \left(  [f]_{BMO_{loc}\left(  B\left(  \overline{x},R;KR\right)
\right)  }+\frac{1}{\left\vert B\left(  \overline{x},R\right)  \right\vert
^{1/p}}\Vert f\Vert_{L^{p}\left(  B\left(  \overline{x},R\right)  \right)
}\right)
\end{align*}
while, by the commutator theorem on $L^{p}$ (see (\ref{commutator Lp})),%
\[
\Vert C\left[  T_{km},b\right]  f\Vert_{L^{p}(B(\overline{x},KR))}\leq
c(p,G)\cdot m^{\frac{N}{2}}[b]_{LMO_{loc}B\left(  \overline{x},HR,K^{\prime
}R\right)  }\Vert f\Vert_{L^{p}\left(  B\left(  \overline{x},R\right)
\right)  }%
\]

Reasoning like in the previous proof we conclude%
\begin{align*}
&  [C\left[  K_{ij},b\right]  f]_{BMO_{loc}\left(  B\left(  \overline
{x},R;3R\right)  \right)  }\leq c(p,G,\Lambda)\left(  1+\left[  A\right]
_{LMO_{loc}\left(  \Omega\right)  }\right)  [b]_{LMO_{loc}\left(  B\left(
\overline{x},HR;K^{\prime}R\right)  \right)  }\cdot\\
&  \cdot\left(  \lbrack f]_{BMO_{loc}\left(  B\left(  \overline{x}%
,R;KR\right)  \right)  }+\frac{1}{\left\vert B\left(  \overline{x},R\right)
\right\vert ^{1/p}}\Vert f\Vert_{L^{p}\left(  B\left(  \overline{x},R\right)
\right)  }\right)  ,
\end{align*}
which by Lemma \ref{Lem KR} gives the result.
\end{proof}

\section{Local $BMO^{p}$ estimates for second order derivatives}

In this section we will prove our main result. As already explained in the
introduction, the proof will proceed in three steps corresponding to the three subsections.

\subsection{Local estimates for functions with small compact support}

\begin{theorem}
\label{Thm local compact supp}For every $p\in\left(  1,\infty\right)  $ there
exist positive constants $C,R_{0},$ depending on the group $G$, the numbers
$p$ and $\Lambda$, and the $VLMO_{loc}$ moduli of the coefficients $a_{hk}$,
and an absolute constant $K^{\prime}$, such that for every $\overline{x}%
\in\Omega$ with $R\leq R_{0}$ $B\left(  \overline{x},K^{\prime}R\right)
\subset\Omega$ and every $u\in S_{loc}^{2,p,\ast}\left(  B\left(  \overline
{x},R;3R\right)  \right)  $ with sprt$u\subset B\left(  \overline{x},R\right)
$ we have, for $i,j=1,2,...,q,$%
\begin{align*}
\left\Vert X_{i}X_{j}\,u\right\Vert _{BMO_{loc}^{p}\left(  B\left(
\overline{x},R;3R\right)  \right)  }  &  \leq C\left(  1+\left[  A\right]
_{LMO_{loc}\left(  \Omega\right)  }\right)  \cdot\\
&  \cdot\left\{  \left\Vert Lu\right\Vert _{BMO_{loc}^{p}\left(  B\left(
\overline{x},R;3R\right)  \right)  }+\frac{\Vert Lu\Vert_{L^{p}(B(\overline
{x},R))}}{|B(\overline{x},R)|^{1/p}}\right\}  .
\end{align*}

\end{theorem}

\begin{proof}
Let us start from (\ref{rep formula astratta}) for $u\in S_{loc}^{2,p,\ast
}\left(  B\left(  \overline{x},R;3R\right)  \right)  $ and take $BMO_{loc}%
\left(  B\left(  \overline{x},R;3R\right)  \right)  $ seminorms of both sides:%
\begin{align*}
&  \left[  X_{i}X_{j}\,u\right]  _{BMO_{loc}\left(  B\left(  \overline
{x},R;3R\right)  \right)  }\leq\left[  K_{ij}\left(  Lu\right)  \right]
_{BMO_{loc}\left(  B\left(  \overline{x},R;3R\right)  \right)  }+\\
&  +\sum_{h,k=1}^{q}\left[  C\left[  K_{ij},a_{hk}\right]  \left(  X_{h}%
X_{k}\,u\right)  \right]  _{BMO_{loc}\left(  B\left(  \overline{x}%
,R;3R\right)  \right)  }+\left[  \alpha_{ij}\cdot Lu\right]  _{BMO_{loc}%
\left(  B\left(  \overline{x},R;3R\right)  \right)  }%
\end{align*}
by Theorems \ref{Thm Singular integral variable kernel},
\ref{Thm commutator variable kernel}, \ref{Thm multiplication}, Lemma
\ref{Lem KR}, and the bounds (\ref{unif bound alfa_ij}), (\ref{bound alfa ij})
on the $\alpha_{ij}$'s,%
\begin{align*}
&  \leq C\left(  1+\left[  A\right]  _{LMO_{loc}\left(  \Omega\right)
}\right)  \left\{  \left[  Lu\right]  _{BMO_{loc}\left(  B\left(  \overline
{x},R;3R\right)  \right)  }+\frac{\Vert Lu\Vert_{L^{p}(B(\overline{x},R))}%
}{|B(\overline{x},R)|^{1/p}}\right. \\
&  +\left.  \sum_{h,k=1}^{q}[a_{hk}]_{LMO_{loc}\left(  B\left(  \overline
{x},HR;KR\right)  \right)  }\left(  [X_{h}X_{k}\,u]_{BMO_{loc}\left(  B\left(
\overline{x},R;3R\right)  \right)  }\right.  \right. \\
&  \left.  \left.  +\frac{1}{\left\vert B\left(  \overline{x},R\right)
\right\vert ^{1/p}}\Vert X_{h}X_{k}\,u\Vert_{L^{p}\left(  B\left(
\overline{x},R\right)  \right)  }\right)  \right\}  .
\end{align*}
We now exploit the assumption\ $a_{hk}\in VLMO_{loc}\left(  \Omega\right)  $:
for any $\varepsilon>0,$ by (\ref{eta f}) there exists $R_{0}$ such that for
any $R\leq R_{0}$%
\[
\sum_{h,k=1}^{q}\left[  a_{hk}\right]  _{LMO_{loc}\left(  B\left(
\overline{x},HR;KR\right)  \right)  }\leq\varepsilon.
\]
Choosing $\varepsilon$ such that
\[
C\left(  1+\left[  A\right]  _{LMO_{loc}\left(  \Omega\right)  }\right)
\varepsilon\leq\frac{1}{2}%
\]
we get, for $R\leq R_{0},$ $R_{0}$ depending on the $VLMO_{loc}$ moduli of the
$a_{hk}$'s,%
\begin{align*}
&  \sum_{i,j=1}^{q}\left[  X_{i}X_{j}\,u\right]  _{BMO_{loc}\left(  B\left(
\overline{x},R;3R\right)  \right)  }\leq C\left(  1+\left[  A\right]
_{LMO_{loc}\left(  \Omega\right)  }\right)  \cdot\\
&  \cdot\left\{  \left[  Lu\right]  _{BMO_{loc}\left(  B\left(  \overline
{x},R;3R\right)  \right)  }+\frac{1}{|B(\overline{x},R)|^{1/p}}\left[  \Vert
Lu\Vert_{L^{p}(B(\overline{x},R))}+\sum_{i,j=1}^{q}\Vert X_{i}X_{j}%
\,u\Vert_{L^{p}\left(  B\left(  \overline{x},R\right)  \right)  }\right]
\right\}  .
\end{align*}
Also exploiting the known $L^{p}$ estimates on $X_{h}X_{k}\,u$ (see
\cite{bb1}) we conclude:%
\begin{align*}
\sum_{i,j=1}^{q}\left\Vert X_{i}X_{j}\,u\right\Vert _{BMO_{loc}^{p}\left(
B\left(  \overline{x},R;3R\right)  \right)  }  &  \leq C\left(  1+\left[
A\right]  _{LMO_{loc}\left(  \Omega\right)  }\right)  \cdot\\
&  \cdot\left\{  \left\Vert Lu\right\Vert _{BMO_{loc}^{p}\left(  B\left(
\overline{x},R;3R\right)  \right)  }+\frac{\Vert Lu\Vert_{L^{p}(B(\overline
{x},R))}}{|B(\overline{x},R)|^{1/p}}\right\}  .
\end{align*}

\end{proof}

\subsection{Local estimates for functions with noncompact support}

To remove from our basic estimate in Theorem \ref{Thm local compact supp} the
assumption of compact support of the function $u$ we need to use, as usual,
cutoff functions and interpolation inequalities to handle the norms of first
order derivatives of $u$. However, it turns out that the local norms of type
$BMO_{loc}^{p}\left(  B\left(  \overline{x},R;3R\right)  \right)  $ that have
been useful to prove Theorem \ref{Thm local compact supp} are not adequate to
establish interpolation inequalities. Instead, we have to use to this aim the
more standard $BMO^{p}\left(  B\left(  \overline{x},R\right)  \right)  $
norms, which also allows us to apply directly some results proved in
\cite{BB3}. We have the following:

\begin{theorem}
\label{Thm noncomact}For every $p\in\left(  1,\infty\right)  $ there exist
positive constant $C,R_{0},$ depending on the group $G$, the numbers $p$ and
$\Lambda$, and the $VLMO_{loc}$ moduli of the coefficients $a_{hk}$, such that
for every $\overline{x}\in\Omega$ with $B\left(  \overline{x},R_{0}\right)
\subset\Omega$, every $t,R^{\prime}$ with $R_{0}/2\leq t<R^{\prime}\leq
R_{0},$ every $u\in S^{2,p,\ast}\left(  B\left(  \overline{x},R^{\prime
}\right)  \right)  $%
\[
\left\Vert u\right\Vert _{S^{2,p,\ast}\left(  B\left(  \overline{x},t\right)
\right)  }\leq C\left(  \frac{1}{\left(  R^{\prime}-t\right)  ^{\gamma
^{\prime}}}+1\right)  \left(  \left\Vert u\right\Vert _{BMO^{p}\left(
B\left(  \overline{x},R^{\prime}\right)  \right)  }+\left\Vert Lu\right\Vert
_{BMO^{p}\left(  B\left(  \overline{x},R^{\prime}\right)  \right)  }\right)
.
\]

\end{theorem}

\begin{proof}
Fix four numbers $R^{\prime}>R>s>t>0\ $and a cutoff function $\phi\in
C_{0}^{\infty}\left(  B\left(  \overline{x},R^{\prime}\right)  \right)  $ such
that
\[
B\left(  \overline{x},t\right)  \prec\phi\prec B\left(  \overline{x},s\right)
.
\]
Then for any function $u\in S^{2,p,\ast}\left(  B\left(  \overline
{x},R^{\prime}\right)  \right)  $ we have, by \cite[Lemma 4.4, (ii)]{BB3}:%
\[
\left\Vert X_{i}X_{j}u\right\Vert _{BMO^{p}\left(  B\left(  \overline
{x},t\right)  \right)  }=\left\Vert X_{i}X_{j}\left(  u\phi\right)
\right\Vert _{BMO^{p}\left(  B\left(  \overline{x},t\right)  \right)  }\leq
c\left\Vert X_{i}X_{j}\left(  u\phi\right)  \right\Vert _{BMO^{p}\left(
B\left(  \overline{x},s\right)  \right)  }%
\]
by Proposition \ref{Prop norme locali globali} (a) and applying Theorem
\ref{Thm local compact supp} to $\phi u$ on $B\left(  \overline{x}%
,s;3s\right)  $:%
\[
\leq c\left\Vert X_{i}X_{j}\left(  u\phi\right)  \right\Vert _{BMO_{loc}%
^{p}\left(  B\left(  \overline{x},s;3s\right)  \right)  }\leq C_{A}\left\{
\left\Vert L\left(  u\phi\right)  \right\Vert _{BMO_{loc}^{p}\left(  B\left(
\overline{x},s;3s\right)  \right)  }+\frac{\Vert L\left(  u\phi\right)
\Vert_{L^{p}(B(\overline{x},s))}}{s^{Q/p}}\right\}
\]
by \eqref{stimaseminorma}%
\[
\leq C_{A}\left\{  \left[  L\left(  u\phi\right)  \right]  _{BMO\left(
B\left(  \overline{x},R\right)  \right)  }+\left(  \frac{1}{\left(
R-s\right)  ^{Q/p}}+\frac{1}{s^{Q/p}}+1\right)  \left\Vert L\left(
u\phi\right)  \right\Vert _{L^{p}\left(  B\left(  \overline{x},s\right)
\right)  }\right\}  .
\]
Now, by our choice of $\phi$ and \cite[Lemma 4.12]{BB3}%
\begin{align*}
&  \left[  L\left(  u\phi\right)  \right]  _{BMO\left(  B\left(  \overline
{x},R\right)  \right)  }\\
&  \leq\frac{c}{s-t}\left[  Lu\right]  _{BMO\left(  B\left(  \overline
{x},R\right)  \right)  }+\frac{c}{\left(  s-t\right)  ^{2}}\left[
X_{j}u\right]  _{BMO\left(  B\left(  \overline{x},R\right)  \right)  }%
+\frac{c}{\left(  s-t\right)  ^{3}}\left[  u\right]  _{BMO\left(  B\left(
\overline{x},R\right)  \right)  }%
\end{align*}
Next, for $\frac{R_{0}}{2}\leq t<R,$ we pick $s=\left(  t+R\right)  /2,$ hence%
\begin{align*}
&  \left\Vert X_{i}X_{j}u\right\Vert _{BMO^{p}\left(  B\left(  \overline
{x},t\right)  \right)  }\leq C_{A}\left\{  \frac{c}{R-t}\left[  Lu\right]
_{BMO\left(  B\left(  \overline{x},R\right)  \right)  }+\frac{c}{\left(
R-t\right)  ^{2}}\left[  X_{j}u\right]  _{BMO\left(  B\left(  \overline
{x},R\right)  \right)  }\right. \\
&  +\left.  \frac{c}{\left(  R-t\right)  ^{3}}\left[  u\right]  _{BMO\left(
B\left(  \overline{x},R\right)  \right)  }+\left(  \frac{1}{\left(
R-t\right)  ^{Q/p}}+\frac{1}{R_{0}^{Q/p}}+1\right)  \left\Vert L\left(
u\phi\right)  \right\Vert _{L^{p}\left(  B\left(  \overline{x},R\right)
\right)  }\right\}
\end{align*}
so that, adding to both sides $\left\Vert X_{j}u\right\Vert _{BMO^{p}\left(
B\left(  \overline{x},t\right)  \right)  }+\left\Vert u\right\Vert
_{BMO^{p}\left(  B\left(  \overline{x},t\right)  \right)  }$ and exploiting
the estimates on $\left\Vert u\right\Vert _{S^{2,p}\left(  B\left(
\overline{x},R\right)  \right)  }$ which are known by \cite{bb1} we can write%
\begin{align*}
&  \left\Vert u\right\Vert _{S^{2,p,\ast}\left(  B\left(  \overline
{x},t\right)  \right)  }\leq C\left\{  \frac{1}{R-t}\left[  Lu\right]
_{BMO\left(  B\left(  \overline{x},R\right)  \right)  }+\frac{1}{\left(
R-t\right)  ^{2}}\left\Vert X_{j}u\right\Vert _{BMO^{p}\left(  B\left(
\overline{x},R\right)  \right)  }\right. \\
&  \left.  +\frac{1}{\left(  R-t\right)  ^{3}}\left\Vert u\right\Vert
_{BMO^{p}\left(  B\left(  \overline{x},R\right)  \right)  }+\left(  \frac
{1}{\left(  R-t\right)  ^{2+Q/p}}+1\right)  \left[  \left\Vert Lu\right\Vert
_{L^{p}\left(  B\left(  \overline{x},R\right)  \right)  }+\left\Vert
u\right\Vert _{L^{p}\left(  B\left(  \overline{x},R\right)  \right)  }\right]
\right\}  .
\end{align*}

We can then apply the interpolation inequality for $BMO$ seminorms proved in
\cite[Thm 4.15]{BB3} and the analogous interpolation inequality for $L^{p}$
norms proved in \cite[Thm. 21]{bb1}: choosing $R^{\prime}>R$ such that
$R^{\prime}-R=R-t,$ for some $\alpha>0$ and any $\delta>0$ we have%
\[
\left\Vert X_{j}u\right\Vert _{BMO^{p}\left(  B\left(  \overline{x},R\right)
\right)  }\leq\delta\left\Vert X_{i}X_{j}u\right\Vert _{BMO^{p}\left(
B\left(  \overline{x},R^{\prime}\right)  \right)  }+\frac{c}{\delta^{\alpha
}\left(  R^{\prime}-R\right)  ^{2\alpha}}\left\Vert u\right\Vert
_{BMO^{p}\left(  B\left(  \overline{x},R^{\prime}\right)  \right)  }.
\]
Choosing $\delta=\left(  R-t\right)  ^{2}\varepsilon$, with $\varepsilon$ to
be chosen later, we get%
\begin{align*}
&  \left\Vert u\right\Vert _{S^{2,p,\ast}\left(  B\left(  \overline
{x},t\right)  \right)  }\leq C\left\{  \frac{1}{R^{\prime}-t}\left[
Lu\right]  _{BMO\left(  B\left(  \overline{x},R^{\prime}\right)  \right)
}+\varepsilon\left\Vert X_{i}X_{j}u\right\Vert _{BMO^{p}\left(  B\left(
\overline{x},R^{\prime}\right)  \right)  }\right. \\
&  \left.  +\frac{c}{\varepsilon^{\alpha}\left(  R^{\prime}-t\right)
^{2+4\alpha}}\left\Vert u\right\Vert _{BMO^{p}\left(  B\left(  \overline
{x},R^{\prime}\right)  \right)  }+\frac{1}{\left(  R-t\right)  ^{3}}\left\Vert
u\right\Vert _{BMO^{p}\left(  B\left(  \overline{x},R\right)  \right)
}\right. \\
&  \left.  +\left(  \frac{1}{\left(  R-t\right)  ^{2+Q/p}}+1\right)  \left[
\left\Vert Lu\right\Vert _{L^{p}\left(  B\left(  \overline{x},R\right)
\right)  }+\left\Vert u\right\Vert _{L^{p}\left(  B\left(  \overline
{x},R\right)  \right)  }\right]  \right\}  .
\end{align*}
For $\varepsilon$ small enough we then get%
\begin{align*}
\left\Vert u\right\Vert _{S^{2,p,\ast}\left(  B\left(  \overline{x},t\right)
\right)  }  &  \leq\frac{1}{3}\left\Vert u\right\Vert _{S^{2,p,\ast}\left(
B\left(  \overline{x},R^{\prime}\right)  \right)  }\\
&  +C\left(  \frac{1}{\left(  R^{\prime}-t\right)  ^{\gamma}}+1\right)
\left(  \left\Vert u\right\Vert _{BMO^{p}\left(  B\left(  \overline
{x},R^{\prime}\right)  \right)  }+\left\Vert Lu\right\Vert _{BMO^{p}\left(
B\left(  \overline{x},R^{\prime}\right)  \right)  }\right)
\end{align*}
for some $\gamma>0$ and any $R^{\prime}>t>R_{0}/2.$ Applying \cite[Lemma
4.14]{BB3} we finally get%
\[
\left\Vert u\right\Vert _{S^{2,p,\ast}\left(  B\left(  \overline{x},t\right)
\right)  }\leq C\left(  \frac{1}{\left(  R^{\prime}-t\right)  ^{\gamma}%
}+1\right)  \left(  \left\Vert u\right\Vert _{BMO^{p}\left(  B\left(
\overline{x},R^{\prime}\right)  \right)  }+\left\Vert Lu\right\Vert
_{BMO^{p}\left(  B\left(  \overline{x},R^{\prime}\right)  \right)  }\right)
.
\]

\end{proof}

\subsection{Interior estimates in a domain}

Theorem \ref{Thm noncomact}, by the same techniques in \cite[proof of Thm
4.8]{BB3} immediately gives the following:

\begin{theorem}
\label{Thm main variation}For any $\Omega^{\prime}\Subset\Omega,$ with
$\Omega,\Omega^{\prime}$ regular domains (see below), every $p\in\left(
1,\infty\right)  $ there exists a positive constant $C$ depending on $\Omega$,
$\Omega^{\prime}$, the group $G$, the numbers $p$ and $\Lambda$, the
$VLMO_{loc}$ moduli of the coefficients $a_{hk}$, such that for every $u\in
S^{2,p,\ast}\left(  \Omega\right)  $%
\[
\left\Vert u\right\Vert _{S^{2,p,\ast}\left(  \Omega^{\prime}\right)  }\leq
C\left(  \left\Vert u\right\Vert _{BMO^{p}\left(  \Omega\right)  }+\left\Vert
Lu\right\Vert _{BMO^{p}\left(  \Omega\right)  }\right)  .
\]

\end{theorem}

We recall that in \cite{BB3} a domain $\Omega$ is called \emph{regular} if it
satisfies the property%
\begin{equation}
\left\vert B\left(  x,r\right)  \cap\Omega\right\vert \geq c\left\vert
B\left(  x,r\right)  \right\vert \label{regular}%
\end{equation}
for any $x\in\Omega,0<r<$diam$\Omega$. It is proved in \cite[Lemma 4.2]{BB3}
that, in particular, a metric ball is regular.

We are also interested in deriving from Theorem \ref{Thm noncomact} a version
of the above estimates involving local BMO norms, which is our main result,
stated in \S \ \ref{sec main result}, and \emph{does not }require assumption
(\ref{regular}). Namely, we now come to the

\bigskip

\begin{proof}
[Proof of Theorem \ref{Thm main}]For fixed $\Omega_{1}\Subset\Omega_{2}%
\Subset\Omega$, let $R_{0}>R_{2}>0$ be two numbers such that for any
$\overline{x}\in\Omega_{1}$ $B\left(  \overline{x},R_{2}\right)  \subset
\Omega_{2}$ and $B\left(  \overline{x},R_{0}\right)  \subset\Omega.$

Pick $t<\min\left(  \frac{R_{0}}{6},\frac{R_{2}}{2}\right)  ,$ then by Theorem
\ref{Thm noncomact} and Proposition \ref{Prop norme locali globali} (a) we
have%
\begin{align}
\left\Vert u\right\Vert _{S^{2,p,\ast}\left(  B\left(  \overline{x},t\right)
\right)  }  &  \leq C\left(  \left\Vert u\right\Vert _{BMO^{p}\left(  B\left(
\overline{x},2t\right)  \right)  }+\left\Vert Lu\right\Vert _{BMO^{p}\left(
B\left(  \overline{x},2t\right)  \right)  }\right) \nonumber\\
&  \leq C\left(  \left\Vert u\right\Vert _{BMO_{loc}^{p}\left(  B\left(
\overline{x},2t,6t\right)  \right)  }+\left\Vert Lu\right\Vert _{BMO_{loc}%
^{p}\left(  B\left(  \overline{x},2t,6t\right)  \right)  }\right) \nonumber\\
&  \leq C\left(  \left\Vert u\right\Vert _{BMO_{loc}^{p}\left(  \Omega
_{2},\Omega\right)  }+\left\Vert Lu\right\Vert _{BMO_{loc}^{p}\left(
\Omega_{2},\Omega\right)  }\right)  . \label{final 1}%
\end{align}
Recall that the constant $C$ in the above estimate depends on the domains but
not on $\overline{x}$. Clearly, the norm $\left\Vert u\right\Vert
_{S^{2,p}\left(  \Omega_{2}\right)  }$ is bounded by a finite sum of $N$ terms
of the kind $\left\Vert u\right\Vert _{S^{2,p}\left(  B\left(  \overline
{x}_{i},t\right)  \right)  },$ and then is bounded by $N$ times the right hand
side of (\ref{final 1}). Next, to bound the terms%
\[
\left[  X_{i}X_{j}u\right]  _{BMO_{loc}\left(  \Omega_{1},\Omega_{2}\right)
}+\left[  X_{j}u\right]  _{BMO_{loc}\left(  \Omega_{1},\Omega_{2}\right)
}+\left[  u\right]  _{BMO_{loc}\left(  \Omega_{1},\Omega_{2}\right)  },
\]
take a ball $B\left(  \overline{x},r\right)  $ centered at some $\overline
{x}\in\Omega_{1}$ and contained in $\Omega_{2}$. If $r\leq t$ (with $t$ as in
(\ref{final 1})) then%
\begin{align*}
&  \frac{1}{\left\vert B_{r}\left(  \overline{x}\right)  \right\vert }%
\int_{B_{r}\left(  \overline{x}\right)  }\left\vert X_{i}X_{j}u\left(
y\right)  -X_{i}X_{j}u_{B_{r}\left(  \overline{x}\right)  }\right\vert
dy\leq\left\Vert u\right\Vert _{S^{2,p,\ast}\left(  B\left(  \overline
{x},t\right)  \right)  }\\
&  \leq C\left(  \left\Vert u\right\Vert _{BMO_{loc}^{p}\left(  \Omega
_{2},\Omega\right)  }+\left\Vert Lu\right\Vert _{BMO_{loc}^{p}\left(
\Omega_{2},\Omega\right)  }\right)  .
\end{align*}
If $r>t$ (but $B\left(  \overline{x},r\right)  \subset\Omega_{2}$) then,
exploiting the local $S^{2,p}$ estimates of \cite{bb1},%
\begin{align*}
&  \frac{1}{\left\vert B_{r}\left(  \overline{x}\right)  \right\vert }%
\int_{B_{r}\left(  \overline{x}\right)  }\left\vert X_{i}X_{j}u\left(
y\right)  -\left(  X_{i}X_{j}u\right)  _{B_{r}\left(  \overline{x}\right)
}\right\vert dy\leq2\frac{1}{\left\vert B_{r}\left(  \overline{x}\right)
\right\vert }\int_{B_{r}\left(  \overline{x}\right)  }\left\vert X_{i}%
X_{j}u\left(  y\right)  \right\vert dy\\
&  \leq2\left(  \frac{1}{\left\vert B_{r}\left(  \overline{x}\right)
\right\vert }\int_{B_{r}\left(  \overline{x}\right)  }\left\vert X_{i}%
X_{j}u\left(  y\right)  \right\vert ^{p}dy\right)  ^{1/p}\leq\frac{c}{t^{Q/p}%
}\left\Vert X_{i}X_{j}u\right\Vert _{L^{p}\left(  \Omega_{2}\right)  }\\
&  \leq\frac{c}{t^{Q/p}}\left(  \left\Vert u\right\Vert _{L^{p}\left(
\Omega\right)  }+\left\Vert Lu\right\Vert _{L^{p}\left(  \Omega\right)
}\right)  \leq\frac{c}{t^{Q/p}}\left(  \left\Vert u\right\Vert _{BMO_{loc}%
^{p}\left(  \Omega_{2},\Omega\right)  }+\left\Vert Lu\right\Vert
_{BMO_{loc}^{p}\left(  \Omega_{2},\Omega\right)  }\right)  .
\end{align*}
In any case
\[
\left[  X_{i}X_{j}u\right]  _{BMO_{loc}\left(  \Omega_{1},\Omega_{2}\right)
}\leq C\left(  \left\Vert u\right\Vert _{BMO_{loc}^{p}\left(  \Omega
_{2},\Omega\right)  }+\left\Vert Lu\right\Vert _{BMO_{loc}^{p}\left(
\Omega_{2},\Omega\right)  }\right)
\]
Analogously we can bound $\left[  X_{j}u\right]  _{BMO_{loc}\left(  \Omega
_{1},\Omega_{2}\right)  }+\left[  u\right]  _{BMO_{loc}\left(  \Omega
_{1},\Omega_{2}\right)  }$. Exploiting again the local $S^{2,p}$ estimates we
conclude%
\[
\left\Vert u\right\Vert _{S_{loc}^{2,p,\ast}\left(  \Omega_{1},\Omega
_{2}\right)  }\leq C\left(  \left\Vert u\right\Vert _{BMO_{loc}^{p}\left(
\Omega_{2},\Omega\right)  }+\left\Vert Lu\right\Vert _{BMO_{loc}^{p}\left(
\Omega_{2},\Omega\right)  }\right)
\]
and we are done.
\end{proof}

\section{Appendix. Uniform bounds on the oscillation of the fundamental
solution\label{sec uniform bound}}

The aim of this section is to prove Theorem \ref{Thm unif BF}, that we have
exploited in \S \ 3 to prove $LMO$ uniform bounds on the coefficients
$c_{ij}^{km}$ which appear in the spherical harmonics expansion of
$\Gamma_{ij}$. Let us recall its statement:

\textbf{Theorem} \textbf{3.4} \textit{For any nonnegative integer }%
$p,$\textit{ there exists a constant }$c_{\Lambda,p}$\textit{ such that for
any }$x_{1},x_{2},y\in R^{N}$\textit{ we have: }%
\begin{equation}
\left\vert X_{i_{1}}X_{i_{2}}...X_{i_{p}}\Gamma\left(  x_{1},y\right)
-X_{i_{1}}X_{i_{2}}...X_{i_{p}}\Gamma\left(  x_{2},y\right)  \right\vert \leq
c_{\Lambda,p}\left\Vert A\left(  x_{1}\right)  -A\left(  x_{2}\right)
\right\Vert \left\Vert y\right\Vert ^{2-Q-p} \label{stime BLU rivisit}%
\end{equation}
\textit{where the differential operators }$X_{i_{j}}$\textit{ act on the }%
$y$\textit{-variable.}

Let us recall again that $\left\Vert y\right\Vert $ stands for the homogeneous
norm in $G$ while $\left\Vert A\left(  x_{1}\right)  -A\left(  x_{2}\right)
\right\Vert $ stands for the usual matrix norm in $\mathbb{R}^{2q}$. Actually,
the proof of this theorem amounts to revising some results and techniques
contained in \cite{BU}, \cite{BLU1} and \cite{BLU2}. Namely, the following
fact is proved in \cite[Corollary 7.13]{BLU2}:

\begin{theorem}
\label{Thm BLU} For any nonnegative integer $p,$ there exists a positive
constant $c_{\Lambda,p}$ such that%
\begin{align}
&  \left\vert X_{i_{1}}X_{i_{2}}...X_{i_{p}}\Gamma\left(  x_{1},y\right)
-X_{i_{1}}X_{i_{2}}...X_{i_{p}}\Gamma\left(  x_{2},y\right)  \right\vert
\label{stima BLU}\\
&  \leq c_{\Lambda,p}\left\Vert A\left(  x_{1}\right)  -A\left(  x_{2}\right)
\right\Vert ^{1/s}\left\Vert y\right\Vert ^{2-Q-p}\nonumber
\end{align}
for every $i_{1},...,i_{p}\in\left\{  1,2,...,q\right\}  ,$ and for every
$y\in\mathbb{R}^{N}\setminus\left\{  0\right\}  ,$ where the differential
operators $X_{i_{j}}$ act on the $y$-variable.
\end{theorem}

Recall that $s$ is the maximum lenght of commutators required to span
$\mathbb{R}^{N}$.

Comparing (\ref{stima BLU}) with (\ref{stime BLU rivisit}), one sees that what
we need is to replace the exponent $1/s$ of the matrix norm on the right hand
side of (\ref{stima BLU}) with the exponent $1$.

In order to prove Theorem \ref{Thm unif BF}, we will now revise the proof of
Theorem \ref{Thm BLU} given in \cite{BLU2}, which proceeds in four steps:

Step 1. The Authors assume $G$ to a be a free Carnot group. Under this
assumption, they consider the \textit{evolution operator}%
\[
H_{A}=\partial_{t}-\sum a_{ij}X_{i}X_{j}%
\]
where $A=\left\{  a_{ij}\right\}  $ is a symmetric positive, constant, matrix,
in a fixed \textquotedblleft ellipticity class\textquotedblright%
\ $\mathcal{M}_{\Lambda}$:%
\[
\Lambda|\xi|^{2}\leq\sum_{i,j=1}^{q}a_{ij}\xi_{i}\xi_{j}\leq\Lambda^{-1}%
|\xi|^{2}\text{ }\forall\xi\in{\mathbb{R}}^{q}.
\]
For this operator $H_{A}$ they consider the fundamental solution
(\textquotedblleft heat kernel\textquotedblright) $h_{A}\left(  x,t\right)  $
and prove the following estimate (see \cite[Theorem 7.5]{BLU1}):%
\begin{equation}
\left\vert h_{A_{1}}\left(  x,t\right)  -h_{A_{2}}\left(  x,t\right)
\right\vert \leq c_{\lambda}\left\Vert A_{1}-A_{2}\right\Vert ^{1/s}%
t^{-Q/2}\exp\left(  -\frac{\left\Vert x\right\Vert ^{2}}{c_{\Lambda}t}\right)
\label{Step 1 BLU}%
\end{equation}
for $x\in\mathbb{R}^{N},t>0,$ any $A_{1},A_{2}\in\mathcal{M}_{\Lambda}$.

Step 2. Under the same assumptions, the Authors extend the previous bound to
the derivatives of $h_{A},$ proving (see \cite[Theorem 7.7]{BLU1}): for any
nonnegative integers $p,m,$ there exist positive constants $c_{\lambda}$ and
$c_{\lambda,p,m}$ such that%
\begin{align}
&  \left\vert X_{i_{1}}...X_{i_{p}}\left(  \partial_{t}\right)  ^{m}h_{A_{1}%
}\left(  x,t\right)  -X_{i_{1}}...X_{i_{p}}\left(  \partial_{t}\right)
^{m}h_{A_{2}}\left(  x,t\right)  \right\vert \leq\label{BLU bound derivatives}%
\\
&  \leq c_{\Lambda,p,m}\left\Vert A_{1}-A_{2}\right\Vert ^{1/s}t^{-\left(
Q+p+2m\right)  /2}\exp\left(  -\frac{\left\Vert x\right\Vert ^{2}}{c_{\Lambda
}t}\right) \nonumber
\end{align}
for $x\in\mathbb{R}^{N},t>0,$ any $i_{1},i_{2},...,i_{p}\in\left\{
1,2,...,q\right\}  ,$ and any $A_{1},A_{2}\in\mathcal{M}_{\Lambda}.$

Step 3. The assumption that $G$ is free in now removed, by a suitable
\textquotedblleft lifting result\textquotedblright\ proved in \cite{BU} (see
also \cite[Theorem 8.3]{BLU1}), so that (\ref{BLU bound derivatives}) is
established for \textit{any} Carnot group.

Step 4. The Authors prove (see \cite[Theorem 3.9]{BLU1}) that the function%
\[
\Gamma_{A}\left(  x\right)  =\int_{0}^{+\infty}h_{A}\left(  x,t\right)  dt
\]
is the fundamental solution to the stationary operator%
\[
L_{A}=-\sum a_{ij}X_{i}X_{j}.
\]
Then, integrating in the $t$ variable the estimate
(\ref{BLU bound derivatives}) they get%
\[
\left\vert X_{i_{1}}...X_{i_{p}}\Gamma_{A_{1}}\left(  x,t\right)  -X_{i_{1}%
}...X_{i_{p}}\Gamma_{A_{2}}\left(  x,t\right)  \right\vert \leq c_{\Lambda
,p,q}\left\Vert A_{1}-A_{2}\right\Vert ^{1/s}\left\Vert y\right\Vert ^{2-Q-p}%
\]
which is essentially (\ref{stima BLU}).

We are going to show that the following refinement of (\ref{Step 1 BLU}) can
be proved:

\begin{proposition}
\label{Prop Step 1b}Under the same assumptions and with the same notation of
Step 1, we have%
\[
\left\vert h_{A_{1}}\left(  x,t\right)  -h_{A_{2}}\left(  x,t\right)
\right\vert \leq c_{\Lambda}\left\Vert A_{1}-A_{2}\right\Vert \cdot
t^{-Q/2}\exp\left(  -\frac{\left\Vert x\right\Vert ^{2}}{c_{\Lambda}t}\right)
\]
for any $x\in G,t>0.$
\end{proposition}

(The improvement consists in the exponent $1$ instead of $1/s,$ for
$\left\Vert A_{1}-A_{2}\right\Vert $). Before proving this proposition, we
collect in the following theorem a number of results proved in \cite{BU}, that
we will need.

\begin{theorem}
Let $G$ be a free homogeneous Carnot group and let $A\in\mathcal{M}_{\Lambda}%
$. Then there exists a Lie group automorphism $T_{A}$ of $G$, commuting with
the dilations of $G$, such that%
\begin{equation}
h_{A}\left(  x,t;\xi,\tau\right)  =J_{A}\left(  x\right)  \cdot h_{G}\left(
T_{A_{1}}\left(  x\right)  ,t;T_{A_{1}}\left(  \xi\right)  ,\tau\right)
\label{h_A}%
\end{equation}
where $h_{G}$ is the heat kernel for $\partial_{t}-\sum X_{i}^{2}$ on $G$ ,
which is actually a convolution kernel:%
\begin{equation}
h_{G}\left(  T_{A_{1}}\left(  x\right)  ,t;T_{A_{1}}\left(  \xi\right)
,\tau\right)  =h_{G}\left(  T_{A_{1}}\left(  \xi\right)  ^{-1}\circ T_{A_{1}%
}\left(  x\right)  ,t-\tau\right)  \label{h_G}%
\end{equation}
and has the following homogeneity%
\begin{equation}
h_{G}\left(  D\left(  \lambda\right)  x,\lambda^{2}t\right)  =t^{-Q}%
h_{G}\left(  x,t\right)  \text{ for any }\lambda>0, \label{h_G homog}%
\end{equation}
while $J_{A}\left(  x\right)  =\left\vert \det\mathcal{J}_{T_{A}}\left(
x\right)  \right\vert $, where $\mathcal{J}_{T_{A}}$ is the Jacobian of
$T_{A}$. Moreover $J_{A}\left(  x\right)  $ turns out to be constant in $x$,
and%
\begin{align}
\left(  c_{\Lambda}\right)  ^{-1}  &  \leq J_{A}\leq c_{\Lambda}%
\label{BU 19}\\
\left\vert J_{A_{1}}-J_{A_{2}}\right\vert  &  \leq c_{\Lambda}\left\Vert
A_{1}-A_{2}\right\Vert \label{BU 20}\\
\left(  c_{\Lambda}\right)  ^{-1}\left\Vert x\right\Vert  &  \leq\left\Vert
T_{A}\left(  x\right)  \right\Vert \leq c_{\Lambda}\left\Vert x\right\Vert
\label{BU 21}\\
\left\vert T_{A_{1}}\left(  x\right)  -T_{A_{2}}\left(  x\right)  \right\vert
&  \leq c_{\Lambda}\left\Vert A_{1}-A_{2}\right\Vert , \label{BU last}%
\end{align}
for any $x\in G,$ any $A,A_{1},A_{2}\in\mathcal{M}_{\Lambda}.$
\end{theorem}

Relations (\ref{h_A}), (\ref{BU 19}), (\ref{BU 20}), (\ref{BU 21}) are (1.5),
(2.19), (2.20), (2.21) in \cite{BU}, respectively (note that our symbol $h$
corresponds to $\Gamma$ in \cite{BU});\ for (\ref{BU last}), see the last line
in \cite{BU}; (\ref{h_G}) and (\ref{h_G homog}) are known properties of the
heat kernel on Carnot groups, see also \cite{BLU1}.

\bigskip

\begin{proof}
[Proof of Proposition \ref{Prop Step 1b}]Here we somewhat revise the proof of
\cite[Theorem 7.5]{BLU1}.%
\begin{align}
&  \left\vert h_{A_{1}}\left(  x,t\right)  -h_{A_{2}}\left(  x,t\right)
\right\vert =\left\vert J_{A_{1}}h_{G}\left(  T_{A_{1}}\left(  x\right)
,t\right)  -J_{A_{2}}h_{G}\left(  T_{A_{2}}\left(  x\right)  ,t\right)
\right\vert \leq\nonumber\\
&  \leq\left\vert J_{A_{1}}-J_{A_{2}}\right\vert \left\vert h_{G}\left(
T_{A_{1}}\left(  x\right)  ,t\right)  \right\vert +J_{A_{2}}\left\vert
h_{G}\left(  T_{A_{1}}\left(  x\right)  ,t\right)  -h_{G}\left(  T_{A_{2}%
}\left(  x\right)  ,t\right)  \right\vert \equiv I+II. \label{bound I+II}%
\end{align}

By (\ref{BU 20}), the Gaussian estimate for\ $h_{G}$ and (\ref{BU 21}):%
\begin{equation}
I\leq c_{\Lambda}\left\Vert A_{1}-A_{2}\right\Vert t^{-Q/2}\exp\left(
-\frac{\left\Vert T_{A_{1}}\left(  x\right)  \right\Vert }{ct}\right)  \leq
c_{\Lambda}\left\Vert A_{1}-A_{2}\right\Vert t^{-Q/2}\exp\left(
-\frac{\left\Vert x\right\Vert }{ct}\right)  , \label{bound I}%
\end{equation}
while by (\ref{BU 19})%
\begin{equation}
II\leq c_{\Lambda}\left\vert h_{G}\left(  T_{A_{1}}\left(  x\right)
,t\right)  -h_{G}\left(  T_{A_{2}}\left(  x\right)  ,t\right)  \right\vert .
\label{bound II}%
\end{equation}
We now need the following

\begin{claim}
There exists $k>0$ such that if $\left\Vert x\right\Vert =1,$ then%
\begin{equation}
\left\vert h_{G}\left(  T_{A_{1}}\left(  x\right)  ,t\right)  -h_{G}\left(
T_{A_{2}}\left(  x\right)  ,t\right)  \right\vert \leq c_{\Lambda}\left\Vert
A_{1}-A_{2}\right\Vert t^{-Q/2-k}\exp\left(  -\frac{1}{ct}\right)  .
\label{claim}%
\end{equation}

\end{claim}

Let us first show how we can conclude the proof using the claim, then we will
prove (\ref{claim}).

For any $x\in G,$ let $x=D_{\left\Vert x\right\Vert }\left(  x^{\prime
}\right)  $ with $\left\Vert x^{\prime}\right\Vert =1,$ and let us apply
(\ref{claim}), keeping in mind (\ref{h_G homog}) and the fact that $T_{A}$
commutes with dilations:%
\[
T_{A_{i}}\left(  x\right)  =T_{A_{i}}\left(  D_{\left\Vert x\right\Vert
}\left(  x^{\prime}\right)  \right)  =D_{\left\Vert x\right\Vert }\left(
T_{A_{i}}\left(  x^{\prime}\right)  \right)  \text{ for }i=1,2,\text{ hence}%
\]%
\begin{align*}
&  \left\vert h_{G}\left(  T_{A_{1}}\left(  x\right)  ,t\right)  -h_{G}\left(
T_{A_{2}}\left(  x\right)  ,t\right)  \right\vert =\\
&  =\left\vert h_{G}\left(  D_{\left\Vert x\right\Vert }\left(  T_{A_{1}%
}\left(  x^{\prime}\right)  \right)  ,\left\Vert x\right\Vert ^{2}\frac
{t}{\left\Vert x\right\Vert ^{2}}\right)  -h_{G}\left(  D_{\left\Vert
x\right\Vert }\left(  T_{A_{2}}\left(  x^{\prime}\right)  \right)  ,\left\Vert
x\right\Vert ^{2}\frac{t}{\left\Vert x\right\Vert ^{2}}\right)  \right\vert
=\\
&  =\left\Vert x\right\Vert ^{-Q}\left\vert h_{G}\left(  T_{A_{1}}\left(
x^{\prime}\right)  ,\frac{t}{\left\Vert x\right\Vert ^{2}}\right)
-h_{G}\left(  T_{A_{2}}\left(  x^{\prime}\right)  ,\frac{t}{\left\Vert
x\right\Vert ^{2}}\right)  \right\vert \leq\\
&  \leq\left\Vert x\right\Vert ^{-Q}c_{\Lambda}\left\Vert A_{1}-A_{2}%
\right\Vert \left(  \frac{t}{\left\Vert x\right\Vert ^{2}}\right)
^{-Q/2-k}\exp\left(  -\frac{\left\Vert x\right\Vert ^{2}}{ct}\right)  =\\
&  =c_{\Lambda}\left\Vert A_{1}-A_{2}\right\Vert t^{-Q/2}\left(
\frac{\left\Vert x\right\Vert ^{2}}{t}\right)  ^{k}\exp\left(  -\frac
{\left\Vert x\right\Vert ^{2}}{ct}\right)  \leq\\
&  \leq c_{\Lambda}\left\Vert A_{1}-A_{2}\right\Vert t^{-Q/2}\exp\left(
-\frac{\left\Vert x\right\Vert ^{2}}{ct}\right)  ,
\end{align*}
possibly changing the constant $c$ inside the $\exp$. By (\ref{bound I+II}),
(\ref{bound I}), (\ref{bound II}) this implies the result.

Let us now prove the Claim. We will bound the left hand side of (\ref{claim})
applying Lagrange theorem (in the standard form, instead of the Lagrange
theorem for vector fields which is applied in the proof of \cite[Theorem
7.5]{BLU1}).%
\begin{equation}
\left\vert h_{G}\left(  T_{A_{1}}\left(  x\right)  ,t\right)  -h_{G}\left(
T_{A_{2}}\left(  x\right)  ,t\right)  \right\vert \leq\left\vert T_{A_{1}%
}\left(  x\right)  -T_{A_{2}}\left(  x\right)  \right\vert \sup_{y\in\left[
T_{A_{1}}\left(  x\right)  ,T_{A_{2}}\left(  x\right)  \right]  }\left\vert
\nabla_{y}h_{G}\left(  y,t\right)  \right\vert . \label{Lagrange}%
\end{equation}
Recalling (\ref{BU 21}), we now note that for some constants $\delta_{0}%
,c_{0}\in\left(  0,1\right)  $ we can say that%
\[
\text{if }\left\Vert x\right\Vert =1,y\in\left[  T_{A_{1}}\left(  x\right)
,T_{A_{2}}\left(  x\right)  \right]  \text{ and }\left\vert T_{A_{1}}\left(
x\right)  -T_{A_{2}}\left(  x\right)  \right\vert \leq\delta_{0}\text{, then
}\left\Vert y\right\Vert \geq c_{0}\text{.}%
\]
We then distinguish two cases.

\textbf{Case 1.} $\left\vert T_{A_{1}}\left(  x\right)  -T_{A_{2}}\left(
x\right)  \right\vert \leq\delta_{0}.$ Then we proceed from (\ref{Lagrange}),
expressing Euclidean derivatives of $h_{G}$ in terms of the vector fields
$X_{i}$'s and their commutators, and exploiting Gaussian bounds for $h_{G}$
proved in \cite[Theorem 5.3]{BLU1}%
\begin{align}
&  \sup_{y\in\left[  T_{A_{1}}\left(  x\right)  ,T_{A_{2}}\left(  x\right)
\right]  }\left\vert \nabla_{y}h_{G}\left(  y,t\right)  \right\vert
\leq\label{grad bound}\\
&  \leq ct^{-Q/2-k}\cdot\sup_{y\in\left[  T_{A_{1}}\left(  x\right)
,T_{A_{2}}\left(  x\right)  \right]  }\exp\left(  -\frac{\left\Vert
y\right\Vert }{c_{\Lambda}t}\right)  \leq ct^{-Q/2-k}\exp\left(  -\frac
{1}{c_{\Lambda}t}\right)  .\nonumber
\end{align}
By (\ref{Lagrange}), (\ref{grad bound}) and (\ref{BU last}) we get
(\ref{claim}).

\textbf{Case 2.} $\left\vert T_{A_{1}}\left(  x\right)  -T_{A_{2}}\left(
x\right)  \right\vert >\delta_{0}.$ We then apply (\ref{Step 1 BLU}) (that is
the result already proved in \cite{BLU1}):%
\begin{align*}
&  \left\vert h_{G}\left(  T_{A_{1}}\left(  x\right)  ,t\right)  -h_{G}\left(
T_{A_{2}}\left(  x\right)  ,t\right)  \right\vert \leq c_{\Lambda}\left\Vert
A_{1}-A_{2}\right\Vert ^{1/s}t^{-Q/2}\exp\left(  -\frac{1}{ct}\right)  \leq\\
&  \leq\frac{c}{\delta_{0}}\left\vert T_{A_{1}}\left(  x\right)  -T_{A_{2}%
}\left(  x\right)  \right\vert \left\Vert A_{1}-A_{2}\right\Vert
^{1/s}t^{-Q/2}\exp\left(  -\frac{1}{ct}\right)  \leq
\end{align*}
by (\ref{BU last})%
\[
\leq c\left\Vert A_{1}-A_{2}\right\Vert ^{1+1/s}t^{-Q/2}\exp\left(  -\frac
{1}{ct}\right)  \leq c\left\Vert A_{1}-A_{2}\right\Vert t^{-Q/2}\exp\left(
-\frac{1}{ct}\right)
\]
since the matrix norm $\left\Vert A_{1}-A_{2}\right\Vert $ is always bounded
in $\mathcal{M}_{\Lambda}.$ Finally, the last expression can be bound by%
\[
c\left\Vert A_{1}-A_{2}\right\Vert t^{-Q/2-k}\exp\left(  -\frac{1}{ct}\right)
,
\]
possibly changing the exponent inside the $\exp$.

So the Claim is proved and the proof of the Proposition is complete.
\end{proof}

Starting from the previous Proposition one can now proceed following word by
word the arguments of Steps 2, 3 and 4 in \cite{BLU1}, concluding the proof of
Theorem \ref{Thm unif BF}.

\bigskip

Marco Bramanti

Dipartimento di Matematica

Politecnico di Milano

Via Bonardi 9

20133 Milano, ITALY

marco.bramanti@polimi.it

\bigskip

Maria Stella Fanciullo

Dipartimento di Matematica e Informatica

Universit\`{a} di Catania

Viale Andrea Doria 6

95125 Catania, ITALY

fanciullo@dmi.unict.it

\end{document}